\documentclass[11pt]{article}
\usepackage{fullpage}
\usepackage{amssymb,latexsym,amsmath,amsthm}     
\usepackage{mathtools}   
\usepackage{bbm}
\usepackage[utf8]{inputenc}
\usepackage{enumerate}
\usepackage{enumitem}
\usepackage{subfigure}
\usepackage{color}
\usepackage[dvipsnames]{xcolor}
\usepackage[toc,page]{appendix}
\usepackage{scalerel,stackengine}
\usepackage{titling}
\usepackage{mathrsfs}
\usepackage{colonequals}
\usepackage{float}

\usepackage[style=alphabetic,
            maxbibnames=99,  
            giveninits=true, 
            sortcites=true,  
            backend=biber]{biblatex}
\usepackage[disable]{todonotes}

\providecommand{\todoneukamm}[1]{\todo[inline,author={SN}]{#1}}
\providecommand{\todosander}[1]{\todo[inline,author={OS},color=green]{#1}}

\usepackage[leftcaption]{sidecap}
\usepackage{hyperref}
\hypersetup{
    colorlinks=true,
    linkcolor=orange,
    filecolor=blue,      
    urlcolor=blue,
    citecolor=blue,
    linktocpage=true
}

\sidecaptionvpos{figure}{t}
\newtheorem{theorem}{Theorem}[section] 
 
\newtheorem*{proposition*}{Proposition}
\newtheorem{definition}[theorem]{Definition}
\newtheorem*{definition*}{Definition}
\newtheorem{lemma}[theorem]{Lemma}
\newtheorem*{lemma*}{Lemma}

\newtheorem{remark}[theorem]{Remark}
\newtheorem{example}[theorem]{Example}

\newcommand{\expect}[1]{\left\langle {#1} \right\rangle}

\newcommand{\e}{_{\varepsilon}}
\newcommand{\eps}{\varepsilon}
\renewcommand{\hom}{{\mathrm{hom}}}

\newcommand{\p}{_{\mathrm{pot}}}
\newcommand{\brac}[1]{\left({#1}\right) }
\newcommand{\cb}[1]{\left\lbrace {#1} \right\rbrace}
\newcommand{\R}{\mathbb{R}}
\newcommand{\Z}{\mathbb{Z}}
\newcommand{\N}{\mathbb{N}}

\newcommand{\E}{\mathcal{E}}

\newcommand\Id{\operatorname{Id}}

\DeclareMathOperator*{\argmin}{arg\,min}
\newcommand{\rcal}{\mathcal{R}}

\newcommand{\wto}{\rightharpoonup}

\DeclarePairedDelimiter{\abs}{\lvert}{\rvert}
\DeclarePairedDelimiter{\norm}{\lVert}{\rVert}

\makeatletter
\newcommand{\setword}[2]{%
  \phantomsection
  #1\def\@currentlabel{\unexpanded{#1}}\label{#2}%
}
\makeatother

\def\Xint#1{\mathchoice
{\XXint\displaystyle\textstyle{#1}}%
{\XXint\textstyle\scriptstyle{#1}}%
{\XXint\scriptstyle\scriptscriptstyle{#1}}%
{\XXint\scriptscriptstyle\scriptscriptstyle{#1}}%
\!\int}
\def\XXint#1#2#3{{\setbox0=\hbox{$#1{#2#3}{\int}$ }
\vcenter{\hbox{$#2#3$ }}\kern-.6\wd0}}
\def\dashint{\Xint-}    

\def\wtto{\stackrel{2}{\rightharpoonup}}
\def\stto{\stackrel{2}{\rightarrow}}
\def\ctto{\stackrel{2c}{\longrightarrow}}

\graphicspath{{gfx/}}

\usepackage{authblk}
\title{Representative volume element approximations in elastoplastic spring networks}

\author{Sabine Haberland\thanks{\href{mailto:sabine.haberland@tu-dresden.de}{sabine.haberland@tu-dresden.de}}}
\author{Patrick Jaap\thanks{\href{mailto:patrick.jaap@tu-dresden.de}{patrick.jaap@tu-dresden.de}}}
\author{Stefan Neukamm\thanks{\href{mailto:stefan.neukamm@tu-dresden.de}{stefan.neukamm@tu-dresden.de}}}
\author{Oliver Sander\thanks{\href{mailto:oliver.sander@tu-dresden.de}{oliver.sander@tu-dresden.de}}}
\author{Mario Varga\thanks{\href{mailto:mario.varga@tu-dresden.de}{mario.varga@tu-dresden.de}}}
\affil{Fakult\"at Mathematik, Technische Universit\"at Dresden}
\begin{document}
\maketitle

\begin{abstract}
  We study the large-scale behavior of a small-strain lattice model for a network composed of elastoplastic springs with random material properties.
  We formulate the model as an evolutionary rate independent system (ERIS).
  In an earlier work we derived a homogenized continuum model, which has the form of linearized elastoplasticity, as evolutionary $\Gamma$-limit as the lattice parameter tends to zero.
  In the present paper we introduce a periodic representative volume element approximation (RVE) for the  homogenized system.
  As a main result we prove convergence of the RVE approximation as the size of the RVE tends to infinity.
  We also show that the hysteretic stress--strain relation of the effective system can be described with help of a generalized Prandt--Ishlinskii operator,
  and we prove convergence of a periodic RVE approximation for that operator.
  We combine the RVE approximation with a numerical scheme for rate-independent systems and obtain a computational scheme that we use to numerically investigate the homogenized system in the specific case when the original network is given by a two-dimensional lattice model.
  We simulate the response of the system to cyclic and uniaxial, monotonic loading, and numerically investigate the convergence rate of the periodic RVE approximation.
  In particular, our simulations show that the RVE error decays with the same rate as the RVE error in the static case of linear elasticity.
  \smallskip


\end{abstract}

\tableofcontents
\section{Introduction}
Spring network models play an important role in mechanics for two reasons:
On one hand, such models serve as a numerical finite-difference approximation of continuum models.
On the other hand, they are used as reduced models for materials with an inherent microstructural discreteness, e.g., truss-structures or polymers, e.g., see~\cite{hahn2010discrete,ostoja2002lattice,jagota1994spring,alicandro2011integral}.
Often such materials are heterogeneous on the microscale.
This complicates the numerical treatment of such systems.
In situations when the microstructure of the material features some spatial regularity, in particular, when it is periodic or random and statistically homogeneous,
the homogenization approach can be used to derive a homogenized model that features the same large-scale behavior as the discrete, microstructured  model.

In this paper, we consider a small-strain model for a network composed of elastoplastic springs with random material properties with a stationary and ergodic distribution.
More precisely, we consider a spring network that in its stress-free reference configuration is represented by a lattice graph with node set $\eps\Z^d$ and edge set $\eps\mathsf E$ where $0<\eps\ll1$ denotes a small a scaling parameter.
In our model each edge $\mathsf e\in\eps\mathsf E$ represents an elastoplastic spring that we describe by a zero-dimensional model of linearized plasticity model with (Prager's type) kinematic hardening.
Each spring comes with three material parameters (the elastic modulus $a$, the hardening parameter $h$, and the yield strength $\sigma_{\rm yield}$) that are chosen randomly with a stationary and ergodic distribution.
We consider a model that describes the evolution of the network contained in a macroscopic domain $Q\subset\R^d$.
The system is described by a pair $y_\eps=(u_\eps,p_\eps)$ of two kinematic variables: The displacement $u_\eps: \varepsilon\Z^d\cap Q \to \R^d$ and $p_\eps: \eps\mathsf E\cap Q \to \R$, which describe the plastic strain longitudinal to the springs.
As outlined in Section~\ref{sec:setting}, we describe the evolution of the entire network by an evolutionary rate independent system (ERIS) of the form
\begin{equation*}
0  \in \partial \mathcal{R}_{\varepsilon}(\dot{y}\e(t))+ D_y\E_{\varepsilon}(t, y\e(t)),
\end{equation*}
where $\mathcal{R}_{\varepsilon}$ denotes a positively $1$-homogeneous dissipation functional, $\mathcal{E}\e$ is a quadratic energy functional.  $\partial$ and $D_y$ denote the subdifferential and Gateaux derivative of $\mathcal R_\eps$ and $\mathcal E_\eps(t,\cdot)$, respectively.
We refer to the Appendix~\ref{sec:eris} for the basic notion of ERIS and the existence theory of energetic solutions for ERIS.

In an earlier work, \textcite{neukamm2017stochastic} derived as an evolutionary $\Gamma$-limit for $\eps\to 0$ a homogenized, continuum model that comes in the form of an ERIS with energy functional $\mathcal E_{\hom}$ and dissipation functional $\rcal_\hom$, see Theorem~\ref{T:hom}.
The state variable $y_{\hom}$ of the homogenized ERIS contains three kinematic variables: The macroscopic displacement $u:Q\to\R^d$ and two internal variables, $p$ and $\chi_s$ that describe the internal state of the system. While $p$ describes the configuration of the plastic deformation on the microscopic scale, $\chi_s$ can be viewed as a quantity that describes microscopic fluctuations of the elastic strain of the system, see Section~\ref{S:two-scale_system} for further explanations.
The two internal variables $(p,\chi_s)$ are defined on the extended domain $Q\times\Omega$, where $\Omega$ denotes a probability space that is used to model the random configurations of the discrete system.
As we shall explain in Section~\ref{S:random-network}, each $\omega\in\Omega$ represents a realization of the random heterogeneous material.

As a first new result of this paper, we prove in Theorem~\ref{T:hyst_hom} that the homogenized ERIS can be brought into a form that is more standard in the engineering literature:
\begin{equation*}
  -\nabla\cdot\sigma_{\hom}(t,x)=l(t,x)\qquad\text{ in a distributional sense in $Q$},
\end{equation*}
with stress--strain relation given by
\begin{equation*}
  \sigma_{\hom}(t,x)=\mathcal W_{\hom}[\operatorname{sym}\nabla u(\cdot,x)](t)\qquad\text{for all $x\in Q$},
\end{equation*}
where $\mathcal W_{\hom}$ denotes a generalized Prandtl--Ishlisinskii hysteresis operator.
In particular, this means that the value of $\sigma_{\hom}(t,x)$ for a fixed $t>0$ and material point $x\in\R^d$ depends on the entire history of the strain $[0,t)\ni s\mapsto \operatorname{sym}\nabla u(s,x)$ at the material point $x\in\R^d$.
The theorem reveals that the hysteresis operator $\mathcal W_{\hom}$ only depends on the lattice geometry and the distribution of the random material parameters and can be defined by an ERIS with an abstract state space based on the probability space.
The hysteresis represents the homogenized material law of the system.

Since the homogenized system (and likewise the definition of $\mathcal W_{\hom}$) invokes the probability space $\Omega$,  numerical simulations of the homogenized system and of $\mathcal W_{\hom}$ require to resolve the dependence of variables on $\omega\in\Omega$.
A common approach in computational homogenization to tackle this difficulty is based on the method of the representative volume element approximation (RVE), see, e.g., \cite{feyel1999multiscale,miehe2002strain,schroder2013plasticity,nejad2019parallel}.
In this procedure a macroscopic discretization of $Q$ is considered and to each node of the discretization
another (discretized) representative volume domain is assigned. This domain is then used
for the computation of averaged stresses (see also \cite{eigel2015convergent} for a related stochastic approach).

In our paper we consider a periodic RVE:
To each macroscopic point $x\in Q$ we associate a discrete, representative volume
element with nodes $\Lambda_L\colonequals \Z^d\cap [0,L)^d$ and edges $\mathsf E\cap\Lambda_L$. Here, $L$ denotes a large integer that denotes the size of the RVE.
We see this set as a microscopic ``averaging domain''.
We approximate the pair of internal variables by fields on this representative volume that satisfy periodic boundary conditions.
The periodic RVE approximation comes itself in the form of an ERIS, see Section~\ref{S:RVE_def}.

As the main result of our paper we prove  convergence of the periodic RVE as $L\to\infty$:
In Theorems~\ref{thm:1462:d}~and~\ref{T:RVE-W} we do this for the homogenized ERIS and the hysteresis operator $\mathcal W_{\hom}$, respectively.
While convergence results for the approximation of (linear and monotone) elliptic equations and systems via RVEs are well-known, e.g., see \cite{owhadi2003approximation, bourgeat2004approximations, fischer2019optimal, gloria2015quantification}, to our knowledge, our theorems are the first analytical results that prove the convergence of an RVE-approximation for a rate-independent system in the case of random heterogeneous coefficients.
Our proof relies on the one hand on the general theory of ERIS and the concept of evolutionary $\Gamma$-convergence, see \cite{mielke2005evolution, mielke2016evolutionary, mielke2012generalized}, and on the other hand, on a transformation that is close to the stochastic unfolding procedure, see \cite{neukamm2017stochastic, varga2019stochastic, heida20191, heida20192, neukamm2020}.

In Sections~\ref{sec:numerical_approach}~and~\ref{S:exp} we explore the homogenized system with help of numerical simulations.
We focus on the hysteresis operator $\mathcal W_{\hom}$.
While $\mathcal W_{\hom}$ cannot be directly computed, we consider its periodic RVE approximation $\mathcal W_L$ introduced in Lemma~\ref{L:hyst_hom}.
For a sample $\omega\in\Omega$ (which represents a typical realization of the random material parameters), the hysteresis operator $\mathcal W_L^\omega$ can be evaluated
by numerically solving an ERIS with state variables $(p_L,\varphi_L)$ that  are on the finite  graph $\Lambda_L\cap \Z^d$.
As we explain in Section~\ref{sec:numerical_approach}, a time discretization leads to a sequence of incremental minimization problems for the variables $p_l$ and $\varphi_L$.
The spatial problems to be solved at each time step are strictly convex, coercive
minimization problems, and 
as such closely related to minimization problems that result
from finite element discretizations of continuous formulations of primal
small-strain plasticity~\cite{sander2020plasticity}.
We use the Truncated Non\-smooth Newton Multigrid (TNNMG) method \cite{sander2020plasticity,graeser2019tnnmg}
to solve these minimization problems.

Based on this computational scheme, we explore in Section~\ref{S:exp} the response of the homogenized system to cyclic and monotonic uniaxial loading in the case of a two-dimensional triangular lattice.
While the $0$-dimensional model of the individual springs only features the typical bilinear stress-strain curve of linear kinematic hardening, our simulation shows that the homogenized system features a more complex, nonlinear stress curve, see Figure~\ref{fig:Hystereseloop}. Furthermore, the observed  qualitative properties of the stress--strain relation gives rise to introduce three regimes that are discriminated by the configuration of the internal plastic strain.
More precisely, we consider an elastic and plastic regime corresponding to a state when none or all of the springs are plastically deformed, respectively. The remaining case is referred to as transitional regime.
It turns out that in the elastic and plastic regime the stress-strain relation is linear,
while in the transitional regime the stress-strain curve is nonlinear, see Figure~\ref{fig:SigmaHomTimeFirstThirdComp_ratio}.

The question of how large an RVE has to be chosen to become representative for the effective behavior of the mechanical system is widely discussed in the computational engineering literature, see, e.g.,~\cite{drugan1996micromechanics, kanit2003determination, schneider2022representative} and the references therein.
Our convergence results of Theorems~\ref{thm:1462:d}~and~\ref{T:RVE-W} show that the periodic RVE becomes representative in the limit $L\to\infty$, however, these results are purely qualitative and do not yield a rate of convergence.
In fact, only very little is known theoretically about the speed of convergence:
Existing results only cover the case of linear elasticity and are recent, see \cite{GO11,GO12, gloria2015quantification, BellaOtto16,fischer2019optimal, GNO5} where the theory is developed for elliptic equations and systems.

Motivated by this, in Section~\ref{S:error} we numerically investigate the error of the RVE approximation.
It is customary to decompose the RVE error into a random part, which monitors the fluctuations of the periodic RVE around its mean, and a systematic part, which measures the distance of the mean of the RVE to the homogenized system.
The main point in our numerical study is to identify the rate of convergence of the RVE error as $L\to\infty$.
In our numerical study we consider the case of monotonic loading and explore the RVE-error at the three loading states that correspond to the elastic, plastic and transitional regime, see Figures~\ref{fig:SystematicError}~and~\ref{fig:RandomErrorGridSize}.
Theory only exist for the elastic regime, which in fact can be treated by the existing elliptic theory.
Our simulations suggest that the scaling of the random and systematic error in the transitional and plastic regime are the same as in the elastic regime---this might be surprising: From the elliptic theory we know that the scaling of the error is governed by the mixing properties of the microstructural disorder of the system; however, in the non-elliptic regimes the system features a microstructural disorder that is not only determined by the randomness of the material properties, but also by the state of the internal variables which depends on the path history.

Parts of the results presented in this paper have already been announced in the PhD thesis of \textcite{varga2019stochastic} and the master's thesis of \textcite{haberland:2021}.



\section{Rate-independent model for elastoplastic spring networks}\label{sec:setting}
In this section we introduce a model for the evolution of networks of linear
elastoplastic springs with random mechanical properties. As we shall see,
the system as well as the corresponding homogenized continuum models can be described
as evolutionary rate-independent systems (ERIS) with quadratic energies.
For the readers convenience we recall the basic theory of such ERIS in
the appendix~\ref{sec:eris}. We refer to \cite{mielke2004rate,mielke2005evolution,mielke2015rate}
for a concise introduction to the theory of rate-independent systems. Section~\ref{S:modeling-det} covers networks with heterogeneous
yet \emph{deterministic} material properties. The modeling of the random
material parameters is explained afterwards in Section~\ref{S:random-network}.

\subsection{Modeling of elastoplastic spring networks}\label{S:modeling-det}

\subsubsection{Lattice graphs and discrete derivatives}
We represent networks of elastoplastic springs as graphs whose nodes model material points
that can move in space, and whose edges correspond to the springs. More precisely,
we consider networks that can be described by lattice graphs of the form $(\Z^d,\mathsf E)$
where $\Z^d$ is the set of nodes and $\mathsf E\subset\Z^d\times\Z^d$ denotes
the set of (directed) edges. We assume that the set of edges is periodic
$\mathsf E=\{ (x,y)\,:\,x,y\in\Z^d\text{ with }y-x\in\mathsf E_0\}$,
with a finite subset $\mathsf E_0$ of $\Z^d\setminus\{0\}$. Given an edge $\mathsf e=(x,y)$,
we call $\overline{\mathsf e} \colonequals x$ and $\underline{\mathsf e} \colonequals y$
the start and end nodes, respectively. We also consider for $\varepsilon>0$
the scaled lattice graph $(\varepsilon\Z^d,\varepsilon\mathsf E)$.

For a function $f:\eps\Z^d\to\R^m$, $m \in \N$ we define the \emph{edge derivative}
$\nabla f:\varepsilon\mathsf E\to\R^m$ as
\begin{equation}\label{D:edgederiv}
  \nabla f(\mathsf e)\colonequals \frac{f(\overline{\mathsf e})-f(\underline{\mathsf e})}{|\overline{\mathsf e}-\underline{\mathsf e}|}.
\end{equation}
For $u:\eps\Z^d\to\R^d$ we define the \emph{projected edge derivative}
$\nabla_s u:\varepsilon\mathsf E\to\R$ via
\begin{equation}\label{D:D:edgeproj}
  \nabla_s u(\mathsf e)\colonequals \frac{\overline{\mathsf e}-\underline{\mathsf e}}{|\overline{\mathsf e}-\underline{\mathsf e}|}\cdot\nabla u(\mathsf e).
\end{equation}

\begin{remark}
If we think of $u$ as the \emph{displacement} of the network, then $\nabla_su$ can be seen as a discrete version of the infinitesimal strain tensor (see \cite[Introduction]{neukamm2017stochastic}). In particular, if $u$ is a linear displacement, say $u(x)=Fx$, then $\nabla_su(\mathsf e)=\frac{\overline{\mathsf e}-\underline{\mathsf e}}{|\overline{\mathsf e}-\underline{\mathsf e}|}\cdot F\frac{\overline{\mathsf e}-\underline{\mathsf e}}{|\overline{\mathsf e}-\underline{\mathsf e}|}$ and we see that $\nabla_s u$ only depends on the symmetric part ${\rm sym} F:=\tfrac12(F+F^\top)$.  
\end{remark}
\medskip

We assume that the lattice $(\Z^d,\mathsf E)$ is non-degenerate in the following sense:
\begin{enumerate}[label=(A\arabic*)]
\item \label{B0:pl:d} The set $\mathsf E_0$ contains (at least) the standard basis $\{e_1,\ldots,e_d\}$ of $\R^d$ and there exists a constant $C>0$ such that for all $u:\Z^d\to\R^d$ with compact support the discrete Korn inequality
  \begin{equation*}
    \sum_{\mathsf e\in\mathsf E}|\nabla u(\mathsf e)|^2\leq C\sum_{\mathsf e\in\mathsf E}|\nabla_s u(\mathsf e)|^2
  \end{equation*}
holds.
\end{enumerate}

\begin{remark}
  By considering affine displacements (multiplied with a cut-off function to yield a compact support), it is easy to see that Assumption~\ref{B0:pl:d} implies that 
  \begin{equation}\label{eq:12333}
    |\operatorname{sym} F|^2\leq C\sum_{e\in\mathsf E_0}\big(\tfrac{e}{|e|}\cdot F\tfrac{e}{|e|}\big)^2\qquad\text{for all }F\in\R^{d\times d}.
  \end{equation}
  This explains why we call the estimate in Assumption~\ref{B0:pl:d} a discrete Korn inequality.
  Mechanically, it means that the network can resist shear deformations. A basic example that violates Assumption \ref{B0:pl:d} is the lattice $\Z^d$ with nearest-neighbor edges $\mathsf E_0=\{e_1,\ldots,e_d\}$. In that case the right-hand side of \eqref{eq:12333} only controls the diagonal of $F$ but not the entire symmetric part.
  On the other hand, if $\mathsf E_0$ contains the set
  \begin{equation*}
    \{e_1,\ldots,e_d\}\cup \{e_j-e_i\,:\,1\leq i<j\leq d\,\}
  \end{equation*}
  then Assumption~\ref{B0:pl:d} is satisfied.
\end{remark}

\begin{figure}
  \begin{center}
  \includegraphics[width=5cm]{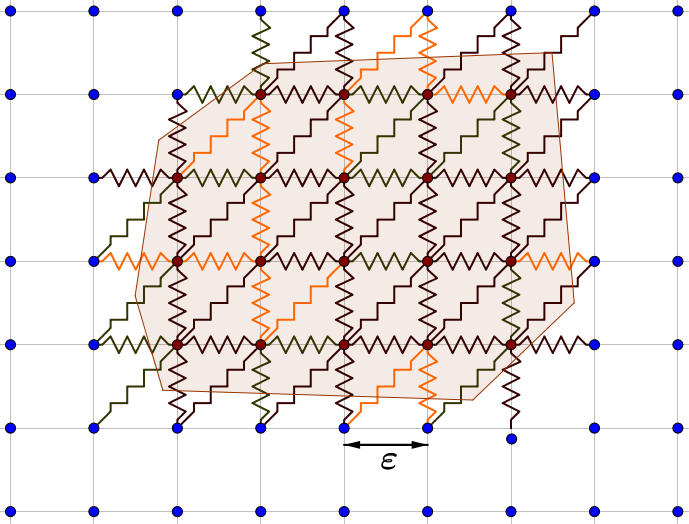}
\end{center}
 \caption{The two-dimensional, triangular network is generated by the edge set
   $\mathsf E_0\colonequals \{\,(1,0),\,(0,1),\,(1,1)\,\}$. A macroscopic domain $Q\subset\R^2$ is shown in orange. The nodes in $\eps\Z^d\cap Q$ are red. The edges in $\eps\mathsf E\cap Q$ are shown as springs.  }
 \label{fig:example_network}
\end{figure}

\begin{example}[Triangular two-dimensional network]\label{EX1}
  A graph in dimension $d=2$ that satisfies \ref{B0:pl:d} is the triangular lattice generated by the set
  \begin{equation*}
    \mathsf E_0\colonequals \{\,(1,0),\,(0,1),\,(1,1)\,\},
  \end{equation*}
  and shown in Figure~\ref{fig:example_network}. We shall consider this specific example
  in our numerical experiments in Section~\ref{S:exp}.
\end{example}

\subsubsection{Elastoplastic spring networks as evolutionary rate-independent systems}
\label{eq:model}
In the following we introduce a quadratic ERIS that describes the evolution of a network
with elastoplastic springs at edges $\mathsf e\in\varepsilon\mathsf E$. For a fixed time
we describe the state of the deformed network by a displacement field $u:\eps\Z^d\to\R^d$
and a longitudinal plastic strain $p:\eps\mathsf E\to\R$. Let $Q$ be an open, bounded
Lipschitz domain in $\R^d$
and assume that $(u,p)$ vanishes outside of $Q$.
We denote by
\begin{equation}\label{eq:EQ}
  \varepsilon\mathsf E\cap Q
  \colonequals
  \Big\{ (x,y)\in\varepsilon\mathsf E\,:\,\{x,y\}\cap Q\neq\emptyset \Big\}
\end{equation}
the set of edges with at least one node in $Q$.
Although $u$ and $p$ effectively only depend on the nodes and edges in $Q$, we view $u$ and $p$ as functions defined on $\eps\Z^d$ and $\eps\mathsf E$, respectively, since this simplifies the presentation.

Recall from Appendix~\ref{sec:eris} that an ERIS consists of a state space $\widetilde Y\e$,
an energy functional $\widetilde{\mathcal E}\e$, and a dissipation functional
$\widetilde{\mathcal R}\e$.
For the definition of the state space of the system, we introduce the norms
\begin{equation*}
  \|u\|_{L^2(\varepsilon\Z^d)^d}\colonequals \Bigg(\varepsilon^{-d}\sum_{x\in\varepsilon\Z^d}|u(x)|^2\Bigg)^\frac12
  \qquad \text{and} \qquad
  \|p\|_{L^2(\varepsilon\mathsf E)}\colonequals \left(\varepsilon^{-d}\sum_{\mathsf e\in \varepsilon\mathsf E}|p(\mathsf e)|^2\right)^\frac12,
\end{equation*}
and the spaces
\begin{align*}
  L^2(\varepsilon\Z^d)^d\colonequals \Big\{u:\varepsilon\Z^d\to\R^d\,:\,\|u\|_{L^2(\varepsilon\Z^d)}<\infty \Big\}
  \text{ and }
  L^2(\varepsilon\mathsf E)\colonequals \Big\{p:\varepsilon\mathsf E\to\R\,:\,\|p\|_{L^2(\varepsilon\mathsf E)}<\infty \Big\}.
\end{align*}
The norms correspond to scalar products, and therefore the spaces are Hilbert spaces.
As the state space of the ERIS we use the product
\begin{equation*}
  \widetilde Y\e\colonequals L^2_0(\varepsilon\Z^d\cap Q)^d\times L^2_0(\varepsilon\mathsf E\cap Q),
\end{equation*}
where
\begin{align*}
  L^2_0(\varepsilon\Z^d\cap Q)^d&\colonequals \Big\{u\in L^2(\varepsilon\Z^d)^d\,:\,u=0\text{ outside of $Q$} \Big\},\\
  L^2_0(\varepsilon\mathsf E\cap Q)&\colonequals \Big\{p\in L^2(\varepsilon\mathsf E)\,:\,p=0\text{ outside of $\varepsilon\mathsf E\cap Q$} \Big\}.
\end{align*}
On $\widetilde Y\e$ we define the product norm
\begin{equation*}
  \|y\|_{\widetilde Y\e}^2\colonequals \|u\|_{L^2(\varepsilon\Z^d)^d}^2+\|p\|_{L^2(\varepsilon\mathsf E)}^2,
\end{equation*}
which turns $\widetilde Y\e$ into a Hilbert space.

We use the notation $y=(u,p)$ for the components of the state variable.
The stored energy of the system is described by $\widetilde{\mathcal E}\e(t,\cdot):\widetilde Y\e\to\R$, 
\begin{equation*}
  \widetilde{\mathcal E}\e(t,y)\colonequals \frac{1}{2}\eps^d\sum_{\mathsf e\in\varepsilon\mathsf E}\Big(a(\mathsf e/\varepsilon)\big(\nabla_s u(\mathsf e)-p(\mathsf e)\big)^2+h(\mathsf e/\varepsilon)p(\mathsf e)^2\Big)-\eps^d\sum_{x\in \varepsilon\Z^d}l(t,x)\cdot u(x),
\end{equation*}
where $a,h:\mathsf E\to(0,\infty)$ describe the elastic modulus and kinematic hardening parameter
of the springs, respectively, and $l:[0,T]\times \varepsilon\Z^d\to\R^d$ denotes an external force acting on the nodes.
Note that by assumption the sums in the definition of $\widetilde{\mathcal E}\e$
effectively only run over the finite sets $\eps\mathsf E\cap Q$ and $\eps\Z^d\cap Q$, respectively.

As dissipation functional we use
\begin{equation*}
  \widetilde{\mathcal R}\e(\dot y)\colonequals \eps^d\sum_{\mathsf e\in\varepsilon\mathsf E}\sigma_{\rm yield}(\mathsf e/\varepsilon)|\dot p(\mathsf e)|,
\end{equation*}
where $\sigma_{\rm yield}:\mathsf E\to(0,\infty)$ describes the yield strength of the springs. (Later we shall consider more general dissipation functionals.)
We assume that the springs are non-degenerate in the sense that there exists $\lambda>0$ such that
\begin{equation}\label{eq:unifparam}
  \lambda\leq a(\mathsf e),h(\mathsf e)\leq\frac1\lambda,\qquad 0\leq\sigma_{\rm yield}(\mathsf e)\leq \frac1\lambda\qquad \text{ for all }\mathsf e\in\mathsf E.
\end{equation}

The evolution of the network of elastoplastic springs is described by a curve $y\colonequals (u,p):[0,T]\to\widetilde Y\e$ that satisfies the force balance equations
\begin{equation}\label{eq:FBE}
  0 \in D_{y}\widetilde{\mathcal E}\e(t,y(t))+ \partial \widetilde{\rcal}\e(\dot{y}(t))\qquad\text{for all }t\in[0,T].
\end{equation}
Here, $D_{y}$ denotes the Gateaux derivative with respect to the state variable,
and $\partial$ the subdifferential from convex analysis.

Equivalently, the system is a quadratic ERIS in the sense of Definition~\ref{D:ERIS}.
Indeed, the state space $\widetilde Y\e$ is Hilbert space and the dissipation functional $\widetilde\rcal\e$
is convex, continuous and positively homogeneous of degree $1$, and thus satisfies the required properties.
Therefore, we only need to show that the quadratic energy is generated by
a positive definite operator. To that end, first note that,
provided $l(t,x)$ is sufficiently regular, we may rewrite $\widetilde{\mathcal E}\e$ as
\begin{equation*}
  \widetilde{\mathcal E}\e(t,y)=  \frac12\expect{\widetilde{\mathbb A}\e
    \begin{pmatrix}
      u\\p
    \end{pmatrix},    \begin{pmatrix}
      u\\p
    \end{pmatrix}
  }-  \expect{\widetilde{\ell}\e(t)
    \begin{pmatrix}
      u\\p
    \end{pmatrix},    \begin{pmatrix}
      u\\p
    \end{pmatrix}
  }
\end{equation*}
for some $\widetilde \ell\e\in W^{1,1}((0,T);(\widetilde Y\e)^*)$ and a linear and symmetric operator $\widetilde{\mathbb A}\e:\widetilde Y\e\to (\widetilde Y\e)^*$ that is uniquely determined by
\begin{equation*}
  \expect{\widetilde{\mathbb A}\e
    \begin{pmatrix}
      u\\p
    \end{pmatrix},    \begin{pmatrix}
      u\\p
    \end{pmatrix}
}\colonequals \varepsilon^d\sum_{\mathsf e\in\varepsilon\mathsf E}\Big(a(\mathsf e/\varepsilon)\big(\nabla_su(\mathsf e)-p(\mathsf e)\big)^2+h(\mathsf e/\varepsilon)p(\mathsf e)^2\Big)\qquad\text{for all }(u,p)\in\widetilde Y\e.
\end{equation*}
By \eqref{eq:unifparam} there exists $c=c(\lambda)c>0$ such that
\begin{equation*}
  \frac1c\expect{\widetilde{\mathbb A}\e
    \begin{pmatrix}
      u\\p
    \end{pmatrix},    \begin{pmatrix}
      u\\p
    \end{pmatrix}
  }\geq \varepsilon^d\sum_{\mathsf e\in\varepsilon\mathsf E}|\nabla_su(\mathsf e)|^2+|p(\mathsf e)|^2.
\end{equation*}
In view of the discrete Korn inequality of \ref{B0:pl:d}, and the fact that for any $u\in L^2_0(\varepsilon\Z^d\cap Q)^d$ the discrete Poincaré inequality
\begin{equation*}
  \|u\|_{L^2(\varepsilon\Z^d)^d}\leq C\|\nabla u\|_{L^2(\varepsilon\mathsf E)}
\end{equation*}
holds with a constant only depending on $Q$, 
we deduce that $\widetilde{\mathbb A}\e$ is positive definite (uniformly in $\eps$), and $(\widetilde Y\e,\widetilde{\mathcal E}\e, \widetilde{\rcal}\e)$ is indeed a quadratic ERIS. Consequently Theorem~\ref{T:energetic_solution} yields for any stable (in the sense of the theorem) initial state $\tilde y^0$
the existence of a unique function  $\tilde y\e\in W^{1,1}((0,T);\widetilde Y\e)$ satisfying $\tilde y\e(0)=\tilde y^0$ and the force balance equations \eqref{eq:FBE}.

\subsection{Random heterogeneous elastoplastic spring network}\label{S:random-network}

The material parameters of the spring network introduced in the previous section
were deterministic, and given by the three functions $a,h,\sigma_{\rm yield}:\mathsf E\to(0,\infty)$.
Having homogenization of random heterogeneous networks in mind, we are interested
in situations where these material parameters are rescaled random fields with a
stationary and ergodic distribution. Roughly speaking this means that the parameter functions
are random fields on $\mathsf E$ that are statistically homogeneous and feature
a decay of correlations.  This means, for example, that the field $a$ and the
shifted field $a(x+\cdot)$ decorrelate for shifts $x\in\Z^d$ that diverge to infinity.

In the following we explain the functional analytic framework that we use to describe the random situation. Furthermore, we introduce the model class of \textit{random heterogeneous rate-independent networks}, which is based on slightly more general, node-centric formulation and contains the spring network of the previous section as a special case.

\subsubsection{Stationary and ergodic random fields}\label{S:probsetup}
Let  $\brac{\Omega,\mathcal{F},P}$ denote a complete and separable probability space that is equipped with a discrete dynamical system $\tau \colonequals \cb{\tau_x:\Omega\to \Omega}_{x\in \Z^d}$ such that:
\begin{enumerate}[label=(S\arabic*)]
\item \label{item:assumption_measurability} (Measurability) The maps $\tau_x:\Omega\to\Omega$ are invertible and measurable for all $x \in \mathbb{Z}^d$.
\item (Group property) $\tau_0=\Id$ and $\tau_{x+y}=\tau_x \circ \tau_y$ for all $x,y\in \mathbb{Z}^d$.
\item (Measure preservation) $P(\tau_x E)=P(E)$ for all $E\in \mathcal{F}$ and $x\in \mathbb{Z}^d$.\label{S:stat}
\item \label{item:assumption_ergodicity} (Ergodicity) Every $E\in \mathcal{F}$ that is shift invariant (i.e., $\tau_x E=E$ for all $x\in\mathbb Z^d$) satisfies $P(E)\in\{0,1\}$.
\end{enumerate}
We denote by $\mathbb E[\cdot]$ the expectation with respect to $P$.
\medskip

Following \cite{Papanicolaou1979} and \cite{neukamm2017stochastic} we model general stationary random fields as follows: To any random variable $a_0:\Omega\to\R$, we may associate the random field $a:\Omega\times\Z^d\to\R$, $a(\omega,x)\colonequals a_0(\tau_x\omega)$, which in view of \ref{S:stat} is stationary. The latter means that finite dimensional distributions of $a$ are invariant under spatial shifts, i.e., for any $M\in\N$, $x_1,\ldots,x_m\in\Z^d$ and shift $z\in\Z^d$, the random variable $\{a(\omega,x_1+z),\ldots, a(\omega,x_M+z)\}$ has a joint probability distribution that is independent of $z$. We remark that stationary random fields are used to describe statistically homogeneous material properties of random heterogeneous materials, see \cite{Torquato2013}.

A consequence of the ergodicity assumption~\ref{item:assumption_ergodicity}
is the possibility to approximate the expectation of a stationary random field with help of a spatial averages:
\begin{theorem}[Ergodic theorem]\label{T:ergodic}
  Assume \ref{item:assumption_measurability}--\ref{item:assumption_ergodicity} and $\varphi\in L^p(\Omega)$ with $1\leq p<\infty$. Then
  \begin{equation*}
    \mathbb E[\varphi]=\lim\limits_{L\to\infty}L^{-d}\sum_{z\in \Lambda_L}\varphi(\tau_z\cdot)\text{ in }L^p(\Omega),
  \end{equation*}
  where $\Lambda_L:=\Z^d\cap [0,L)^d$.
  Moreover, this convergence also holds pointwise $P$-a.e.~in $\Omega$.
\end{theorem}

The following example will be explored numerically in Section~\ref{S:exp}.

\begin{example}[Network with i.i.d.\ coefficients]\label{EX2}
  A simple example of a situation with stationary and ergodic material properties is the case of independent and identically distributed (i.i.d.) parameters: We assume that $a(\mathsf e)$, $h(\mathsf e)$, and $\sigma_{\rm yield}(\mathsf e)$ with $\mathsf e\in\mathsf E$ are independent random variables, and that for $i=1,\dots,|\mathsf E_0|$ the random variables $\{a((x,x+e_i))\}_{x\in\Z^2}$, $\{h((x,x+e_i))\}_{x\in\Z^2}$, and $\{\sigma_{\rm yield}((x,x+e_i))\}_{x\in\Z^2}$)
  are identically distributed with laws $P_{a,i}$, $P_{h,i}$, and $P_{\sigma_{\rm yield},i}$,
  respectively. Hence, the random heterogeneous material is fully determined by the
  $3 |\mathsf E_0|$ probability measures $P_{a,1},\ldots, P_{\sigma_{\rm yield},|\mathsf E_0|}$. For instance we may consider for these probability measures uniform distributions on compact intervals contained in $(0,\infty)$.
\end{example}

\subsubsection{Node-centric formulation}
\label{sec:node_centric_formulation}

In our network model, material properties are given as functions defined on edges.
However, for the notion of stationarity and related tools from stochastic homogenization it is more natural to consider functions defined on nodes.
Therefore, in the following we identify functions defined on edges with (vector valued) functions defined on nodes. To that end, we fix an enumeration of the vectors $\mathsf E_0$, i.e.,
\begin{equation*}
  \mathsf E_0=\{e_1,\ldots,e_d,e_{d+1},\ldots,e_k\}
\end{equation*}
with the convention that $e_1,\ldots,e_d$ denotes the standard basis of $\R^d$.
We then associate with an edge function $f:\varepsilon\mathsf E\to\R$ the $\R^k$-valued node function $\bar f:\varepsilon\Z^d\to\R^k$,
\begin{equation*}
  \bar f(x)\colonequals \Big(p((x,x+\varepsilon e_1)),\ldots,p((x,x+\varepsilon e_k))\Big)\in\R^k,\qquad x\in\varepsilon\Z^d.
\end{equation*}
Throughout the paper, we shall simply write $f$ instead of $\bar f$ and thus identify any edge function with the associated $\R^k$-valued node function. In particular, the symmetrized gradient
of a displacement $u:\varepsilon\Z^d\to\R^d$ takes the form
\begin{equation*}
  \nabla_su:\varepsilon\Z^d\to\R^k, \qquad \nabla_su(x)=\Big(\nabla_s u((x,x+\varepsilon e_1)),\ldots,\nabla_s u((x,x+\varepsilon e_k))\Big),
\end{equation*}
where $\nabla_s$ on the right-hand side denotes the projected edge derivative.

\begin{example}\label{E:iid}
  This situation of Example~\ref{EX2} can be phrased in form of the abstract framework introduced above. To that end consider the ``single node'' probability space
  \begin{equation*}
    \Omega_0:=\R^{2|\mathsf E_0|},\ \mathcal F_0:=\mathcal B\big(\R^{2|\mathsf E_0|}\big),\,P_0:=P_{a,1}\otimes\cdots\otimes P_{\sigma_{\rm yield},|\mathsf E_0|},
  \end{equation*}
  which describes the joint distribution of the material parameters of the $|\mathsf E_0|$ edges whose starting node is the origin. We then consider the product probability space
  \begin{equation*}
    \Omega:=\otimes_{\Z^d}\Omega_0=\{\omega:\Z^d\to\R^{2|\mathsf E_0|}\},\qquad \mathcal F:=\otimes_{\Z^d}\mathcal F_0,\qquad P:=\otimes_{\Z^d}P_0,
  \end{equation*}
  and the shift $\tau:\Omega\times\Z^d\to\Omega$, $\tau_z\omega:=\omega(\cdot+z)$. Then it is well-known that properties~(S1)--(S4) are satisfied. In order to model the material parameters $a,h$ and $\sigma_{\rm yield}$ as random fields on $\Omega$, we may simply set
  \begin{equation*}
    a(\omega,(x,x+e_i))=\omega_{i}(x),\quad h(\omega,(x,x+e_i))=\omega_{|\mathsf E_0|+i}(x),\quad \sigma_{\rm yield}(\omega,(x,x+e_i))=\omega_{2|\mathsf E_0|+i}(x).
  \end{equation*}
\end{example}

  In line with the convention to identify edge functions with node functions, we may rewrite the energy functional of Section~\ref{eq:model} as
\begin{equation*}
  \widetilde{\mathcal E}\e(t,y)= \eps^d\sum_{x\in\eps\Z^d} \bigg[\frac12\widetilde A(x/\varepsilon)\binom{\nabla_su(x)}{p(x)}\cdot \binom{\nabla_s u(x)}{p(x)}-l(t,x)\cdot u(x) \bigg]
\end{equation*}
for a positive definite and symmetric coefficient matrix $\widetilde A:\Z^d\to\R^{2k\times 2k}_{\rm sym}$ that is uniquely determined by the identity
\begin{equation*}
  \widetilde A(x)\binom{d}{p}\cdot\binom{d}{p}=\sum_{i=1}^ka((x,x+e_i))(d_i-p_i)^2+h((x,x+e_i))p_i^2\qquad\text{for all }d,p\in\R^k.
\end{equation*}
Likewise the dissipation functional can be rewritten as
\begin{equation*}
  \widetilde{\rcal}\e(\dot y)=\varepsilon^d\sum_{x\in\eps\Z^d}\tilde\rho(x/\varepsilon,\dot p)
\end{equation*}
where $\tilde\rho:\Z^d\times\R^k\to[0,\infty)$ is defined as
\begin{equation*}
  \tilde\rho(x,\dot p)=\sum_{x\in\eps\Z^d}\sum_{i=1}^k\sigma_{\rm yield}((x,x+e_i))|\dot p_i|.
\end{equation*}
We note that $\tilde\rho(x,\cdot)$ is convex and positively homogeneous of degree 1.


\subsubsection{Random heterogeneous rate-independent network}
We now introduce a discrete rate-independent system that is
the starting point of our asymptotic analysis. It is formulated as an ERIS with a form similar to the one derived in Section~\ref{sec:node_centric_formulation}, and it contains as a special case the model of Section~\ref{eq:model} for the network composed of elastoplastic springs with random heterogeneous material parameters.

In the following we often use the tensor product space notation for products of Hilbert spaces. We explain the notation in the appendix~\ref{app:tensorproduct}.

The state space is
 \begin{equation*}
  Y\e=\brac{L^2(\Omega)\otimes L^2_0(Q\cap \varepsilon\mathbb{Z}^{d})}^d \times \brac{L^2(\Omega)\otimes L^2(\eps\Z^d)}^k,
\end{equation*}
where the first factor again captures the node displacements (which are now random),
and the second factor captures the plastic strains.
With $y = (u,p)$ we consider the energy functional $\mathcal{E}\e:[0,T]\times Y_{\varepsilon} \to \mathbb{R}$,
\begin{align*}
  \mathcal{E}\e(t,y)
  &\colonequals
  \mathbb E\bigg[\varepsilon^d\sum_{x\in\eps\Z^d}\frac12 A(\tau_{\frac{x}{\varepsilon}}\omega)\binom{\nabla_su(\omega,x)}{p(\omega,x)}\cdot \binom{\nabla_s u(\omega,x)}{p(\omega,x)}-l\e(t,x)\cdot u(\omega,x)\bigg]
\end{align*}
where $l_{\varepsilon}:[0,T] \times \varepsilon \Z^d \to \R^d$ describes the applied load. We assume that
\begin{equation}\label{ass:smooth_load}
  \forall x\in\varepsilon\Z^d\,:\,l\e(\cdot,x)\in W^{1,1}([0,T]).
\end{equation}
The random material properties (e.g., the elastic modulus and the kinematic hardening parameter in the network model of Section~\ref{eq:model})
are described by the random matrix $A:\Omega\to\R^{2k\times 2k}_{\mathrm{sym}}$ and we assume that
\begin{enumerate}[label=(A\arabic*)]
  \setcounter{enumi}{1}
\item \label{B1:pl:d} $A\in L^{\infty}(\Omega,\mathbb{R}^{2k\times 2k}_{\mathrm{sym}})$ and there exists $c>0$ such that $A(\omega)F\cdot F\geq c|F|^2$ for $P$-a.a. $\omega\in \Omega$ and all $F\in \mathbb{R}^{2k}$.
\end{enumerate}
Plastic dissipation is modeled by the dissipation functional $\rcal\e:Y_{\varepsilon}\to [0,\infty]$,
\begin{equation*}
  \rcal\e(\dot{y})\colonequals \mathbb E\bigg[\varepsilon^d\sum_{x\in\eps\Z^d}\rho(\tau_{\frac{x}{\varepsilon}}\omega,\dot{p}(\omega,x))\bigg].
\end{equation*}
We assume that
\begin{enumerate}[label=(A\arabic*)]
  \setcounter{enumi}{2}
 \item \label{B2:pl:d} $\rho:\Omega\times \mathbb{R}^k\to [0,\infty)$ is measurable 
   and $\rho(\omega,\cdot)$ is convex and positively homogeneous of degree~$1$ for $P$-a.a.\ $\omega$. Moreover, we assume that there exists a function $\psi \in L^2(\Omega)$ such that
   for $P$-a.a.~$\omega$ we have the bound $\rho(\omega,F)\leq \psi(\omega)(|F|+1)$ for all $F\in \R^k$.
 \end{enumerate}

 We remark that the lattice model of Section~\ref{sec:node_centric_formulation} is a special case of the above setting, which allows for more general coefficient matrices $A$ and dissipation densities $\rho$ than those considered in Section~\ref{sec:node_centric_formulation}.
In particular, since we allow for $\rho$ to depend on the strain rates in all edges that start at node, it is not always possible to write the energy functional and the dissipation functional in an edge-centric formulation.

 \begin{remark}
   Under assumptions \ref{item:assumption_measurability}--\ref{item:assumption_ergodicity},
   \ref{B0:pl:d}--\ref{B2:pl:d}, and \eqref{ass:smooth_load}, the system $\brac{Y_{\varepsilon}, \mathcal{E}_{\varepsilon}, \mathcal{R}_{\varepsilon}}$ is a quadratic ERIS in the sense of Definition~\ref{D:ERIS} and we may appeal to Theorem~\ref{T:energetic_solution} for the existence and uniqueness of energetic solutions to the initial value problem. Since we assume
   that $u\e$ vanishes outside of $Q$, it is natural to assume that the initial plastic strain $p^0$ vanishes for all edges that are not contained in $\eps\mathsf E\cap Q$ (see \eqref{eq:EQ}), i.e.,
   \begin{equation}\label{eq:technical_condition_p0}
     \forall x\in\eps\Z^d, i\in\{1,\ldots,k\}:\qquad(x,x+\eps e_i)\not\in\eps\mathsf E\cap Q\,\Rightarrow\,p^0_i(x)=0.
   \end{equation}
    One can easily check that this property is inherited by the energetic solution of the ERIS: Indeed, for a.a.~$t\in[0,T]$ the plastic strain $p\e(t,\cdot)$ vanishes for all edges that are not contained in $\eps\mathsf E\cap Q$.
  \end{remark}

  \begin{remark}[Mean versus quenched]
    The ERIS $(Y\e, {\mathcal E}\e,\rcal\e)$ is a mean model, i.e., the definition of the energy and dissipation functional invokes the expectation $\mathbb E[\cdot]$, and thus an average over all realizations $\omega\in\Omega$ of the random network. A corresponding \textit{quenched} ERIS is given by considering for each realization $\omega\in\Omega$ sampled with $P$, the functionals
    \begin{align*}
      \mathcal{E}\e^\omega(t,y)
  &\colonequals
    \varepsilon^d\sum_{x\in\eps\Z^d}\frac12 A(\tau_{\frac{x}{\varepsilon}}\omega)\binom{\nabla_su(x)}{p(\omega,x)}\cdot \binom{\nabla_s u(x)}{p(x)}-l\e(t,x)\cdot u(x),\\
      \rcal^\omega\e(\dot y)
      &\colonequals \varepsilon^d\sum_{x\in\eps\Z^d}\rho(\tau_{\frac{x}{\varepsilon}}\omega,\dot{p}(x)),
    \end{align*}
    together with the state space $\widetilde Y\e$. The mean model may be viewed as an integrated version of the quenched model. In particular, if $y^\omega$ and $y$ are solutions to the quenched and mean model, then $y^\omega(t,x)=y(\omega,t,x)$ for $P$-a.e.~$\omega\in\Omega$, cf.~Remark~\ref{R:reltoquenched}.
  \end{remark}



\section{The homogenized  rate-independent system}
As shown in \cite{neukamm2017stochastic}, in the limit $\varepsilon\to 0$ the ERIS $(Y\e,\mathcal E\e,\rcal\e)$
describing an elastoplastic spring network on a lattice with spacing $\varepsilon$
converges (in the sense of evolutionary $\Gamma$-convergence) to a homogenized ERIS $(Y_{\hom},\mathcal E_{\hom},\rcal_{\hom})$ that describes the rate-independent evolution of an elastoplastic material
on the continuum domain $Q$. In this section we present this homogenized ERIS and recall the convergence result from \cite{neukamm2017stochastic}.
Moreover, as a new result we show that the homogenized ERIS can be reformulated as a classical force balance equation with an effective Cauchy stress tensor $\sigma_{\hom}$.
This tensor is related to the linearized continuum strain by a generalized
Prandtl--Ishlinskii operator which captures the microstructure evolution.

\subsection{Definition of  the homogenized ERIS \texorpdfstring{$(Y_{\hom},\mathcal E_{\hom},\rcal_{\hom})$}{(Yhom,Ehom,Rhom)}}
\label{sec:definition_homogenized_system}
Evolutionary $\Gamma$-convergence of the spring network model $(Y\e,\mathcal E\e,\rcal\e)$
leads to a homogenized ERIS that, besides the \emph{macroscopic} displacement field and plastic strain,
features additional \emph{microscopic} degrees of freedom. The latter describe the evolving,
microscopic distortion of the system and influence the homogenized behavior.
They are described by strain vectors that are defined on the probability space $(\Omega,\mathcal F,P)$ and take values in $\R^k$ where $k \colonequals |\mathsf E_0|$, i.e., each component of the strain vector corresponds to one of the $k$ springs in $\mathsf E_0$.
To be precise, we denote by
$L^2_{s}(\Omega)\subset L^2(\Omega)^k$ the closed subspace of random strain vectors defined as the closure in $L^2(\Omega)^k$ of the space random, projected partial derivatives,
\begin{equation}\label{eq:L2pot}
  \bigg\{\bigg(\frac{e_1}{|e_1|}\cdot (u(\tau_{e_1}\omega){-}u(\omega)),\ldots,\frac{e_k}{|e_k|}\cdot (u(\tau_{e_k}\omega){-}u(\omega))\bigg)^\top\,:\,u\in L^2(\Omega)^d\bigg\}.
\end{equation}
The homogenized ERIS is then defined on the state space
\begin{equation}
\label{D:Yhom}  Y_{\hom} \colonequals H^1_0(Q)^d\times \big(L^2(\Omega)\otimes L^2(Q)^k\big)\times \big(L^2_{s}(\Omega)\otimes L^2(Q)\big).
\end{equation}
An element $y\in Y_{\hom}$ consists of three components $y=(u,p,\chi)$, where $u$
is the macroscopic displacement, $p$ is the plastic strain vector (with macroscopic contribution
$\mathbb E[p]$ and microscopic fluctuation $p-\mathbb E[p]$) and $\chi_s$ has the meaning of a
microscopic strain vector. The two strain vectors $p$ and $\chi_s$ are random fields,
while $u$ is a deterministic quantity.

To relate the strain vectors from $\R^k$ to the symmetric $d \times d$ strain matrices
typically used in $d$-dimensional small-strain elastoplasticity, we introduce a linear operator:
\begin{definition}[Tensor-to-vector conversion]
  Let $\{e_1,\ldots,e_k\}$ denote an enumeration of the edge-generating set $\mathsf E_0$ with the convention that $e_1,\ldots,e_d$ is the standard basis of $\R^d$. We define the \textit{tensor-to-vector conversion operator} as
  \begin{equation*}
    P_s:\R^{d\times d}\to\R^k,\qquad P_s(F)
    \colonequals
    \Big(\frac{e_1}{|e_1|}\cdot F\frac{e_1}{|e_1|},\ldots,\frac{e_k}{|e_k|}\cdot F\frac{e_k}{|e_k|}\Big)^\top.
  \end{equation*}  
\end{definition}
In particular, we apply this map to gradients of displacement fields $u\in H^1_0(Q)$
to obtain \emph{macroscopic elastic strain vectors}
\begin{equation*}
 P_s(\nabla u(x))\in\R^k.
\end{equation*}

\begin{remark}[Vector-to-tensor conversion]\label{R:strainvector}
  By definition, $P_s(F)$  only depends on the symmetric part of $F\in\R^{d\times d}$. In view of Assumption~\ref{B0:pl:d} there exists a constant $c>0$ such that for any $F\in\R^{d\times d}$
  \begin{equation*}
    \tfrac1c|\operatorname{sym} F|\leq |P_s(F)|\leq c|F|,
  \end{equation*}
  and thus $P_s\vert_{\R^{d\times d}_{\rm sym}}$ is injective. We denote by $P_s^*:\operatorname{range}(P_s)\to\R^{d\times d}_{\rm sym}$ the adjoint of $P_s$ defined by the identity $v\cdot P_s(F)=P_s^*(v)\cdot F$ for all $v\in \operatorname{range}(P_s)$ and $F\in\R^{d\times d}_{\rm sym}$.
\end{remark}

\begin{example}
\label{ex:stress_vector}
  In the case of the triangular, two-dimensional lattice $\Z^2$ of Example~\ref{EX1}
  with $\mathsf E_0 =\{e_1,e_2,e_3\colonequals e_1+e_2\}$ we have
  \begin{equation*}
    P_s(\sigma)=\Big(\sigma_{11},\sigma_{22},\frac12\sum_{i,j=1}^2\sigma_{ij}\Big)^\top,
  \end{equation*}
  for all $\sigma\in\R^{2\times 2}_{\rm sym}$,  and for any $s\in\operatorname{range}(P_s)$ we have
  \begin{equation*}
    P_s^*(s)=
      \sigma=
      \begin{pmatrix}
        s_1+\frac12 s_3 &\frac12 s_3\\
        \frac12 s_3&s_2+\frac12 s_3
      \end{pmatrix}.
  \end{equation*}
\end{example}

The energy functional $\mathcal E_{\hom}:[0,T]\times Y_{\hom}\to\R$ of the homogenized ERIS is
\begin{align*}
  \mathcal E_{\hom}(t,y) \colonequals \frac12\int_Q \mathbb E\bigg[A(\omega)
  \begin{pmatrix}
    \nabla_su(x)+\chi_s(\omega,x)\\
    p(\omega,x)
  \end{pmatrix}\cdot   \begin{pmatrix}
    \nabla_su(x)+\chi_s(\omega,x)\\
    p(\omega,x)
  \end{pmatrix}\bigg]\,dx-\int_Ql(t,x)\cdot u(x)\,dx,
\end{align*}
and the dissipation functional $\rcal_{\hom}:Y_{\hom}\to[0,\infty]$ is
\begin{equation*}
  \rcal_{\hom}(\dot y)\colonequals\int_Q\mathbb E\big[\rho(\omega,\dot p(\omega,x))\big]\,dx,
\end{equation*}
where $y=(u,p,\chi_s)$.

\subsection{Interpretation as a two-scale system}\label{S:two-scale_system}
  There is a natural interpretation of the homogenized system: To each macroscopic material point
  $x\in Q$ we may associate a copy of the network $(\Z^d,\mathsf E)$, which we would recover
  by zooming into the material at $x$. In view of the definition of the state space $Y_\textup{hom}$,
  to each $x\in Q$ we may assign besides the macroscopic, elastic strain vector $\nabla_su(x)\in\R^k$ the two random vectors
  \begin{equation*}
    p(\cdot,x)\in L^2(\Omega)^k
    \qquad \text{and} \qquad
    \chi_s(\cdot,x)\in L^2_s(\Omega).
  \end{equation*}
  As we shall see next, these two random vectors have the meaning of a microscopic plastic
  elastic strains of the lattice $(\Z^d,\mathsf E)$ attached at $x$, respectively.
  Indeed, for a typical sample $\omega\in\Omega$ (and thus for a configuration of the material
  of the network) and fixed $x\in Q$ we may consider the function
  \begin{equation*}
    p^{\omega,x}:\Z^d\to\R^k,\qquad p^{\omega,x}(z) \colonequals p(\tau_{z}\omega,x)\in\R^k,
  \end{equation*}
  which describes the plastic strain of the network $(\Z^d,\mathsf E)$ attached to $x$.
  Likewise, one can show that there exists a sublinearly growing displacement field $\varphi^{\omega,x}:\Z^d\to\R^d$ with $\varphi(0)=0$ such that
  \begin{equation*}
    \nabla_s\varphi^{\omega,x}(z)=\chi_s(\tau_{z}\omega,x).
  \end{equation*}
  The field $\varphi^{\omega,x}$ describes fluctuations in the displacement of the lattice $(\Z^d,\mathsf E)$, which is attached to the macroscopic material point $x$ and macroscopically deformed by the linear displacement $z\mapsto u(x)+\nabla u(x)(z-x)$.
  Using the ergodic theorem, we have
  \begin{multline}\label{eq:ergodicity_exploit}
    \mathbb E\bigg[A
  \begin{pmatrix}
    \nabla_su(x)+\chi_s(x)\\
    p(x)
  \end{pmatrix}\cdot   \begin{pmatrix}
    \nabla_su(x)+\chi_s(x)\\
    p(x)
  \end{pmatrix}\bigg]\\
    = \lim\limits_{L\to\infty}L^{-d}\sum_{z\in \Z^d\cap[0,L)^d}A(\tau_z\omega)
  \begin{pmatrix}
    \nabla_su(x)+P_s(\nabla\varphi^{\omega,x}(z))\\
    p^{\omega,x}(z)
  \end{pmatrix}\cdot   \begin{pmatrix}
    \nabla_su(x)+P_s(\nabla\varphi^{\omega,x}(z))\\
    p^{\omega,x}(z)
  \end{pmatrix},
  \end{multline}
  and thus the stored elastic energy at $x\in Q$ is given by a spatial average of the stored elastic energy of the deformed lattice $(\Z^d,\mathsf E)$ attached to $x$.

\subsection{Convergence of the spring network to the homogenized system}
In this section we recall the convergence result for $(Y\e,\mathcal E\e,\rcal\e)\to(Y_{\hom},\mathcal E_{\hom},\rcal_{\hom})$ proved in \cite{neukamm2017stochastic}. It uses the notion of stochastic two-scale
convergence in the mean introduced in \cite{bourgeat1994stochastic}.
Since the limit additionally invokes a transition from a discrete to a continuous model
we shall use the following variant that can handle such transitions. It also first
appeared in \cite{neukamm2017stochastic}, together with an equivalent characterization via stochastic unfolding.
\begin{definition}[Two-scale convergence in the mean]\label{D:two-scale}
  Let $Q\subset\R^d$ be an open, bounded Lipschitz domain, and let~$(v\e)$ denote a sequence of
  random fields $v\e:\Omega\times\eps\Z^d\to\R$ depending on a parameter $\eps > 0$.
  We say that $v\e$ weakly two-scale converges to $v\in L^2(\Omega)\otimes L^2(Q)$ in $L^2$
  as $\eps\to 0$, if $(v\e)$ is bounded, i.e.,
  \begin{equation*}
    \limsup\limits_{\eps\to 0}\|v\e\|_{L^2(\Omega)\otimes L^2(\eps\Z^d)}<\infty,
  \end{equation*}
  and
  \begin{equation*}
    \lim_{\eps\to 0}\mathbb E\bigg[\eps^{d}\sum_{x\in\eps\Z^d}v\e(\omega,x)\varphi(\tau_{\frac{x}{\eps}}\omega)\eta(x)\bigg]
    =
    \mathbb E\Big[\int_{Q}v(\omega,x)\varphi(\omega)\eta(x)\,dx\Big]
  \end{equation*}
  for all $\varphi\in L^2(\Omega)$ and $\eta\in C^\infty_c(Q)$. If in addition
  \begin{equation*}
    \lim\limits_{\eps\to 0}\|v\e\|_{L^2(\Omega)\otimes L^2(\eps\Z^d)}=\|v\|_{L^2(\Omega)\otimes L^2(Q)},
  \end{equation*}
  then we say that $(v\e)$ strongly two-scale converges to $v$ in $L^2$. We use the notations $v\e\wtto v$ and $v\e\stto v$ for weak and strong two-scale convergence, respectively.
Moreover, for a sequence $(y\e)=(u\e,p\e)$ in $Y\e$ and $y=(u,p,\chi_{s})\in Y_{\hom}$ we write $y\e\ctto y$ if
\begin{equation*}
  u\e\stto u,\qquad \nabla_su\e\stto \nabla_su+\chi_{s},\qquad p\e\stto p.
\end{equation*}
\end{definition}

We can now state that the random heterogeneous rate-independent network converges to the homogenized system
$(Y_\textup{hom}, \mathcal{E}_\textup{hom}, \rcal_\textup{hom})$ defined in
Section~\ref{sec:definition_homogenized_system} in this sense.

\begin{theorem}[Evolutionary $\Gamma$-convergence, \mbox{\cite[Theorem~4.10]{neukamm2017stochastic}}]\label{T:hom}
  Let $Q\subset\R^d$ be an open, bounded Lipschitz domain, and assume \ref{B0:pl:d}--\ref{B2:pl:d},
  \ref{item:assumption_measurability}--\ref{item:assumption_ergodicity}, and condition \eqref{ass:smooth_load}
  on the smoothness of the external load.  Consider initial conditions $y\e^0\in Y\e$ that are stable
  in the sense of Theorem~\ref{T:energetic_solution} for the ERIS $(Y\e,\mathcal E\e,\rcal\e)$, and that satisfy \eqref{eq:technical_condition_p0}. Assume that
  \begin{equation*}
    y\e^0\ctto y^0_{\hom}\qquad\text{and}\qquad l\e\stto l_{\hom}
  \end{equation*}
  for some $y^0_{\hom}\in Y_{\hom}$ and $l_{\hom}\in W^{1,1}([0,T],L^2(Q)^d)$. Then $y^0_{\hom}$
  is a stable initial state of the ERIS $(Y_{\hom},\mathcal E_{\hom},\rcal_{\hom})$
  and it holds
  \begin{equation*}
    y\e(t)\ctto y_{\hom}(t)\qquad\text{for all }t\in[0,T],
  \end{equation*}
  where $y\e\in W^{1,1}((0,T);Y\e)$ is the unique energetic solution to the ERIS $(Y\e,\mathcal E\e,\rcal\e)$ with $y\e(0)=y\e^0$, and $y_{\hom}\in W^{1,1}((0,T);Y_{\hom})$ is the unique energetic solution to the ERIS $(Y_{\hom},\mathcal E_{\hom},\rcal_{\hom})$ with $y_{\hom}(0)=y_{\hom}^0$.
\end{theorem}

\subsection{Representation as generalized Prandtl--Ishlinskii models}\label{S:hyst}

We show that the homogenized ERIS can reformulated in the form of a force balance equations that only invokes the displacement $u$, which is a form that is more standard in the engineering literature. However, the elimination of the additional ``internal'' variables $(p,\chi_{s})$ leads to a complex, hysteretic stress-strain relation.
We show that the latter is described by a generalized Prandtl--Ishlinskii operator $\mathcal W_{\hom}$ that maps the history of the strain $\operatorname{sym}\nabla u\vert_{[0,t)}$ to the stress-tensor at time $t$. We refer to \cite{brokate1996hysteresis} for a standard reference on the mathematical theory of hysteresis operators, and remark that the classical Prandtl--Ishlinskii hysteresis operators denote a class of hysteresis operators obtained by a weighted average of elementary play operators---a construction that is widely used in rheological modeling. We shall see that in our situation the emerging hysteresis operator is a homogenization average of the zero-dimensional hysteresis operators describing the behavior of the individual springs.
Moreover, in contrast to rheological modeling, where the choice of the weights is modeled phenomenologically, in our setting the averaging weights are a result of homogenization.

More precisely, let $y_{\hom}=(u_{\hom},p_{\hom},\chi_{s,\hom})$ denote the energetic solution of the homogenized ERIS $(Y_{\hom},\mathcal E_{\hom},\rcal_{\hom})$ where for simplicity we assume throughout this section that the initial condition and the initial loading are trivial, i.e.,
\begin{equation}\label{eq:trivial:initial}
  y_{\hom}^0=0
  \qquad \text{and} \qquad
  l_{\hom}(0)=0.
\end{equation}
We shall see that $u_{\hom}$ is the unique solution in $W^{1,1}((0,T);H^1_0(Q))$ with initial value condition $u_{\hom}(t,\cdot)=0$ of the force balance equation
\begin{equation}\label{eq:hysteresis}
  -\nabla\cdot \sigma_{\hom}(t,\cdot)=l(t,\cdot)\qquad\text{in a distributional sense in $Q$ for all $t\in(0,T]$},
\end{equation}
where the stress tensor $\sigma_{\hom}$ is obtained
from $\operatorname{sym}\nabla u_{\hom}$ with help of hysteresis operator $\mathcal W_{\hom}$.

\begin{theorem}[Generalized Prandtl--Ishlinskii model]\label{T:hyst_hom}
  Consider the situation of Theorem~\ref{T:hom} and additionally assume the initial condition \eqref{eq:trivial:initial}. Then the $u_{\hom}$-component of the energetic solution satisfies for all $t\in[0,T]$ the force balance equation \eqref{eq:hysteresis} with $\sigma_{\hom}$ given by the stress--strain relation
  \begin{equation*}
    \sigma_{\hom}(t,x)=\mathcal W_{\hom}\big[\mathrm{sym}\nabla u_{\hom}(\cdot,x)\big](t)
  \end{equation*}
  where
  \begin{equation*}
    \mathcal W_{\hom}: W^{1,1}_{o}\to W^{1,1}_{o},
    \qquad
    W^{1,1}_{o}\colonequals \Big\{F\in W^{1,1}\big((0,T);\R^{d\times d}_{\rm sym}\big)\,:\,F(0)=0\Big\}
  \end{equation*}
  denotes the generalized Prandtl--Ishlinskii hysteresis operator introduced in Definition~\ref{D:Whom} below.
\end{theorem}

The definition of the hysteresis operator $\mathcal W_{\hom}$ relies on an ERIS $(\widehat Y_\hom,\widehat{\mathcal E}_\hom(\cdot;F),\widehat{\rcal}_\hom)$ with state space
\begin{equation*}
\widehat Y_\hom \colonequals L^2(\Omega)^k\times L^2_s(\Omega),
\end{equation*}
and an energy functional $\widehat{\mathcal E}_\hom(\cdot;F)$ that invokes a loading term
 $\ell_F:[0,T]\to {\widehat Y_\hom}^*$,
\begin{equation*}
  \expect{\ell_F(t),\hat y} \colonequals   \expect{\ell_F(t),\binom{\hat p}{\hat \chi_s}}\colonequals -\mathbb E\Big[A(\omega)
  \begin{pmatrix}
    P_s(F(t))\\0
  \end{pmatrix}\cdot\binom{\hat\chi_s}{\hat p}\Big]\qquad\text{for all }t\in[0,T].
\end{equation*}
It is parameterized by a strain trajectory $F\in W^{1,1}_o$ that plays the role of the input of the hysteresis operator. The energy functional $\widehat{\E}_\hom(\cdot;F):[0,T]\times \widehat Y)_\hom\to\R$ and the dissipation functional $\widehat{\rcal}_\hom:\widehat Y\to[0,\infty]$ are
\begin{equation}\label{def:hat_ERIS}
  \begin{aligned}
    \widehat{\mathcal E}_\hom(t,\hat y;F)
    &\colonequals
    \widehat{\mathcal E}_\hom(t,\hat p,\hat\chi_s;F)
    \colonequals
    \frac12\mathbb E\Big[A(\omega)
    \begin{pmatrix}
      \hat\chi_s\\\hat p
    \end{pmatrix}\cdot   \begin{pmatrix}
      \hat\chi_s\\\hat p
    \end{pmatrix}\Big]-\expect{\ell_F(t),\hat y}\\
    \widehat{\rcal}_\hom(\dot{\hat y})&\colonequals\widehat{\rcal}_\hom(\dot{\hat p},\dot{\hat\chi}_s):=\mathbb E[\rho(\omega,\dot{\hat p})].
  \end{aligned}
\end{equation}
The system $(\widehat Y_\hom,\widehat{\mathcal E}_\hom(\cdot;F),\widehat{\rcal}_\hom)$ is a quadratic ERIS in the sense of Definition~\ref{D:ERIS}. In particular, by Theorem~\ref{T:energetic_solution}, for any $F\in W^{1,1}_0$ there exists a unique energetic solution in $\hat y_\hom\in W^{1,1}((0,T);\widehat Y)$ with $\hat y_\hom(0)=0$. We conclude that the map
\begin{equation*}
  [0,T]\ni t\mapsto \mathbb E\Big[A
  \begin{pmatrix}
    P_s(F(t))+\hat\chi_{s,\hom}(t)\\\hat p_\hom(t)
  \end{pmatrix}\Big]\in\R^k
\end{equation*}
is of class $W^{1,1}((0,T);\R^{2k})$. Moreover, it vanishes for $t=0$, since $P_s(F(0))=\hat\chi_{s,\hom}(0)=\hat p_\hom(0)=0$ by assumption.
We summarize this construction in the following definition.

\begin{definition}[The hysteresis operator $\mathcal W_{\hom}$]\label{D:Whom}
  We define
  \begin{equation*}
    \mathcal W_{\hom}: W^{1,1}_o\to W^{1,1}_o
  \end{equation*}
  as follows:
  For an input function $F\in W^{1,1}_o$ let $\hat y_\hom=(\hat p_\hom,\hat\chi_{s,\hom})\in W^{1,1}((0,T);\R^{d\times d}_{\rm sym})$ denote the unique energetic solution to $(\widehat Y_\hom,\widehat{\mathcal E}_\hom(\cdot;F),\widehat{\rcal}_\hom)$ with initial condition $\hat y_\hom(0)=0$. Define for all $t\in[0,T]$,
  \begin{equation*}
    \mathcal W_{\hom}[F](t) \colonequals (P_{s}^*\circ\pi_k)\Bigg(\mathbb E\Big[A
    \begin{pmatrix}
      P_s(F(t))+      \hat\chi_{s,\hom}(t)\\\hat p_\hom(t)
    \end{pmatrix}\Big]\Bigg),
  \end{equation*}
  where $\pi_k:\R^{2k}\to\R^k$ denotes the projection of a vector $(a_1,\ldots, a_{2k})$ to its first $k$ coordinates $(a_1,\ldots, a_k)$, and $P_s^*$ denotes the operator that maps strain vectors to strain matrices defined in Remark~\ref{R:strainvector}.
\end{definition}


\section{Representative volume element approximation}
As explained in Section~\ref{S:two-scale_system} the homogenized ERIS $(Y_{\hom},\E_{\hom},\rcal_{\hom})$
is a two-scale system that involves, besides the macroscopic displacement $u_{\hom}$ the
additional degrees of freedom $(p_{\hom},\chi_{s,\hom})$ that describe the evolution
of the microscopic plastic strain and microscopic displacement.
For each material point $x\in Q$ and time $t$, these quantities are described
by random vectors defined on the probability space $(\Omega,\mathcal F,P)$,
which typically is a product space with an infinite number of factors (see, e.g., Example~\ref{EX2}).

In view of this, finite element (FE) approximations of the time incremental problems
need to discretize both the macroscopic displacement space $H^1(Q)^d$ and the random
space $L^2(\Omega)$, where $\Omega$ denotes the probability space and practical examples is often given as the $\Z^d$-fold product of a simple probability space, cf.~Example~\ref{E:iid}.
For the latter, we introduce a representative volume element (RVE) approximation
for the homogenized system. It exploits ergodicity in the sense that expectations of stationary
random fields are approximated by means of spatial averages on large but finite domains.
The approximation takes itself the form of a quadratic ERIS $(Y_L^\omega,\mathcal E_L^\omega,\rcal_L^\omega)$ parameterized by a parameter $L$ that encodes the size of the averaging domain.
Note that the ERIS is a \emph{quenched} model: The coefficients of the ERIS are also
parameterized by a configuration $\omega\in\Omega$ that is sampled according to $P$. Likewise, we introduce an RVE approximation for the generalized Prandtl--Ishlinskii operator $\mathcal W_{\hom}$. The main result of this section is
the proof that both approximations converge in $L^2(\Omega)$ as $L\to\infty$.

The above described procedure is similar to the $\textup{FE}^2$-method from numerical homogenization, see, e.g.,  \cite{feyel1999multiscale,miehe2002strain,schroder2013plasticity,nejad2019parallel,eigel2015convergent}.
There, a macroscopic discretization of the physical domain $Q$ is considered and to each
quadrature point in it,
another (discretized) representative volume domain is assigned, which serves for the computation of the averaged stresses.

\subsection{Motivation and definition of  the RVE problem \texorpdfstring{$(Y_L^\omega,\mathcal E_L^\omega,\rcal_L^\omega)$}{(YL,EL,RL)}}\label{S:RVE_def}

Reconsider the situation described in Section~\ref{S:two-scale_system}.
The numerical evaluation of the expected value~\eqref{eq:ergodicity_exploit}
(which in particular requires to compute $\chi_s$ first)
involves summation over an infinite lattice, which cannot be performed computationally.

An approximation of the limit in~\eqref{eq:ergodicity_exploit} is obtained by
setting the domain size $L \in \mathbb{N}$ to a large but fixed number. In other words,
we approximate the expectation by the spatial average over a \emph{large, but finite} volume
\begin{equation*}
  \Lambda_L\colonequals \Z^d\cap[0,L)^d\qquad \text{ with }L\in\N\text{ large},
\end{equation*}
and by replacing the fields $p^{\omega,x}$ and $\varphi^{\omega,x}$ (which in \eqref{eq:ergodicity_exploit} are functions defined on the infinite domain $\Z^d$) by approximations  $p_L^{\omega,x}$ and $\varphi^{\omega,x}_L$ that are defined on the finite volume $\Lambda_L$ and that can be computed pathwise, i.e., for fixed samples $\omega\in\Omega$.

To implement this idea, one needs to modify the state space of the system: Morally speaking,
$L^2(\Omega)$ and $L^2_s(\Omega)$ are replaced by gradient fields of functions defined on $\Lambda_L$.
Moreover, in order to get a \emph{coercive} energy functional defined on the modified state space, we have
restrict the space of functions on $\Lambda_L$. In particular, it is necessary to introduce boundary conditions for  $\varphi^{\omega,x}_L$. These conditions need to be consistent in the sense that $\nabla\varphi^{\omega,x}_L$ is mean free, i.e., $\sum_{z\in\Lambda_L}\nabla\varphi_L^{\omega,x}(z)=0$, which corresponds to the Hill--Mandel condition in computational homogenization and mechanically means that the microscopic fluctuation of the strain  described by $\varphi_L^{\omega,x}$ does not contribute to the macroscopic (averaged) strain.

In this paper we use \emph{periodic} boundary conditions:
For $L\in\N$ we define the following spaces of $L$-periodic discrete functions,
\begin{align*}
  L^2(\Lambda_L)&\colonequals \bigg\{f:\Z^d\to\R\,:\,f(\cdot+k)=f\text{ for all $k\in L\Z^d$}\,\bigg\},\\
  L^2_{\mathrm{av}}(\Lambda_L)&\colonequals \bigg\{f\in L^2(\Lambda_L)\,:\,\sum_{z\in\Lambda_L}f(z)=0\bigg\},
\end{align*}
which we equip with the norms
\begin{equation*}
  \|f\|_{L^2(\Lambda_L)}\colonequals \big(L^{-d}\sum_{z\in\Lambda_L}|f(z)|^2\big)^\frac12,\qquad \|f\|_{L^2_{\mathrm{av}}(\Lambda_L)}:=\frac1L\big(L^{-d}\sum_{z\in\Lambda_L}|f(z)|^2\big)^\frac12,
\end{equation*}
that turn these spaces into Hilbert spaces.
Thanks to the mean free-condition and the scaling incorporated in the definition of $(L^2_{\mathrm{av}}(\Lambda_L),\|\cdot\|_{L^2_{\mathrm{av}}(\Lambda_L)})$ the following Poincaré inequality
\begin{equation*}
  \|f\|_{L^2_{\mathrm{av}}(\Lambda_L)}\leq \|\nabla f\|_{L^2(\Lambda_L)}
\end{equation*}
holds for all $f\in L^2_{\mathrm{av}}(\Lambda_L)$.
Furthermore, we have the Korn inequality
\begin{equation*}
  \|\varphi\|_{L^2_{\mathrm{av}}(\Lambda_L)}\leq c\|\nabla_s \varphi\|_{L^2(\Lambda_L)}
\end{equation*}
for all periodic, mean-free displacements $\varphi\in L^2_{\mathrm{av}}(\Lambda_L)^d$ with a constant $c$ that is independent of $L$.

The state space for the quenched version of the RVE system $(Y^\omega_L,\E^\omega_L,\rcal^\omega_L)$
is now obtained by replacing $L^2(\Omega)$ in definition of $Y_{\hom}$, cf.~\eqref{D:Yhom}, by $L^2(\Lambda_L)$:
\begin{equation*}
  Y^\omega_L\colonequals H^1_0(Q)^d\times \big(L^2(\Lambda_L)\otimes L^2(Q)\big)^k\times \big(L^2_{\mathrm{av}}(\Lambda_L)\times L^2(Q)\Big)^d.
\end{equation*}
(Note that $Y^\omega_L$ is independent of $\omega$. Nevertheless, we write the superscript $\omega$
because $Y_L^\omega$ is the state space for the quenched model). In the following we write
$y_L=(u_L,p_L,\varphi_L)$ for $y_L\in Y^\omega_L$.

Let $\omega$ be an element of the probability space $\Omega$. For each such $\omega$
we define the energy functional
\begin{multline*}
  \E_{L}^\omega(t,y_L)\\
  \colonequals
   \int_{Q}\bigg(\frac1{L^{d}}\sum_{z\in\Lambda_L}\frac12 A(\tau_z\omega)\binom{P_s\nabla u_L(x)+\nabla_{s,z}\varphi_L(z,x)}{p_L(z,x)}\cdot \binom{P_s\nabla u_L(x)+\nabla_{s,z}\varphi_L(z,x)}{p_L(z,x)}\bigg)\,dx\\
   -\int_Ql(t,x)\cdot u_L(x)\,dx,
\end{multline*}
where use the notation $\nabla_{s,z}$ to highlight that $\nabla_s$ acts on the $z$-variable.
We define the dissipation functional
\begin{equation*}
\rcal_{L}^\omega(\dot y_L)
\colonequals
\int_{Q}\Big(\frac1{L^{d}}\sum_{z\in\Lambda_L}\rho(\tau_{z}\omega,\dot{p}_L(z,x))\Big)\,dx.
\end{equation*}

\begin{remark}[Existence and uniqueness for the quenched RVE model]\label{R:RVE}
  Under assumptions \ref{B0:pl:d}--\ref{B2:pl:d}, \ref{item:assumption_measurability}--\ref{item:assumption_ergodicity}, and for $l\in W^{1,1}((0,T);L^2(Q)^d)$, the system $(Y_L^\omega,\mathcal E_L^\omega,\rcal_L^\omega)$ is for $P$-a.e.~$\omega\in\Omega$ a quadratic ERIS in the sense of Definition~\ref{D:ERIS} and we may appeal to Theorem~\ref{T:energetic_solution} for the existence and uniqueness of an energetic solution to the initial value problem. We note that for fixed $\omega$, the ERIS is deterministic in the sense that it invokes only functions defined on the bounded domain $Q\times\Lambda_L\subset\R^d\times\Z^d$. Furthermore, the quadratic part of $\E^\omega_L$ is coercive uniformly in $L$, i.e.,
  \begin{equation*}
    \E_{L}^\omega(t,y_L)+\int_Ql(t)\cdot u_L\,dx\geq C\|y_L\|^2_{Y^\omega_L}.
  \end{equation*}
  where $C>0$ is independent of $L$, $t$ and $\omega$.
\end{remark}

\subsection{Convergence of the RVE approximation}\label{S:RVE_conv}
The convergence for $L\to\infty$ of the RVE model described as the ERIS $(Y^\omega_L,\E^\omega_L,\rcal^\omega_L)$ holds in the mean in the strong $L^2(\Omega)$-topology. For the precise statement, we introduce a mean version of. Morally speaking, it is obtained by taking the expectation of the quenched RVE ERIS with respect to the probability measure $P$. To this end we introduce the state space
\begin{equation*}
Y_L
\colonequals
\brac{L^2(\Omega)\otimes H^1_0(Q)}^d\times \brac{L^2(\Omega)\otimes L^2(\Lambda_L) \otimes L^2(Q)}^k \times \brac{L^2(\Omega) \otimes L^2(\Lambda_L)\otimes L^2(Q)}^d,
\end{equation*}
(which is equivalent to $L^2(\Omega)\otimes Y^\omega_L$), and the energy functional and dissipation functional defined by integrating the quenched functionals, i.e., 
\begin{equation*}
  \E_{L}(t,y_L)\colonequals \mathbb E\Big[\E_L^\omega(t,y_L(\omega,\cdot))\Big],
  \qquad
  \rcal_{L}(\dot{y}_L) \colonequals \mathbb E\Big[\rcal_{L}^\omega(\dot{y}_L(\omega,\cdot))\Big].
\end{equation*}

  Under assumptions \ref{B0:pl:d}--\ref{B2:pl:d}, \ref{item:assumption_measurability}--\ref{item:assumption_ergodicity}, and for $l\in W^{1,1}((0,T);L^2(Q)^d)$, the system $(Y_L,\mathcal E_L,\rcal_L)$ is a quadratic ERIS in the sense of Definition~\ref{D:ERIS} and we may appeal to Theorem~\ref{T:energetic_solution} for the existence and uniqueness of an energetic solution to the initial value problem with a stable initial value $y_L^0\in Y_L$.

\begin{remark}[Relation to the quenched model]\label{R:reltoquenched}
  We may view the mean model $(Y_L,\mathcal E_L,\rcal_L)$ as an integrated version of the
  quenched model $(Y^\omega_L,\mathcal E_L^\omega,\rcal_L^\omega)$: Indeed, by definition
  we may identify $Y_L$ with $L^2(\Omega)\otimes Y^\omega_L$.
  Thus, for $P$-a.e.~$\omega\in\Omega$, we have $y^0_L(\omega,\cdot)\in Y^\omega$. Assume that $y^0_L(\omega,\cdot)\in S^\omega_L(0)$ for $P$-a.e.~$\omega\in\Omega$ and denote by $y^\omega_L$ the unique energetic solution for the quenched ERIS $(Y^\omega_L,\E^\omega_L,\rcal^\omega_L)$ with initial condition $y^\omega_L(0)=y^0_L(\omega,\cdot)$. Then, as shown in \cite[Lemma 7.16]{varga2019stochastic} it turns out that
  \begin{equation*}
    y^\omega_L(t)=y_L(\omega,t)\qquad\text{for a.a.~$t\in[0,T]$ and $P$-a.e. $\omega\in\Omega$.}
  \end{equation*}
\end{remark}


A central result of this paper is the proof of convergence of the mean RVE ERIS to the homogenized ERIS as $L\to\infty$. Since the state spaces $Y_L$ and $Y_{\hom}$ are different, we first need to specify a suitable notion of convergence. To that end we introduce the linear transformation
\begin{equation}\label{eq:TL}
  \begin{aligned}
    &\mathcal T_L:L^2(\Omega)\otimes L^2(\Lambda_L)\otimes L^2(Q)\to L^2(\Omega)\otimes L^2(\Lambda)\otimes L^2(Q)\\
    &\qquad (\mathcal T_L f)(\omega,z,x)\colonequals f(\tau_{-\lfloor Lz\rfloor}\omega,\lfloor Lz\rfloor,x),
  \end{aligned}
  \end{equation}
  where $\Lambda\colonequals [0,1)^d$ and $\lfloor Lz\rfloor=\hat z\in\Z^d$ if and only if $Lz\in\hat z+\Lambda$.
  It invokes a scaling and a piecewise-constant interpolation with respect to the $z$-variable in order to map the discrete space $L^2(\Lambda_L)$ to the continuum space $L^2(\Lambda)$. Moreover, the definition is closely related to the discrete stochastic unfolding operator introduced in \cite{neukamm2017stochastic} (see Proof of Lemma~\ref{L:unfh} for details). As we shall prove in Lemma~\ref{L:unfh}, thanks to assumption (S3), the operator $\mathcal T_L$ is a linear, (non-surjective) isometry. In particular, for all $f\in L^2(\Omega)\otimes L^2(\Lambda_L)\otimes L^2(Q)$ we have
  \begin{equation*}
    \mathbb E\bigg[L^{-d}\sum_{z\in\Lambda_L}\int_Q|f(\omega,z,x)|^2\,dx\bigg]
    =
    \mathbb E\Big[\int_{\Lambda}\int_Q|\mathcal T_Lf(\omega,z,x)|^2\,dx\,dz\Big].
  \end{equation*}
  Moreover, a sequence $(f_L)\subset L^2(\Omega)\otimes L^2(\Lambda_L)\otimes L^2(Q)$ whose image $(\mathcal T_Lf_L)$ is a Cauchy sequence, strongly two-scale converges in the mean in the sense of Definition~\ref{D:two-scale}. In the following we tacitly view $L^2(Q)$ and $L^2(\Omega)\otimes L^2(Q)$ as subspaces of $L^2(\Omega)\otimes L^2(\Lambda)\otimes L^2(Q)$.
\begin{theorem}[Strong two-scale convergence of the approximation]\label{thm:1462:d}
  Assume \ref{B0:pl:d}--\ref{B2:pl:d}, \ref{item:assumption_measurability}--\ref{item:assumption_ergodicity}, and  $l\in W^{1,1}((0,T);L^2(Q)^d)$.
  Set $\|\cdot\|\colonequals\|\cdot\|_{L^2(\Omega)\otimes L^2(\Lambda)\otimes L^2(Q)}$.
  Let $y_{L}^0\in Y_{L}(0)$ denote a sequence of stable initial states, and assume convergence of the initial values in the sense that for $L\to\infty$ we have
  \begin{align*}
    & \|u_L^0 - u^0_{\hom}\|_{L^2(\Omega)\otimes H^1_0(Q)}+\|\mathcal T_Lp_L^0-p^0_{\hom}\|+\|\mathcal T_L\nabla_{s,z}\varphi^0_L-\chi_{s,\hom}^0\|+L^{-1}\|\mathcal T_L\varphi_L^0\|\to 0
  \end{align*}
  for some $y^0_{\hom}=(u^0_\hom,p^0_\hom,\chi^0_{s,\hom})\in Y_{\hom}$. 
  Then $y^0_{\hom}\in S_{\hom}(0)$ and the unique energetic solutions $y_{L}$ and $y_{\hom}$ of $\brac{Y_{L},\mathcal{E}_{L},\rcal_{L}}$ and $(Y_{\hom},\E_{\hom},\rcal_{\hom})$ with $y_{L}(0)=y_{L}^0$ and $y_{\hom}(0)=y^0_{\hom}$ satisfy for a.a.~$t\in [0,T]$,
  \begin{equation*}
    \|u_L(t) - u_{\hom}(t)\|_{L^2(\Omega)\otimes H^1_0(Q)}+\|\mathcal T_Lp_L(t)-p_{\hom}(t)\|+\|\mathcal T_L\nabla_{s,z}\varphi_L(t)-\chi_{s,\hom}(t)\|+L^{-1}\|\mathcal T_L\varphi_L(t)\|\to 0
  \end{equation*}
  and the energy functionals converge,
  \begin{equation*}
    \E_{L}(t,y_L(t))\to \E_{\hom}(t,y_{\hom}(t)).
  \end{equation*}
\end{theorem}
The proof is given in Section~\ref{sec:proofs}.

\subsection{The RVE problem as a Prandtl--Ishlinskii model}\label{S:RVE-PI}

We reformulate the ERIS $(Y^\omega_L,\E^\omega_L,\rcal^\omega_L)$ as a generalized Prandtl--Ishlinskii model. This can be done analogously to Section~\ref{S:hyst}. To that end we first define a generalized Prandtl--Ishlinskii hysteresis operator
  \begin{equation*}
    \mathcal W_{L}^\omega: W^{1,1}_{o}\to W^{1,1}_{o},\qquad W^{1,1}_{o}\colonequals \{F\in W^{1,1}((0,T);\R^{d\times d}_{\rm sym})\,:\,F(0)=0\}
  \end{equation*}
  as follows: For a given input $F\in W^{1,1}_o$ we consider the quadratic ERIS  $(\widehat Y_L^\omega,\widehat{\mathcal E}_L^\omega(\cdot;F),\widehat{\rcal}_L^\omega)$ with
\begin{itemize}
\item state space $\widehat{Y}^\omega_L =  L^2(\Lambda_L)^k \times L^2_{\mathrm{av}}(\Lambda_L)^d$, and notation $y_L=(p_L,\varphi_L)$,
\item energy functional $\widehat{\E}_L^\omega(\cdot;F): [0,T] \times \widehat{Y}_L^\omega \to \R$
\begin{equation*}
  \widehat{\E}^\omega_L(t, y_L;F) \colonequals  \frac{1}{L^{d}}\sum_{z\in\Lambda_L}\Big(\frac12 A(\tau_z \omega)\binom{\nabla_s \varphi_L}{p_L}\cdot \binom{\nabla_s\varphi_L}{p_L}-A(\tau_z \omega)\binom{\nabla_s \varphi_L}{p_L}\cdot \binom{P_S(F(t))}{0}\Big),
\end{equation*}
\item dissipation functional $\widehat{\rcal}^\omega_L: \widehat{Y}_L^\omega\to [0,\infty]$
\begin{equation*}
  \widehat{\rcal}^\omega_L(\dot y_L) \colonequals \tfrac{1}{L^d}\sum_{z\in\Lambda_L}\rho(\tau_z \omega,\dot{p}_L(z)),
\end{equation*}
\end{itemize}
and we define for a.a.~$t\in[0,T]$
\begin{equation}
\label{eq:rve_pi_operator}
  \mathcal W_{L}^\omega[F](t)\colonequals (P_{s}^*\circ\pi_k)\Bigg(\frac{1}{L^d}\sum_{z\in\Lambda_L}A(\tau_z\omega)
  \binom{P_s(F(t))+\nabla_s\hat\varphi_L(t,z)}{\hat p_L(t,z)}\Bigg),
\end{equation}
where $\hat y_L^\omega=(\hat\varphi_L,\hat p_L)\in W^{1,1}((0,T);\widehat Y_L^\omega)$ denotes the energetic solution of $(\widehat Y_L^\omega,\widehat{\mathcal E}_L^\omega(\cdot;F),\widehat{\rcal}_L^\omega)$ with $\hat y^\omega_L(0)=0$,
$\pi_k:\R^{2k}\to\R^k$ denotes the projection of a vector $(a_1,\ldots, a_{2k})$ to its first $k$-coordinates $(a_1,\ldots, a_k)$, and $P_s^*$ denotes the inverse operator that maps a strain matrix to a strain vector.

\begin{lemma}[RVE approximation as generalized Prandtl--Ishlinskii model]\label{L:hyst_hom}
  Consider assumptions \ref{B0:pl:d}--\ref{B2:pl:d}, \ref{item:assumption_measurability}--\ref{item:assumption_ergodicity}, and $l\in W^{1,1}((0,T);L^2(Q)^d)$. Additionally assume that initial state is trivial,
  \begin{equation*}
    y_{L}^0=0
    \qquad \text{and} \qquad
    l(0,\cdot)=0.
  \end{equation*}
  Then the $u_{L}$-component of the energetic solution $y_L=(u_L,p_L,\varphi_L)$ of the ERIS $(Y^\omega_L,\E^\omega_L,\rcal^\omega_L)$ satisfies for a.a.~$t\in[0,T]$ the constitutive equation
  \begin{equation*}
    -\nabla\cdot\sigma_{L}^\omega(t,x)=l(t,x)\qquad\text{in a distributional sense in }Q,
  \end{equation*}
  where the stress tensor $\sigma_{L}^\omega:[0,T]\times Q\to\R^{d\times d}_{\rm sym}$ is given by the relation
  \begin{equation*}
    \sigma_{L}^\omega(t,x)=\mathcal W_{L}^\omega\big[\mathrm{sym}\nabla u_{L}(\cdot,x)\big](t).
  \end{equation*}
\end{lemma}
The proof is analogous to the proof of Theorem~\ref{T:hyst_hom} and thus omitted here.

The next theorem proves convergence of the approximate operator $\mathcal W_L^\omega$:

\begin{theorem}\label{T:RVE-W}
  Assume \ref{B0:pl:d}--\ref{B2:pl:d}, \ref{item:assumption_measurability}--\ref{item:assumption_ergodicity}. Then for any trajectory $F\in W^{1,1}_o$ and a.a.~$t\in[0,T]$ we have
  \begin{equation*}
    \lim\limits_{L\to\infty}\mathbb E\big[|\mathcal W^\omega_L[F](t)-\mathcal W_{\hom}[F](t)|^2\big]=0.
  \end{equation*}
\end{theorem}
The proof is given in Section~\ref{sec:proofs}.


\section{Computing numerical approximations}
\label{sec:numerical_approach}

The homogenized stress tensor $\sigma_{\hom}$, related to the hysteresis operator $\mathcal{W}_{\hom}$ by
\begin{equation*}
  \sigma_{\hom}(t)\colonequals \mathcal W_{\hom}[F](t),
  \qquad
  t\in[0,T],
  \quad
  F \in W^{1,1}_o,
\end{equation*}
describes the stress--strain relation of the homogenized system, and is thus
a central part of the mechanical model.
Since the underlying random heterogeneous network is statistically homogeneous,
the hysteresis operator $\mathcal W_{\hom}$
is the same for all macroscopic material points $x\in Q$.

In this section we introduce a numerical scheme for approximating the hysteresis operator
$\mathcal W_{\hom}[F]$ for a given trajectory of macroscopic strains $F\in W^{1,1}_o$.
To use the operator at a macroscopic point $x \in Q$ in a $\textup{FE}^2$-method, one would have
to consider the trajectory $F(t) = \operatorname{sym}\nabla u(t,x)$.

\subsection{Monte-Carlo sampling}

To approximate $\mathcal W_{\hom}[F]$ we need to numerically solve the random RVE problem $(\widehat Y_L^\omega,\widehat{\E}_L^\omega,\widehat{\rcal}_L^\omega)$,
and then compute~\eqref{eq:rve_pi_operator}.
The randomness is taken care of by Monte-Carlo sampling.  Select a sample size $M \in \N$,
and a set $\boldsymbol \omega$ of $M$ realizations $\omega_1,\ldots,\omega_M$ of the random material parameters.  These are sampled
with respect to the $M$-fold product of $P$, that is, $\omega_1,\ldots,\omega_M$ are independently sampled with $P$.
Then we define the approximate stress tensor
\begin{equation*}
  \sigma_{L}^{M,{\boldsymbol\omega}}(t)
  \colonequals
  \mathcal W_L^{M,\boldsymbol\omega}[F](t)
  \colonequals
  \frac1M\sum_{i=1}^M\mathcal W^{\omega_i}_L[F](t),\qquad t\in[0,T],
\end{equation*}
where $\boldsymbol{\omega}\colonequals(\omega_1,\ldots,\omega_M)$.
We note that $\sigma_{L}^{M,{\boldsymbol\omega}}$ is a time-dependent  random matrix
in $\R^{d\times d}_{\rm sym}$. If we use $\mathbb{E}_M$ to denote the expectation
with respect to the $M$-fold product of $P$, then Theorem~\ref{T:RVE-W} yields
\begin{equation*}
  \lim_{L\to\infty}\mathbb E_M\Big[\abs[\big]{\sigma_{L}^{M,{\boldsymbol\omega}}(t)-\sigma_{\hom}}^2\Big]
  =0
  \qquad\text{for all $M\in\N$ and a.a.\ $t\in[0,T]$}.
\end{equation*}
We remark that we have convergence already for $M=1$. However, as we shall see, the variance
of $\sigma_L^{M, \boldsymbol \omega}$ can be reduced by increasing $M$.

\subsection{Simulating realizations of the RVE problem}
\label{sec:simulating_rve_problem}

For each sample of the Monte-Carlo method, we need to approximate the stress tensor
\begin{equation*}
  \sigma_L^\omega(t)\colonequals \mathcal W^{\omega}_L[F](t),\qquad t\in[0,T]
\end{equation*}
for a fixed $\omega\in\Omega$ sampled with $P$.  We do this by numerical simulation
of the ERIS $(\widehat Y_L^\omega,\widehat{\E}_L^\omega(\cdot;F),\widehat{\rcal}_L^\omega)$
of Section~\ref{S:RVE-PI}.

As the first step we introduce a time discretization:
For $N\in\N$ we split the time interval $[0,T]$ into $N$ equal subintervals defined by the time steps $0=t_0 < t_1 < \cdots < t_N=T$.
Following the standard theory of ERIS~\cite{mielke2015rate}, we can then approximate
the evolution $\hat{y} : [0,T] \to \widehat{Y}_L^\omega$ by a sequence of minimization problems
\begin{equation}\label{eq:discretemin}
  \hat y^l \colonequals \argmin_{z \in \widehat Y_L^\omega} \mathcal J^l(z).
\end{equation}
The increment functional here is
\begin{align}\label{eq:discreteL}
  \mathcal{J}^l(z)
  \colonequals
  \widehat{\mathcal{E}}_L^\omega(t_l, z;F)+\widehat{\mathcal{R}}_L^\omega(z-\hat y^{l-1}),
\end{align}
energy $\widehat{\E}_L^\omega(\cdot;F)$ and dissipation $\widehat{\rcal}_L^\omega$ are defined in Section~\ref{S:RVE-PI},
and we start from the trivial initial value $\hat y(t_0) = 0$.
We associate with $\hat y^0,\ldots, \hat y^N$ its piecewise affine interpolation $\hat y_N:[0,T]\to Y_L^\omega$ and note that $\hat y_N$ is Lipschitz continuous in time. As shown in \cite{mielke2005evolution} it strongly converges to the energetic solution to $(\widehat Y_L^\omega,\widehat{\E}_L^\omega(\cdot;F),\widehat{\rcal}_L^\omega)$ as $N\to\infty$.

In view of the definition of $\mathcal W^\omega_L[F]$ given in~\eqref{eq:rve_pi_operator}, we may now consider the approximate stress tensor
\begin{equation*}
  \sigma_{N,L}^{\omega}(t)\colonequals (P_{s}^{-1}\circ\pi_k)\Bigg(\frac{1}{L^d}\sum_{z\in\Lambda_L}A(\tau_z\omega)
  \binom{P_s(F(t))+\nabla_s \hat\varphi_N(t,z)}{\hat p_N(t,z)}\Bigg),\qquad \hat y_N=(\hat p_N,\hat \varphi_N),
\end{equation*}
and we conclude from the convergence of $\hat y_N$ that $\sigma_{N,L}^{\omega}(t)\to \sigma_L^\omega(t)$ as $N\to\infty$ for $P$-a.e.~$\omega\in\Omega$ and a.a.~$t\in[0,T]$.
In summary we end up with the approximation
\begin{equation}\label{def:sigmaproxy}
  \sigma_{N,L}^{M,\boldsymbol{\omega}}(t)\colonequals \frac1M\sum_{i=1}^M\sigma_{N,L}^{\omega_i}(t),\qquad \lim_{L\to\infty}\lim_{N\to\infty}\sigma_{N,L}^{M,\boldsymbol{\omega}}(t)=\sigma_{\hom}(t)\text{ for all $M\in\N$},
\end{equation}
where the convergence holds in probability for a.a.~$t\in[0,T]$.

It remains to explain how to numerically solve the increment problem~\eqref{eq:discreteL}.
First observe that the domain of $\mathcal J^l$ is finite dimensional. In fact, it is the product
of the two finite-dimensional spaces $L^2(\Lambda_L)^k$ (the space of microscopic plastic strains)
and $L^2(\Lambda_L)^d$ (the space of microscopic displacements).
These spaces have dimensions $n=k\cdot L^d$ and $m=d\cdot L^d$, respectively. We may therefore
identify any $\hat y \in Y^\omega_L$ with a vector $y\in\R^{n+m}$.
The increment functional $\mathcal J^l$ defined in~\eqref{eq:discreteL}
consists of the two parts
\begin{equation*}
  \mathcal E^l(y)
  \colonequals \widehat{\E}^\omega_L(t_l,\hat y;F),
  \qquad \text \qquad
  \mathcal R(y)\colonequals \widehat{\rcal}^\omega_L(\hat y).
\end{equation*}
It is thus the sum of the twice continuously differentiable,
convex functional $\mathcal E^l$, and the convex, lower-semicontinuous, proper functional $\rcal$.
\todosander{Post-Preprint: Ich finde hier ist noch nicht so ganz klar, weshalb
die Funktionale diese Eigenschaften haben. Oder sind sie für das allgemeine Modell
womöglich nur Annahmen?\todoneukamm{Ich habe das nochmal umgeschrieben.}}
Note that the $\hat\varphi$-component of any state  $\hat y=(\hat p,\hat\varphi)\in \widehat Y^\omega_L$ satisfies the condition $\sum_{z\in\Lambda_L}\hat\varphi(z)=0$, which removes constant displacmeent fields from the state space.
Therefore, the functional $\mathcal J^l$ is coercive.
In the implementation we realize this condition by clamping the network at the four corners of the RVE (which is consistent with
the periodic boundary conditions).
Finally, the nonsmooth dissipation functional is block-separable, i.e., it can be written
as a sum
\begin{equation*}
 \rcal(y) = \sum_{i=1}^{L^d} \rcal_i(y_i).
\end{equation*}
Here, $y_i$ denotes the degrees of freedom associated to the $\hat p$-component of the $i$-th node in the
RVE spring network, and $\rcal_i : \R^k \to \R$ is the dissipation of these
degrees of freedom.
\todosander{Post-Preprint: $\rcal_i$ konkret hinschreiben. Hängt es nur von $p$ ab,
oder auch von $\varphi$?}

\begin{remark}
 The mathematical structure of the increment problem is closely related to finite element discretizations
 of continuum small-strain elastoplastic models~\cite{sander2020plasticity}.
 In fact one can interpret the deterministic spring network model as a finite element system defined
 on a one-dimensional network grid using first-order Lagrange finite elements for the micro-displacements
 and piecewise constant functions for the plastic strains.
\end{remark}

We solve the minimization problem~\eqref{eq:discretemin} using the Truncated Nonsmooth Newton Multigrid (TNNMG)
method presented in \cite{graeser2019tnnmg}.
The TNNMG method is a robust and efficient solver for convex block-separable minimization
problems.  It alternates a nonlinear block Gauß--Seidel iteration with a linear correction step.
As the two-dimensional spring networks considered in this work are not very large we use
an exact Newton step computed by a sparse direct solver for the linear corrections.
\citeauthor{graeser2019tnnmg} showed that the TNNMG method converges for the functionals
considered here for any starting iterate~\cite[Corollary~4.4]{graeser2019tnnmg}.
In practice, efficient performance has been observed for a variety of problems,
including in particular small-strain continuum elastoplastic problems~\cite{sander2020plasticity}
and brittle-fracture phase-field models~\cite{graeser_sander:2020}.


We finally turn to the specific case of the elastoplastic spring network of Section~\ref{S:random-network}.
There, the random elastic modulus~$a$,
hardening parameter $h$ and yield stress $\sigma_{\rm yield}$ associated with an edge
\begin{equation*}
  \mathsf e=(x,x+e_i),\qquad i=1,\ldots,k,\,x\in\Z^d,
\end{equation*}
(with $e_1,\ldots,e_k$ denoting the generating edges in $\mathsf E_0$) are of the form
\begin{equation*}
  a(\omega, \mathsf e)=a_i(\tau_{x}\omega),
  \qquad
  h(\omega,\mathsf e)=h_i(\tau_x\omega),
  \qquad
  \sigma_{\rm yield}(\omega,\mathsf e)=\sigma_{{\rm yield},i}(\tau_x\omega),
\end{equation*}
where $a_1,\ldots,a_k,h_1,\ldots,h_k,\sigma_{{\rm yield},1},\ldots,\sigma_{{\rm yield},k}:\Omega\to\R$ are measurable. Furthermore, there exists an ellipticity constant $\lambda>0$ such that
\begin{equation}\label{ass:ellipt}
  \lambda\leq a_1,\ldots,a_k,h_1,\ldots,h_k\leq\tfrac1\lambda
  \qquad \text{and} \qquad
  0\leq\sigma_{{\rm yield},1},\ldots,\sigma_{{\rm yield},k}\leq\tfrac1\lambda\qquad\text{$P$-a.s.}
\end{equation}
We can then go back to an edge-based representation of the energy and dissipation functionals
\begin{align*}
  \widehat{\mathcal E}^\omega_L(t_l,\hat y_L; F)
  & \colonequals
  \frac{L^{-d}}2\sum_{\mathsf e\in\mathsf E\cap\Lambda_L}a(\omega,\mathsf e)\big(\nabla_s\hat\varphi_L(\mathsf e)-\hat p_L(\mathsf e)\big)^2+h(\omega,\mathsf e)\hat p_L^2(\mathsf e)\\
  & \quad
  -L^{-d}\sum_{\mathsf e\in \mathsf E\cap \Lambda_L}a(\omega,\mathsf e)\nabla_s\hat \varphi_L(\mathsf e)\cdot F(t_l,\mathsf e), \\
\intertext{
where $F(t,\mathsf e)\colonequals \frac{\overline{\mathsf e}-\underline{\mathsf e}}{|\overline{\mathsf e}-\underline{\mathsf e}|}\cdot F(t)\frac{\overline{\mathsf e}-\underline{\mathsf e}}{|\overline{\mathsf e}-\underline{\mathsf e}|}$, and}
  \widehat{\rcal}^\omega_L(\dot{\hat y}_L)
  & \colonequals
  L^{-d}\sum_{\mathsf e\in\mathsf E\cap\Lambda_L}\sigma_{\rm yield}(\omega,\mathsf e) \abs[\big]{\dot{\hat p}_L(\mathsf e)}.
\end{align*}

The convexity and coercivity properties of $\mathcal E_l$ and $\rcal$ postulated
in Section~\ref{sec:simulating_rve_problem} follow here from the properties in~\eqref{ass:ellipt}.
Also, $\rcal$ is now block-separable with blocks of size~1. This implies that the
Gauß--Seidel iteration that is part of the TNNMG method now has to solve
sequences of minimization problems in only one variable, which simplifies the implementations.

For an implementation we may reformulate $\mathcal E_l:\R^{n+m}\to\R$ as
\begin{equation*}
  \mathcal E_l(y)=L^{-2}\big(\tfrac12 y\cdot\mathbf A(\omega)y-\mathbf f_l(\omega)\cdot y\big)
\end{equation*}
for a matrix $\mathbf A(\omega)\in\R^{(n+m)\times(n+m)}_{\rm sym}$ and
a vector $\mathbf f_l(\omega)\in\R^{n+m}$, where $\mathbf A(\omega)$ and $\mathbf f_l$
depend on the coefficients $\{a_i(\tau_z\omega)$, $h_i(\tau_z\omega)\,:\,i=1,\dots, k \text{ and $z\in\Lambda_L$}\}$,
and $\mathbf f_l$ additionally depends on the load $F(t_l)$.
Likewise, we may write $\mathcal R:\R^{n+m}\to\R$ in the form
\begin{equation*}
  \mathcal R(y)=L^{-2}\mathbf r(\omega)\cdot
  \begin{pmatrix}
    |y_1|\\
    \vdots\\
    |y_n|
  \end{pmatrix},
\end{equation*}
for a vector $\mathbf r(\omega)\in\R^n$ that only depends on the coefficients $\big\{\sigma_{{\rm yield},i}(\tau_z\omega)\,:\,i=1,\dots,k,\,z\in\Lambda_L\big\}$.


\section{Numerical experiments}\label{S:exp}
In this section, we numerically explore the behavior of the homogenized stress tensor
\begin{equation*}
  \sigma_{\hom}(t)=\mathcal W_{\hom}[F](t)
\end{equation*}
at a macroscopic point $x \in Q$ for different loading trajectories $F:[0,T]\to\R^{2\times 2}_{\rm sym}$,
using the numerical approach presented in Section~\ref{sec:numerical_approach}.
We consider the triangular, two-dimensional spring network with edges $e_1 = (1,0)$, $e_2 = (0,1)$,
$e_3 = (1,1)$ shown in Figure~\ref{fig:example_network},
with independent and identically distributed (i.i.d.) material parameters for the springs.
All material parameters are sampled from uniform distributions:
\begin{equation}\label{tbl:randomlaws}
  a(\mathsf e) \sim U\big(1\cdot 10^6,2\cdot 10^6\big),\qquad h(\mathsf e)\sim U\big(1.25\cdot 10^6,2\cdot 10^6\big),\qquad \sigma_{\rm yield}(\mathsf e)\sim U\big(0.9\cdot 10^3,1.1\cdot 10^3\big).
\end{equation}
%
%
The goal of the numerical study is two-fold: On the one hand, we are interested
in properties of $\sigma_{\hom}$ itself, e.g., its behavior in the case
of cyclic loading (see Section~\ref{S:cycle}). However, in Section~\ref{S:error}
we also investigate the convergence rate of the RVE approximation $\sigma_{N,L}^M(t)$
as $L\to\infty$. To understand the latter is crucial for determining the size $L$
of the RVE required to reach an acceptable approximation error. The choice of the
size of the RVE is intensively discussed in the computational mechanics community,
see, e.g.~\cite{drugan1996micromechanics, schneider2022representative}.
Nevertheless, analytical results are only known for elliptic systems and equations
(including the case of linear elasticity) \cite{GO11,GO12, gloria2015quantification, BellaOtto16,fischer2019optimal, GNO5}.
In particular, we mention \cite{gloria2015quantification} where estimates with
optimal scaling in $L$ and $M$ are obtained in the case of a discrete elliptic
equation with i.i.d. coefficients. In the case of elastoplasticity no analytical results
are available and it is unclear, if the methods developed for  elliptic equations
can be extended to the case of  elastoplasticity. In view of this,
in Section~\ref{S:error} we numerically study the behavior of the error
with respect to~$L$ in the case of monotonic loading.
We observe that it matches analytical results for linear elasticity.
This is surprising: The scaling for the linearly elastic case is based on elliptic regularity for the displacement field and requires the material properties to feature a rapid decay of spatial correlations (see in particular the large-scale regularity theory developed in \cite{gloria2020regularity}). In comparison, in elastoplasticity, no elliptic regularity theory is available for the plastic strain and it is not clear at all whether the plastically deformed material satisfies the required spatial decorrelation properties.

For the presentation of our numerical results we shall visualize stresses
not by means of the stress tensor $\sigma\in\R^{2\times 2}_{\rm sym}$,
but as the associated stress vector $s=P_s^*\sigma \in \R^{|\mathsf E_0|}$.
Recall from Example~\ref{ex:stress_vector} that
for the triangular, two-dimensional lattice, the relation between both quantities
is given by the formula
\begin{equation*}
  \sigma=
  \begin{pmatrix}
    s_1+\frac12 s_3 &\frac12 s_3\\
    \frac12 s_3&s_2+\frac12 s_3
  \end{pmatrix}.
\end{equation*}
Note that the first, second and third entry of $s$ describe the longitudinal stress in the direction of horizontal springs $(x,x+e_1)$, vertical springs $(x,x+e_2)$, and diagonal springs $(x,x+e_3)$, respectively.

\subsection{Uniaxial cyclic loading}\label{S:cycle}

As the first experiment, we consider the cyclic loading trajectory
\begin{equation*}
  F:[0,1]\to\R^{2\times 2}_{\rm sym},\qquad  F(t)=\begin{pmatrix}
    3\cdot10^{-3}\sin(8t) &0\\0&0\end{pmatrix}.
\end{equation*}
We approximate this on the time interval $[0,1]$,
which we split into $N = 50$ subintervals of equal size.  To simplify the notation
we always omit the dependence on $N$. The Monte-Carlo sampling
of the random network properties is done using $M = 5$ samples $\{\omega_i\}_{i=1}^5$
of the random coefficients.

Figure~\ref{fig:Hystereseloop} shows the stress--strain curve for the
RVE-approximation $\sigma_{L}^\omega(t)$ for four RVE sizes $L\in \cb{4, 6, 8, 40}$.
The five curves in each plot correspond to the five realizations of the
random coefficients. In the four plots the horizontal axis corresponds
to the $(1,1)$-component of the strain $F$, and the vertical axis corresponds
to the first component of the stress vector
$s_{L}^\omega(t) \colonequals P_s^{-1}(\sigma_{L}^\omega(t))$.
\begin{figure}
  \begin{minipage}[b]{0.5\linewidth}
    \centering
    \subfigure[$L = 4$]{
    \includegraphics[width=0.85\linewidth]{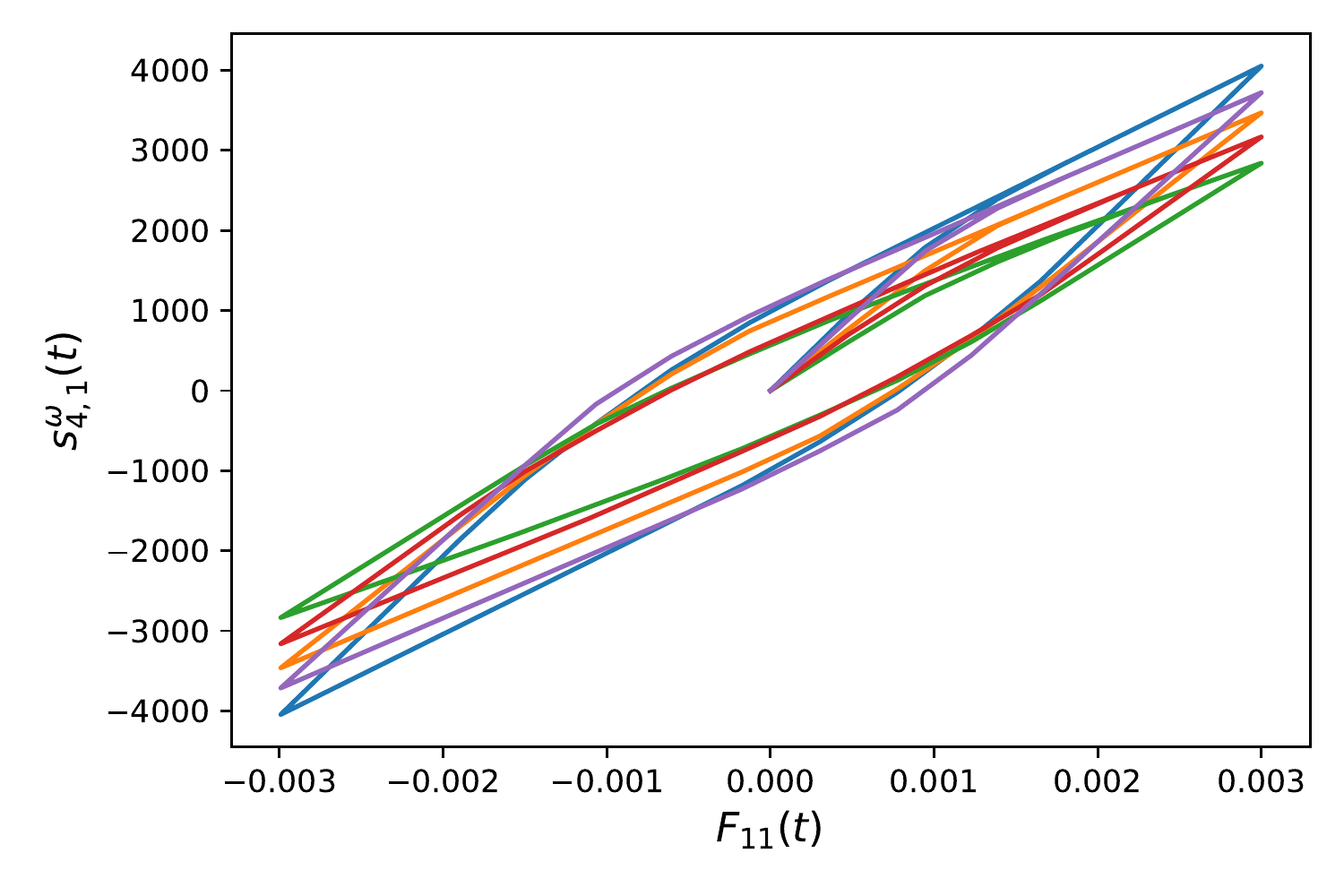}
    }
  \end{minipage}
  \begin{minipage}[b]{0.5\linewidth}
    \centering
    \subfigure[$L = 6$]{
    \includegraphics[width=0.85\linewidth]{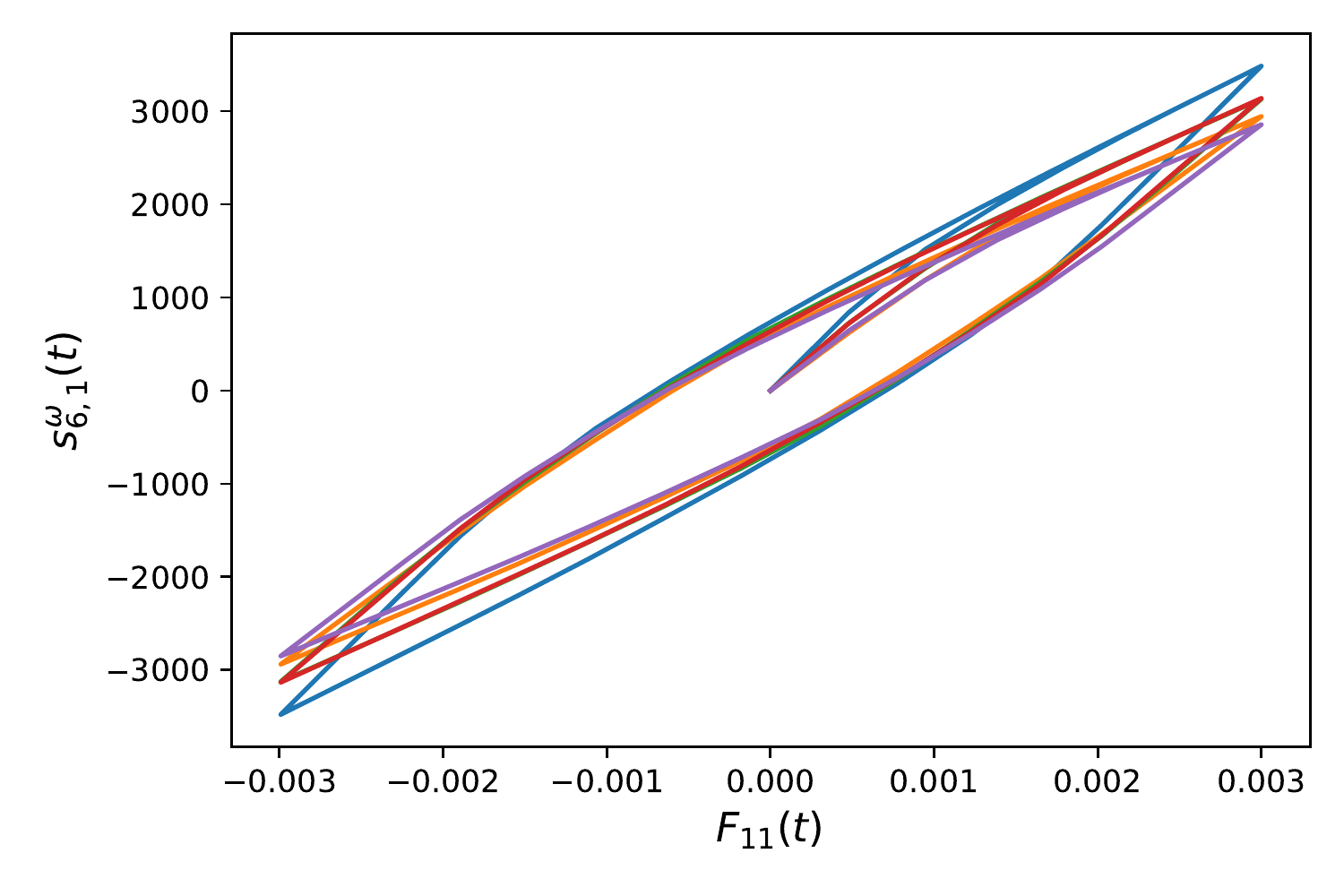}
    }
  \end{minipage} 
  \begin{minipage}[b]{0.5\linewidth}
    \centering
    \subfigure[$L = 8$]{
    \includegraphics[width=0.85\linewidth]{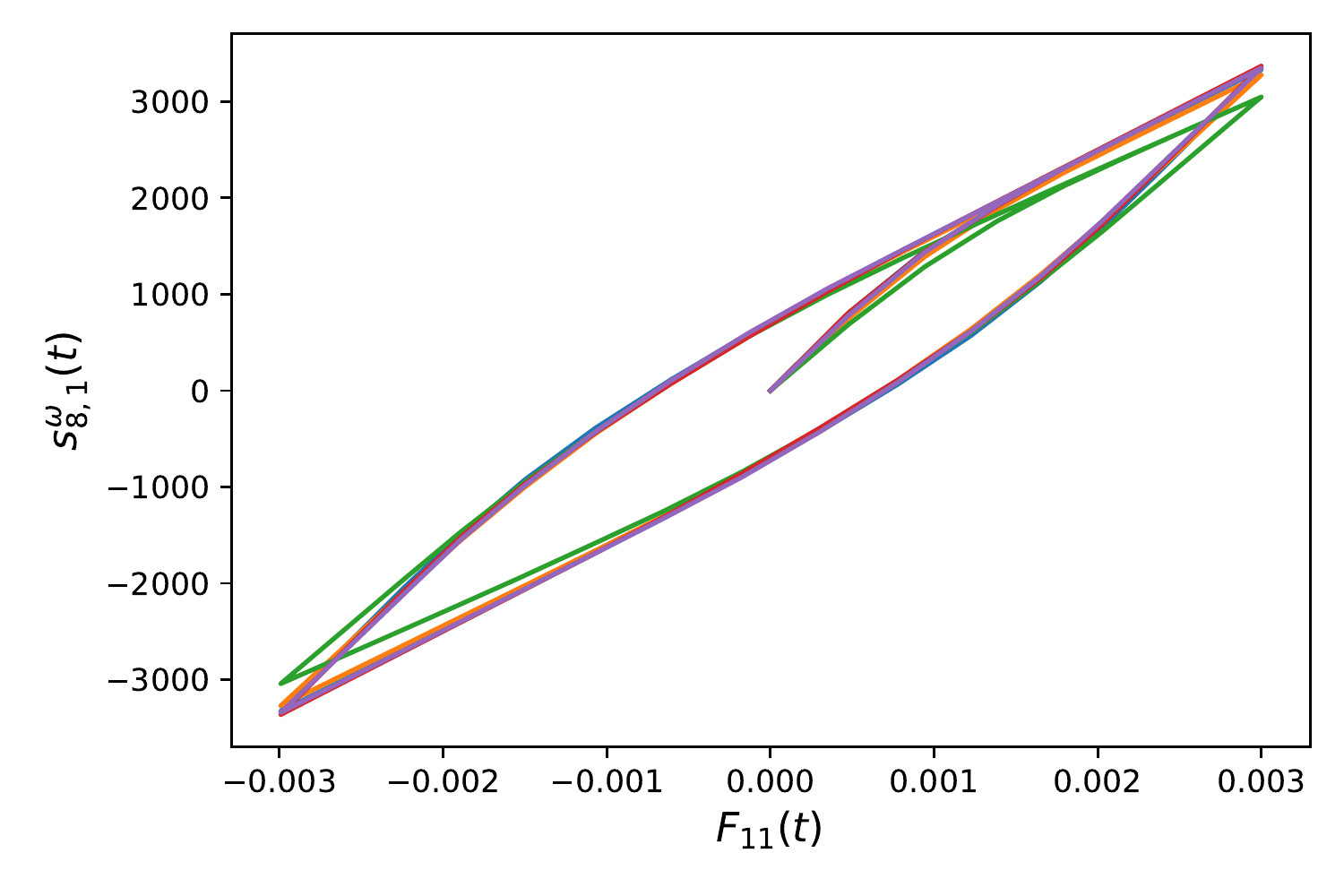}
    }
  \end{minipage}
  \begin{minipage}[b]{0.5\linewidth}
    \centering
    \subfigure[$L = 40$]{
    \includegraphics[width=0.85\linewidth]{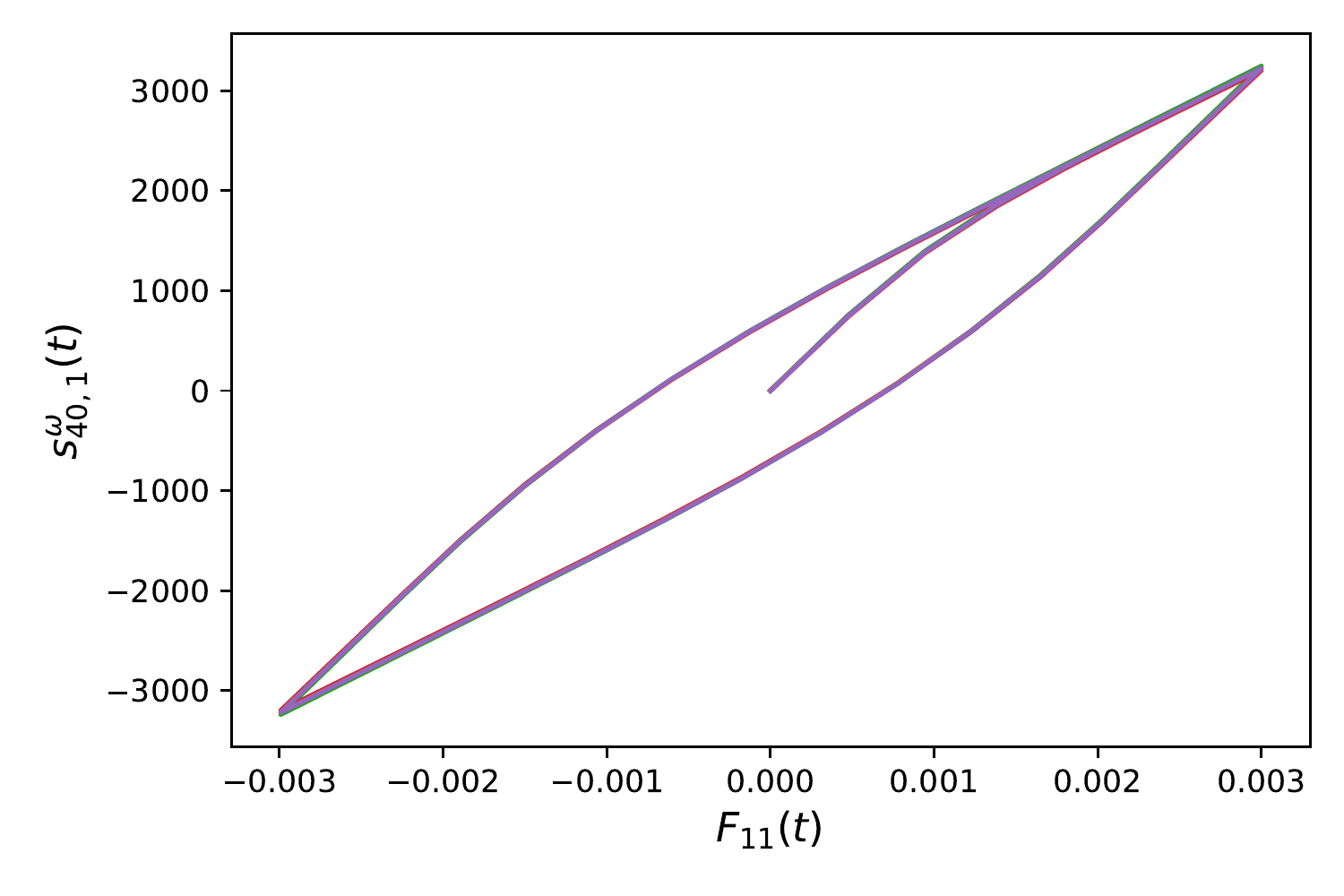}
    }
  \end{minipage} 
  \caption{Stress--strain plots for two-dimensional RVE approximations
   for a cyclic macroscopic deformation $F$. The different colors correspond to
   five different realizations $\{\omega_i\}_{i=1}^5$ of the coefficients.
   The first component of $s^{\omega_i}_{L}(t)$ is mapped as a function of $F_{11}(t),\ t\in[0,T]$.}
\label{fig:Hystereseloop}
\end{figure}

The diagrams indicate that in the limit $L\to\infty$ the stress--strain curves converge
to a single hysteresis curve that is independent of the choice of the random sample.
The limit curve is the one associated with the homogenized stress vector $\sigma_\text{hom}$.
This numerically confirms the convergence statement from Theorem~\ref{T:RVE-W}.

Furthermore, we observe a transition from linear to nonlinear hardening: A closer look
at the stress--strain curves for $L=4$ shows that these curves are piecewise linear
with a finite set of linear segments. This not surprising, since each individual spring
is governed by an elastoplastic material law with linear hardening, and thus
the stress--strain curve of a single spring along a monotone load path
consists of at most two linear segments. In fact, for each finite $L$, the stress--strain
curve will be composed of finitely many linear segments, the number increasing with $L$.
In this way, the nonlinear hardening behavior of the homogenized model
(which exhibits stress--strain curves with curved segments) is approximated by
the piecewise linear curves of the RVE approximation.

\subsection{Uniaxial monotonic loading}\label{S:monotone}

As the second example we consider the monotonic, uniaxial loading trajectory
\begin{equation}\label{eq:850}
  F:[0,1]\to\R^{2\times 2}_{\rm sym},\qquad F(t)=\begin{pmatrix}
    0.0034t &0\\0&0\end{pmatrix}
\end{equation}
on the time interval $[0,1]$, discretized by $N = 50$ time steps.
This can be interpreted as a uniaxial stretching of the system in the direction of
$e_1 = (1,0)$. The values are chosen such that initially only the horizontal
springs undergo plastic deformation. After all horizontal springs have started
to deform plastically, the diagonal springs start to leave the elastic regime.
The vertical springs stay elastic in the considered load range.

This time we simulate $M = 40$ different realizations
$\boldsymbol{\omega} = (\omega_1, \dots, \omega_{40})$ of the random coefficients,
again using the uniform distributions of \eqref{tbl:randomlaws}.
We focus on a single RVE size of $L = 30$. The Monte-Carlo approximation
of the stress vector is then
\begin{equation*}
  \sigma^{M,\boldsymbol{\omega}}_L(t)\colonequals\frac{1}{M}\sum_{i=1}^{M}\sigma_L^{\omega_i}(t).
\end{equation*}
We again visualize the associated stress vector, which we denote by
\begin{equation*}
 \overline{s}(t)
 \colonequals
 P_s^{-1}(\sigma^{M,\boldsymbol{\omega}}_L(t)).
\end{equation*}
  
In Figure~\ref{fig:SigmaHomTimeFirstThirdComp_stress} we visualize the first
and third component of the stress vector $\overline s(t)$ as a function of
the strain component $F_{11}$. We do not discuss the second component in detail, which
corresponds to stresses in the vertical direction---in view of the loading trajectory \eqref{eq:850}, which strains the material in the horizontal direction only, the stress in the vertical direction only plays a minor role for the mechanical response of the system, see Figure~\ref{fig:SigmaHomTimeSecondComp} and note that the stresses are much smaller than those for the horizontal and diagonal direction.

The two plots in Figure~\ref{fig:SigmaHomTimeFirstThirdComp_stress}
show that the stress--strain curves undergo three phases: A (seemingly) linear phase at the beginning, a nonlinear, monotone phase in the middle, and another (seemingly) linear phase (with a smaller slope) at the end. In the following we argue that these phases can be explained by considering the amount of plastic deformation accumulated in the system. To that end we introduce three regimes for the state of the system, namely,
\begin{itemize}
\item An \emph{elastic regime}, where none of the springs in the system is plastically deformed,
\item a \emph{transitional regime}, where some but not all of the springs in the system are plastically deformed,
\item a \emph{plastic regime}, where all of the springs in the system are plastically deformed.
\end{itemize}
In view of the anisotropic geometry of the underlying network, we introduce
these regimes for horizontal, vertical, and diagonal springs separately.
It is convenient to call springs to be of type $\alpha\in\{1,2,3\}$ if the they are
generated by edges $e_\alpha\in \mathsf E_0$. In particular, edges of types~$1$ and~$3$
correspond to horizontal and diagonal springs, respectively.
%
%
%
%
For $\alpha=1,2,3$ we compute the arithmetic mean of the proportion of plastically deformed springs of type $\alpha$ in terms of the ratio
\begin{equation*}
  R_{\alpha}
  \colonequals
  \frac{\text{Number of all plastically deformed springs of type $\alpha$}}{\text{Number of springs of type $\alpha$}},
\end{equation*}
and say that the springs of type $\alpha$ of the system are in the 
\begin{equation*}
  \begin{cases}
    \text{elastic regime}&\text{if }  R_{\alpha}\approx 0,\\
    \text{plastic regime}&\text{if }  R_{\alpha}\approx 1,\\
    \text{transitional regime}&\text{else.}
  \end{cases}
\end{equation*}
The ratios $R_\alpha$ for horizontal and diagonal springs are plotted
in Figure~\ref{fig:SigmaHomTimeFirstThirdComp_ratio}. Moreover, the type of the regime is indicated by the color of the  background of the plots.

Finally, in Figure~\ref{fig:SigmaHomTimeFirstThirdComp_slope} we plot the
numerical slope of the stress--strain curves, which is given by the quotient
\begin{equation*}
  \frac{\overline{s}_\alpha(t_k)-\overline{s}_\alpha(t_{k-1})}{F_{11}(t_k)-F_{11}(t_{k-1})},\qquad \alpha\in\{1,3\}
\end{equation*}
for each time step $t_k$, $k=1,\dots,N$.

\begin{figure}
     \centering
    \subfigure[Stress--strain curves. Left: $\overline{s}_1$, right: $\overline{s}_3$]{
    \includegraphics[width=0.9\linewidth]{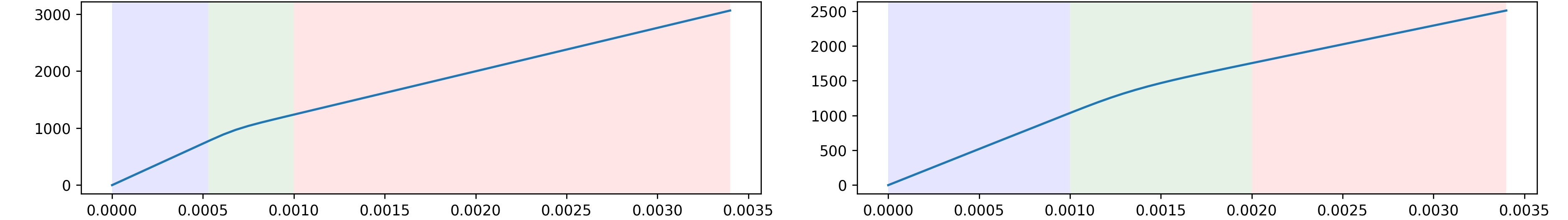}
    \label{fig:SigmaHomTimeFirstThirdComp_stress}
    }
    \subfigure[Fraction of plastically strained springs.
      Left: horizontal springs, right: diagonal springs]{
    \includegraphics[width=0.9\linewidth]{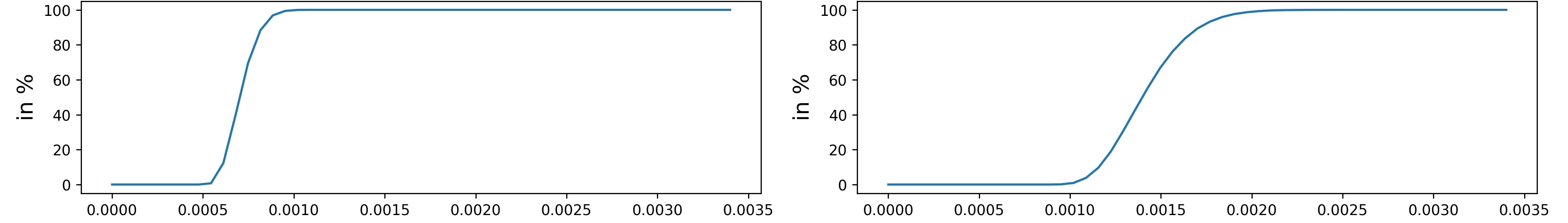}
    \label{fig:SigmaHomTimeFirstThirdComp_ratio}
    }
    \subfigure[Slope of the stress--strain curve. Left: $\overline{s}_1$, right: $\overline{s}_3$]{
    \includegraphics[width=0.9\linewidth]{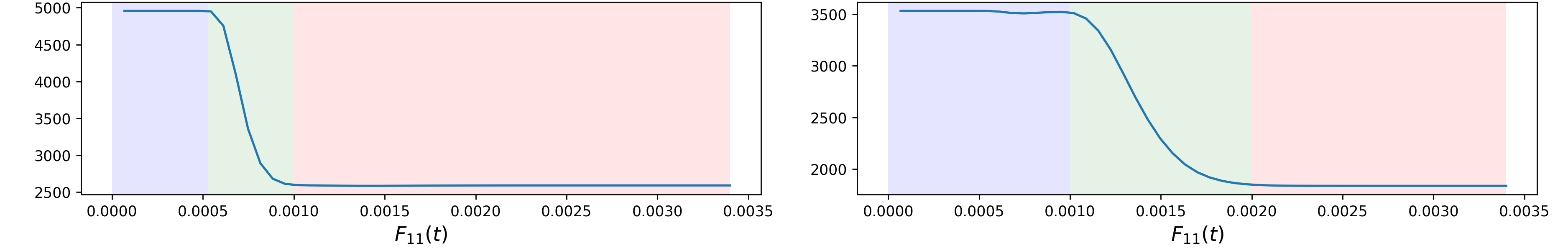}
    \label{fig:SigmaHomTimeFirstThirdComp_slope}
    }
    \caption{Stress response under monotonic load}
\label{fig:SigmaHomTimeFirstThirdComp}
\end{figure}

We interpret these results in the following way:
At the beginning of the loading experiment the plastic strain is zero and all springs
are in the elastic regime. In that regime the stress--strain curve must be linear.
At $F_{11}(t)\approx 0.5\cdot 10^{-3}$ the horizontal springs enter the transitional regime, while the diagonal (and vertical) springs still remain in the elastic regime.
In the (horizontal) transitional regime the slope of $\overline{s}_1$ drops sharply,
while in the third component $\overline{s}_3$ we only observe a small drop in the slope.
At $F_{11}(t)\sim 0.01$ all horizontal springs are deformed plastically
and a clear change in behavior of the first component can be recorded:
Its stress--strain relation seems again linear. However, a closer look at the slope
in the first component shows that linearity is only attained when all diagonal springs are plastically deformed as well.
\todosander{Können wir das noch irgendwie im Bild besser verdeutlichen,
z.B. durch einen farbigen Hintergrund?}

This can be seen from Figure~\ref{fig:ZoomSigmaHomTimeFirstComp}, which shows
a close-up view of the strain region $F_{11} \in [10^{-3}, 3.5 \cdot 10^{-3}]$.
There, all horizontal strings are plastic, but the diagonal springs are not.
The adapted stress range in Figure~\ref{fig:ZoomSigmaHomTimeFirstComp} shows that
while $\overline{s}_1$ may look linear, strictly speaking, it is not.

\begin{figure}
\centering
\includegraphics[width=0.4\linewidth]{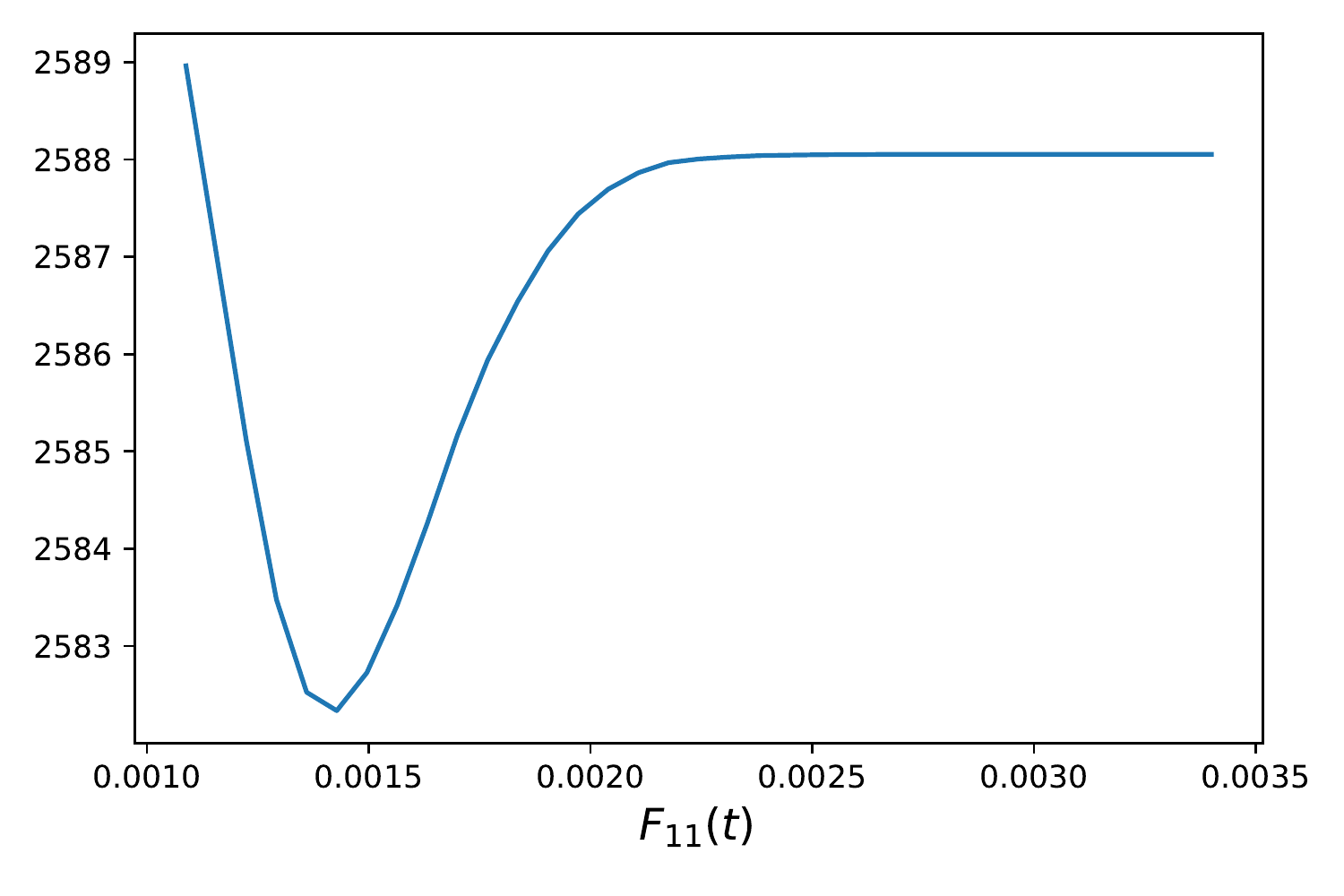}
\caption{Close-up of the slope of the first component of $\overline{s}$
         as a function of $F_{11}$}
\label{fig:ZoomSigmaHomTimeFirstComp}
\end{figure}

The third component of the stress vector features an analogous effect.
In particular, a rapid decrease of the slope of the third component of the
stress--strain curves can be observed, when the diagonal springs are in the
transitional regime. The third component
of the stress-vector starts with a linear increase, followed by oscillations during the horizontal and diagonal transitional regimes. When both, the horizontal and diagonal springs are in the plastic regime, the curve continuous with a linear increase.

\begin{figure}
\centering
\includegraphics[width=0.6\linewidth]{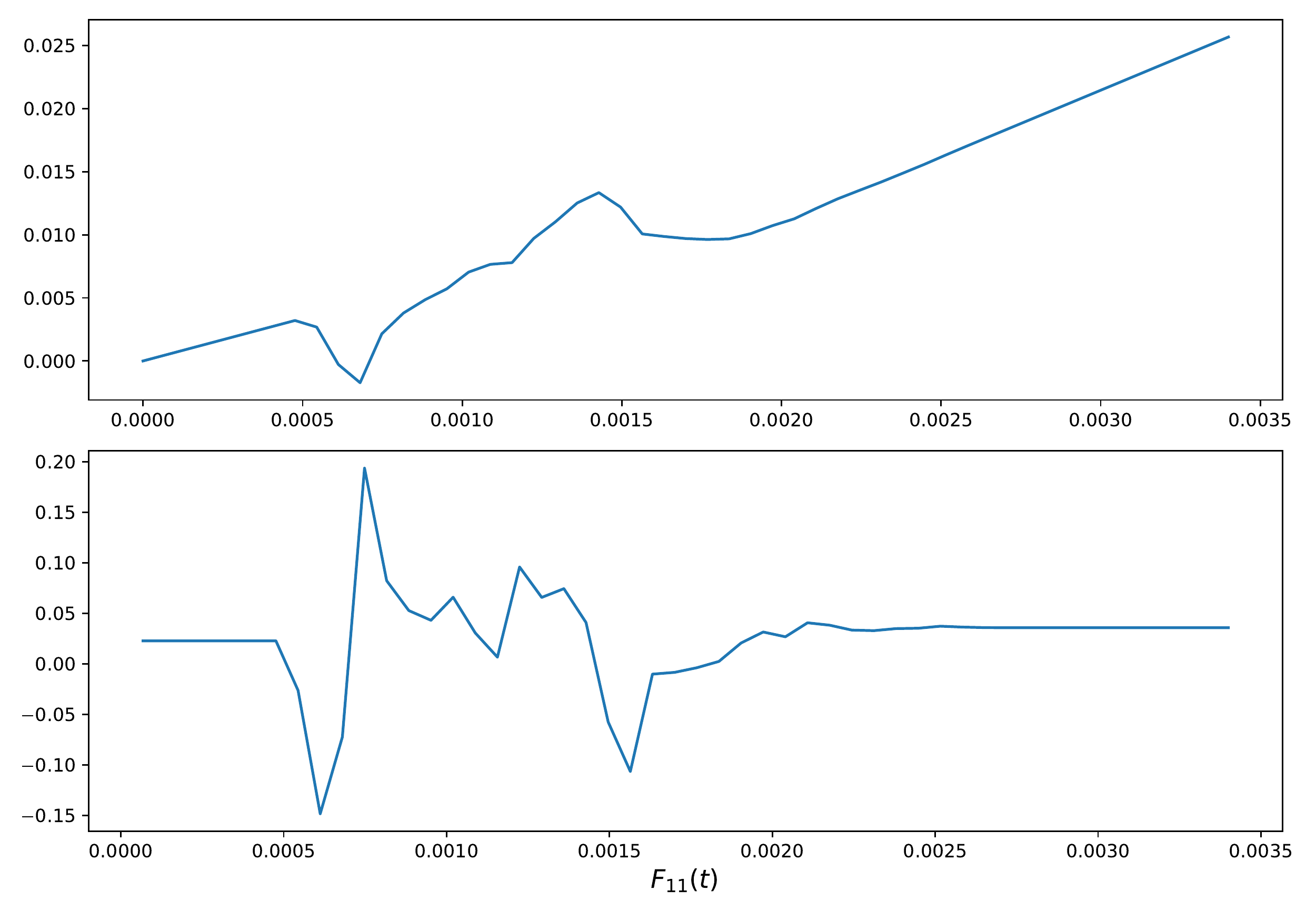}
\caption{(top) Second component of $\overline{s}(t)$ as a function of $F_{11}(t)$;
 (bottom) slope of the second component of $\overline{s}(t)$}
\label{fig:SigmaHomTimeSecondComp}
\end{figure}

\begin{remark}[Conclusion]\label{R:conclusion}
We may summarize our analysis of the numerical, uniaxial, monotonic loading experiment as follows:
\begin{itemize}
\item The stress-strain relation is affine as long as none of the springs is in the transitional regime. In our simulation this is the case for the two intervals approximately given by
  \begin{equation}\label{nontransitional}
    [0,0.0005],\ [0.002,0.034].
  \end{equation}
  The first interval corresponds to the case when all types of springs are in the elastic regime, and the second interval corresponds to the case when springs of type $1,2$, and $3$ are in the plastic, elastic and plastic regime, respectively.
\item If a spring of type $\alpha$ enters the transitional regime, the slope of the
  $\alpha$-component of the  stress response decreases rapidly. The other components are less affected,
  but show oscillations, as shown, e.g., in Figures~\ref{fig:ZoomSigmaHomTimeFirstComp} and \ref{fig:SigmaHomTimeSecondComp}. In our simulation this is the case for the intervals
  \begin{equation}\label{eqtransitional}
    [0.0005, 0.001],\ [0.001,0.002]
  \end{equation}
  where the first interval is dominant for the behavior of the first stress component, while the second interval is dominant for the third stress component.  
\end{itemize}
\end{remark}

\subsection{RVE approximation errors}\label{S:error}
In this section we numerically explore the error of the RVE approximation. We mainly focus
on the monotonic, uniaxial loading experiment of the previous section.

Recall that $\sigma_L^{M,\boldsymbol{\omega}}(t)$ with $\boldsymbol{\omega} = (\omega_1, \dots,\omega_M)$
is a random matrix defined on the $M$-fold product of $(\Omega,\mathcal F,P)$.
Also recall that this quantity depends on the number $N$ of time steps used to
discretize the time interval $[0,T]$. We continue to use the values $T = [0,1]$
and $N = 50$,
and to omit the $N$ from the notation for simplicity.
We denote the expectation associated to the $M$-fold product of $(\Omega,\mathcal F,P)$
by $\mathbb E_M$. A natural measure of the RVE error is then given by
\begin{equation*}
  \norm[\big]{\sigma_L^M(t)-\sigma_{\hom}(t)}
  \colonequals
  \Big(\mathbb E_M\big[\abs[\big]{\sigma_{L}^M(t)-\sigma_{\hom}(t)}^2\big]\Big)^\frac12.
\end{equation*}
This total error can be bounded by the sum of a random and a systematic part, i.e.,
\begin{equation}
\label{eq:total_error_bound}
  \norm[\big]{\sigma_L^M(t)-\sigma_{\hom}(t)}
  \leq
  \operatorname{Var}_M^\frac12\big(\sigma_L^M(t)\big)+ \abs[\big]{\mathbb E_M\big[\sigma_L^M(t)\big]-\sigma_{\hom}(t)}.
\end{equation}
Since $\sigma_L^M$ is defined as an arithmetic mean with respect to~$M$ independent samples
with distribution $P$, we have
\begin{equation*}
  \mathbb E_M\big[\sigma_L^M(t)\big]
  =
  \mathbb E\big[\sigma_L(t)\big],
\end{equation*}
and we see that the systematic error simplifies to
\begin{equation*}
    E_{{\rm sys}}(L;t)
\colonequals \abs[\big]{\mathbb E\big[\sigma_L(t)\big]-\sigma_{\hom}(t)}.
\end{equation*}
In particular, it is independent of the sample size $M$. On the other hand, for the random error we have
\begin{equation*}
  \operatorname{Var}_M^\frac12\big(\sigma_L^M(t)\big)
  \leq
  M^{-1}  E_{{\rm rand}}(L;t),\qquad   E_{{\rm rand}}(L;t)
\colonequals\operatorname{Var}^\frac12\big(\sigma_L(t)\big),
\end{equation*}
and we see that the random error can be reduced by increasing $M$. In view of this, it is important to analyze the scaling for the random and systematic error separately, since a difference in scaling will lead to a different choice $M$ when optimizing the overall computational cost.

\subsubsection{Analytic results for elastic systems}
\label{sec:analytical_error_bounds}

In the completely elastic regime, i.e., when $p(t)=0$, the stress $\sigma_L^M(t)$ is determined
by a lattice model of linear elasticity on a periodic RVE. This is a discrete elliptic system
with independent and identically distributed (i.i.d.) coefficients. Its mathematical structure
is close to the discrete elliptic equations with random coefficients whose correlations
decay rapidly as considered in \cite{gloria2015quantification}. There, it is shown that the
random error for the periodic RVE approximation (with $M=1$) scales like $L^{-\frac d2}$,
while the systematic error scales like $L^{-d}(\ln L)^d$. Since the scaling of
the systematic error is much quicker, the total error can be reduced by an appropriate choice
of $M$. In particular, for $M=L^d$ both errors have the same scaling w.r.t.~$L$ (up to a logarithmic correction).

Similar results have been obtained for continuum, linear elliptic systems in \cite{BellaOtto16},
and for monotone, uniformly elliptic systems in \cite{fischer2019optimal}.
Transferred to our situation we infer that in the elastic regime, the random error
can be estimated as
\begin{equation}\label{eq:RandErrorReference}
  E_{{\rm rand}}(L;t)\leq CL^{-\frac d2}|F(t)|,
\end{equation}
while the systematic error can be estimated as
\begin{equation}
  \label{eq:SystErrorReference}
  E_{{\rm sys}}(L;t)\leq CL^{-d}(\ln L)^d|F(t)|,
\end{equation}
with a constant $C$ independent of the loading trajectory $t\mapsto F(t)$. 

We finally remark that in \cite{GNO5} linear elliptic systems with correlated coefficients
with a slow decay of correlations are considered. The results there indicate
that in such a case the decay rate of the random error can be much smaller.

\subsubsection{Numerical studies of the systematic error}

We numerically explore the systematic error, which we consider componentwise on the
level of the stress vector. That is, for $\alpha=1,2,3$ we consider the quantity
\begin{equation*}
  E_{{\rm sys},\alpha}(L;t)
  \colonequals
  \abs[\Big]{\mathbb E\big[s_{L,\alpha}(t)\big]-s_{\hom,\alpha}(t)},
\end{equation*}
where $s_{L,\alpha}$ and $s_{\hom,\alpha}$ denote the $\alpha$-component of $P_s^{-1}(\sigma_L)$
and $P_s^{-1}\sigma_{\hom}$, respectively. In order to numerically evaluate $E_{{\rm sys},\alpha}$, we proceed as follows:
\begin{enumerate}
\item We sample $M=25$ realizations of the coefficients $a,h,\sigma_{\rm yield}$
on the box $\Lambda_{L_\text{max}}$, with $L_\text{max}\colonequals 42$.
In our notation, we represent these samples by $\omega_1,\ldots,\omega_M$. Realizations of the coefficients for $L\leq L_{\text{max}}$ are obtained by restriction the smaller box $\Lambda_L$.
\item For $L \in \cb{6,10,14,18,22,26,30,34,38,42}$
 \begin{enumerate}
  \item we compute the Monte-Carlo average
   \begin{equation*}
    \overline{s}_L(t)\colonequals \frac{1}{M}\sum_{i=1}^MP_s^{-1}(\sigma_L^{\omega_i}(t)),
   \end{equation*}
  \item we evaluate the quantity
   \begin{equation*}
    \overline{E}_{{\rm sys},\alpha}(L;t)
    \colonequals
    \abs[\big]{\overline{s}_{L,\alpha}(t)-\overline{s}_{L_{\rm max},\alpha}(t)}.
   \end{equation*}
 \end{enumerate}
\end{enumerate}
The quantity $\overline{E}_{{\rm sys},\alpha}(L;t)$ is an approximation of $E_{{\rm sys},\alpha}$,
where the expectation $\mathbb E\big[s_{L,\alpha}(t)\big]$ is approximated by a Monte-Carlo average,
and $s_{\hom,\alpha}(t)$ by approximated by $\overline{s}_{L_{\text{max}},\alpha}$.
In Figure~\ref{fig:SystematicError} we plot the three components of $\overline{E}_{{\rm sys}}(L,t)$
as a function of $F_{11}(t)$ for the different values of $L$. Also, we point out the three regimes
discussed in Section~\ref{S:monotone} by plotting the fractions of plastically deformed springs.

\begin{figure}
  \centering
 \includegraphics[width=0.6\textwidth]{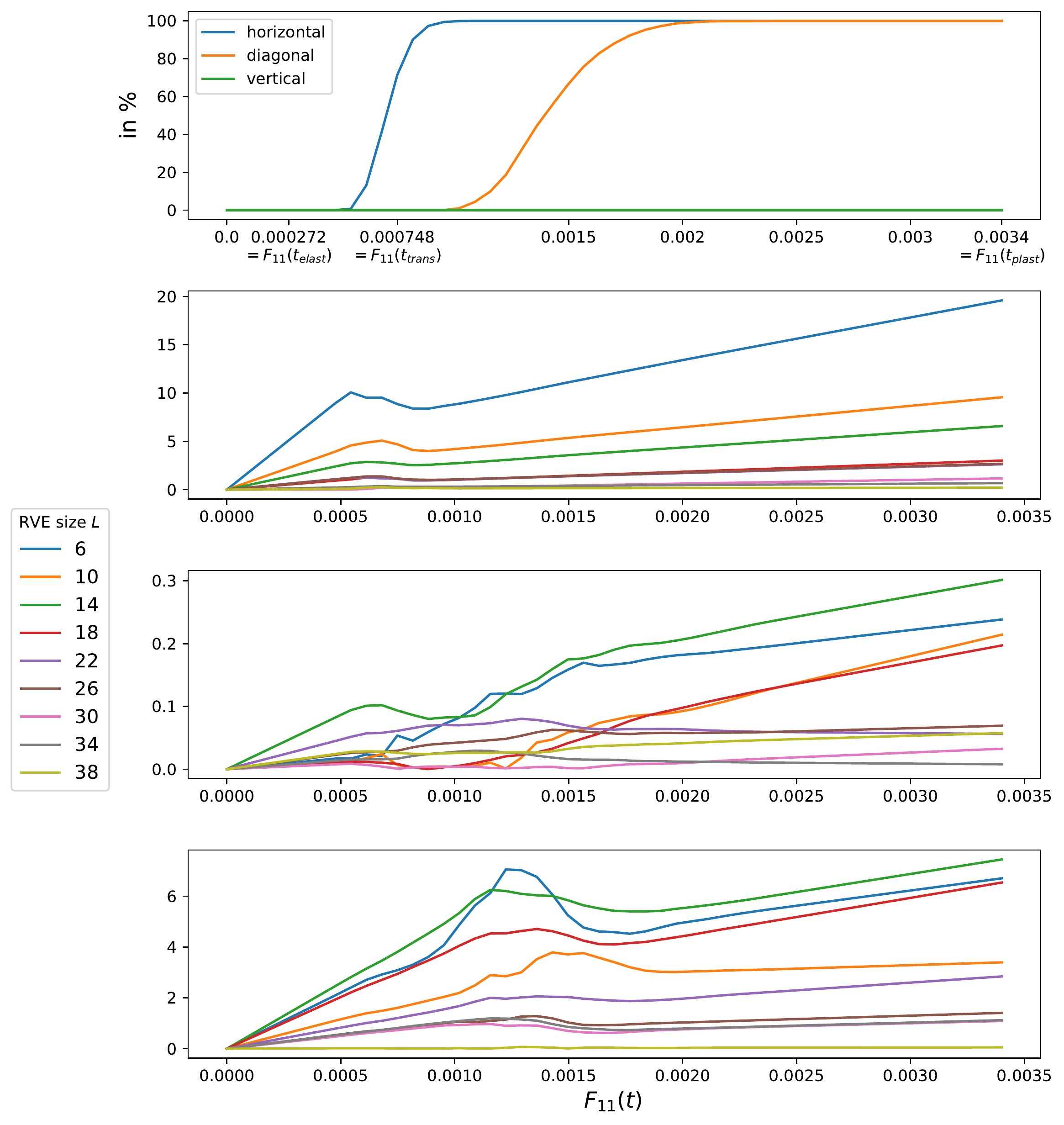}
 \caption{Numerical estimate of the systematic error; (top image) Fractions of deformed springs $R_1,R_2,R_3$;
  (below) $|\overline{s}_{L}(t)-\overline{s}_{42}(t)|$ for first, second and third component, $M=25$ realizations for each RVE size $L$}
 \label{fig:SystematicError}
\end{figure}

The simulations confirm the theoretical prediction that in the elastic regime,
the error $\overline{E}_{{\rm sys}}(L;t)$ is proportional to $F_{11}(t)$ and that
the error decreases with $L$. The simulations also show a linear dependence
of $\overline{E}_{{\rm sys}}(L;t)$ on $F_{11}(t)$ in the complement of the transitional regime
(i.e., for $F_{11}(t)$ above about $0.002$). On the other hand, in the transitional regime (horizontal/diagonal) we observe a non-monotone and oscillating dependence.

In Figure~\ref{fig:SystematicErrorGridSize} we study the systematic errors at selected
particular load steps, again as function of the RVE size $L$. We select pseudo times
\begin{equation}\label{def:times_regimes}
  t_{\rm elast}\colonequals 0.08,
  \qquad
  t_{\rm trans}\colonequals 0.22,
  \qquad
  t_{\rm plast}\colonequals 1.0,
\end{equation}
which correspond to states of the system in the elastic, transitional, and plastic regime.
The corresponding values of $F_{11}$ are shown in the first plot of Figures~\ref{fig:SystematicError}.

We note that this is possible,
since the fraction of plastically deformed springs, and therefore the partition
of the load history into regimes, is almost independent of $L$.

In Figure~\ref{fig:SystematicErrorGridSize} we plot for $t\in\{t_{\rm elast},t_{\rm trans},t_{\rm plast}\}$ and $L \in \cb{6,10,14,18,22,26,30,34,38}$ the quantity
\begin{equation}\label{eq:sysrelerror}
  \frac{\overline{E}_{{\rm sys},1}(L;t)}{\overline{s}_{L_{\rm max},1}(t)},
\end{equation}
which is an approximation of the first component of the systematic, relative error.
In the figure, as a reference curve, a function $L\mapsto \Theta( L^{-2}(\ln L)^2)$ (with Landau notation) is plotted.
It describes the scaling that has been theoretically predicted for elastic systems
(cf.\ \eqref{eq:SystErrorReference}).
We see a good agreement up to $L=26$ in all three regimes.
For $L>26$ our approximation of the systematic error is not sufficiently accurate; in fact, we have $\overline{E}_{{\rm sys},\alpha}(L_{\text{max}},t)=0$ by definition.
In the elastic regime, our simulation confirms the theoretical result
of Section~\ref{sec:analytical_error_bounds}. However, no theoretical results exist that cover the transitional and plastic regime.

\begin{figure}
  \centering
 \includegraphics[width=0.49\textwidth]{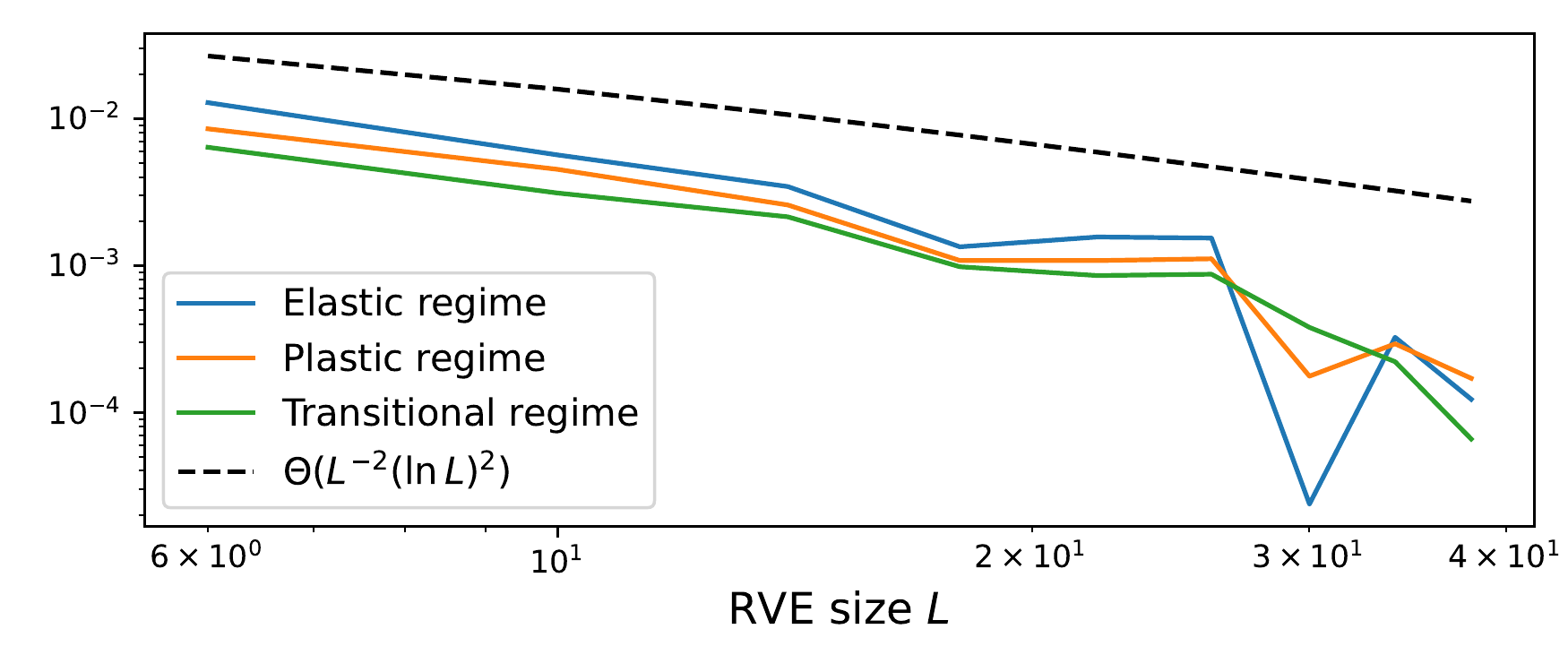}
 \caption{The first component of the relative, systematic error, cf.~\eqref{eq:sysrelerror}, at pseudo time steps in the elastic, transitional and pseudo-plastic, cf.~\ref{def:times_regimes}, for $M=25$  and $L \in \cb{6,10,14,18,22,26,30,34,38}$; the dashed line represents the upper bound of \eqref{eq:SystErrorReference}}
 \label{fig:SystematicErrorGridSize}
\end{figure} 
 
\subsubsection{Numerical study of the random error}

We study the (squared) random error of the RVE approximation.
As shown in~\eqref{eq:total_error_bound} it is a term appearing in the upper bound of the
RVE error.
In order to numerically evaluate this quantity, we consider the biased sample variance:
For each one of the $M$ independent samples $\omega_1,\ldots,\omega_M$ we compute
the approximate stress vector
\begin{equation*}
  s_L^{\omega_i}(t) \colonequals P_s^{-1}(\sigma_L^{\omega_i}(t)),
\end{equation*}
and for each component $\alpha=1,2,3$ the sample variance
\begin{equation}\label{samplevariance}
  {\boldsymbol s}^2_M[s_{L,\alpha}(t)]
  \colonequals
  \frac{1}{M}\sum_{i=1}^M
  \abs[\bigg]{s_{L,\alpha}^{\omega_i}(t)-\frac{1}{M}\sum_{j=1}^M s_{L,\alpha}^{\omega_j}(t)}^2.
\end{equation}
For the simulations in this section we set $M=40$ and $L=30$.
The plots in Figure~\ref{fig:VarianceSigmaHom} show the three components of ${\boldsymbol s}^2_M[s_L(t)]$
as functions of the horizontal macroscopic strain $F_{11}(t)$.
The colored background illustrates the three regimes.
Furthermore, we show the fraction of plastically deformed springs with the same orientation to establish a possible relation with the behavior of the error.
We observe a quadratic dependence on $F_{11}(t)$ in the elastic and plastic regime, while the behavior
is non-monotone in the transitional regime.
To make this more quantitative, we also plot the numerical slope of $  {\boldsymbol s}^2_M[s_{L}(t)]$,
i.e., the difference quotient
\begin{equation}\label{numerical_slope}
 \frac{  {\boldsymbol s}^2_M[s_L(t_k)]-  {\boldsymbol s}^2_M[s_L(t_{k-1})]}{F_{11}(t_k)-F_{11}(t_{k-1})}.
\end{equation}

\begin{figure}
  \centering
  \includegraphics[width=0.5\linewidth]{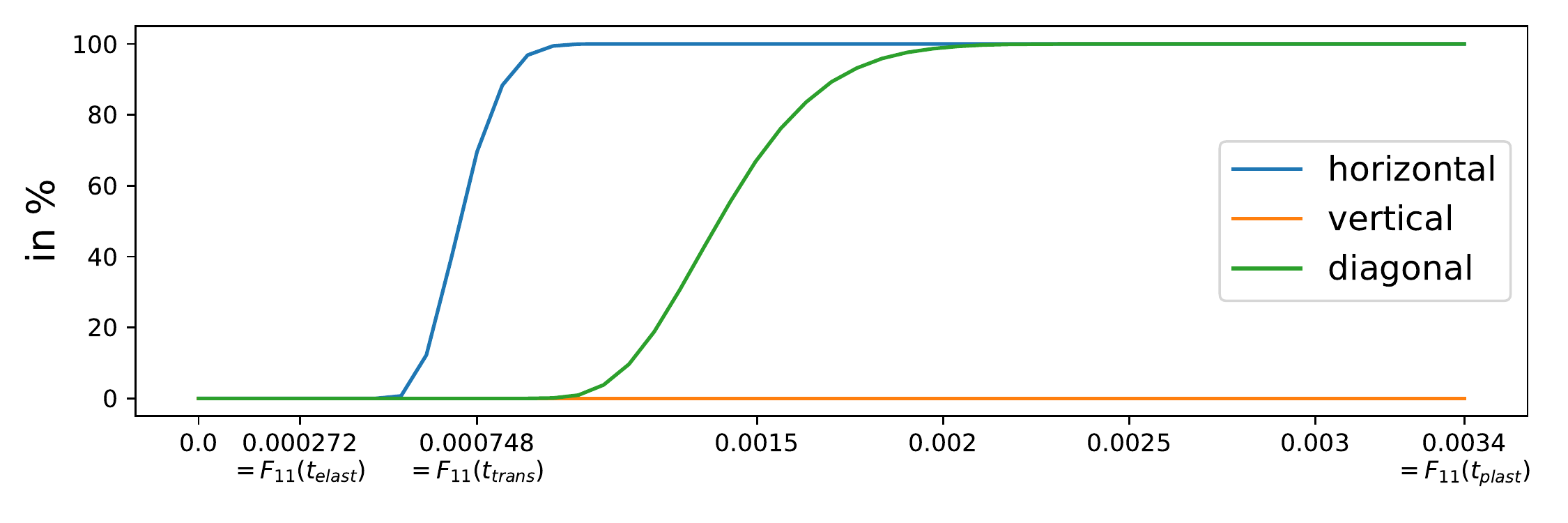}\\
  \begin{minipage}[b]{0.5\linewidth}
    \centering
    \subfigure[Variance of the stress vector components]{
    \includegraphics[width=0.95\linewidth]{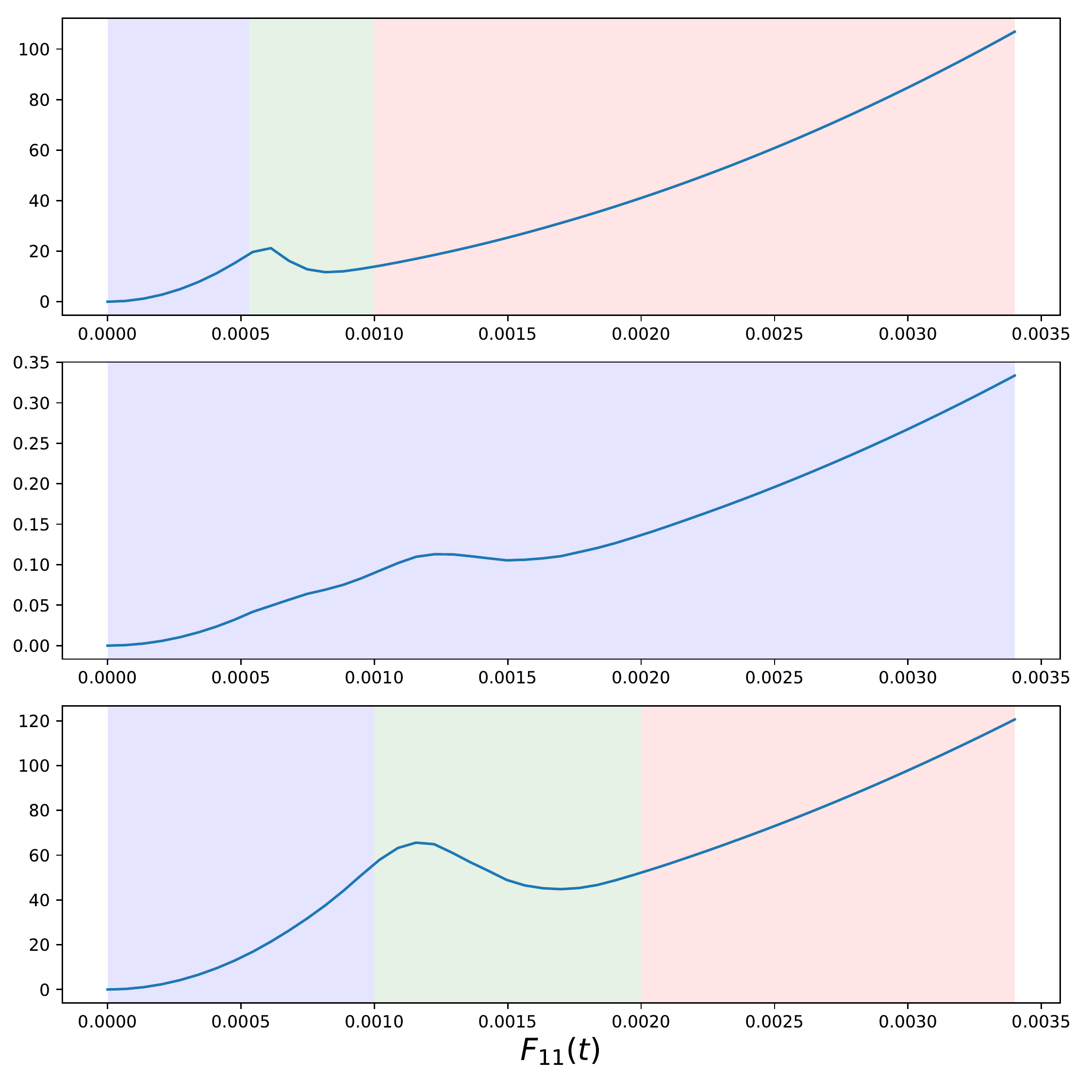}
    }
  \end{minipage}
  \begin{minipage}[b]{0.5\linewidth}
    \centering
   \subfigure[Slope of the variance]{
    \includegraphics[width=0.95\linewidth]{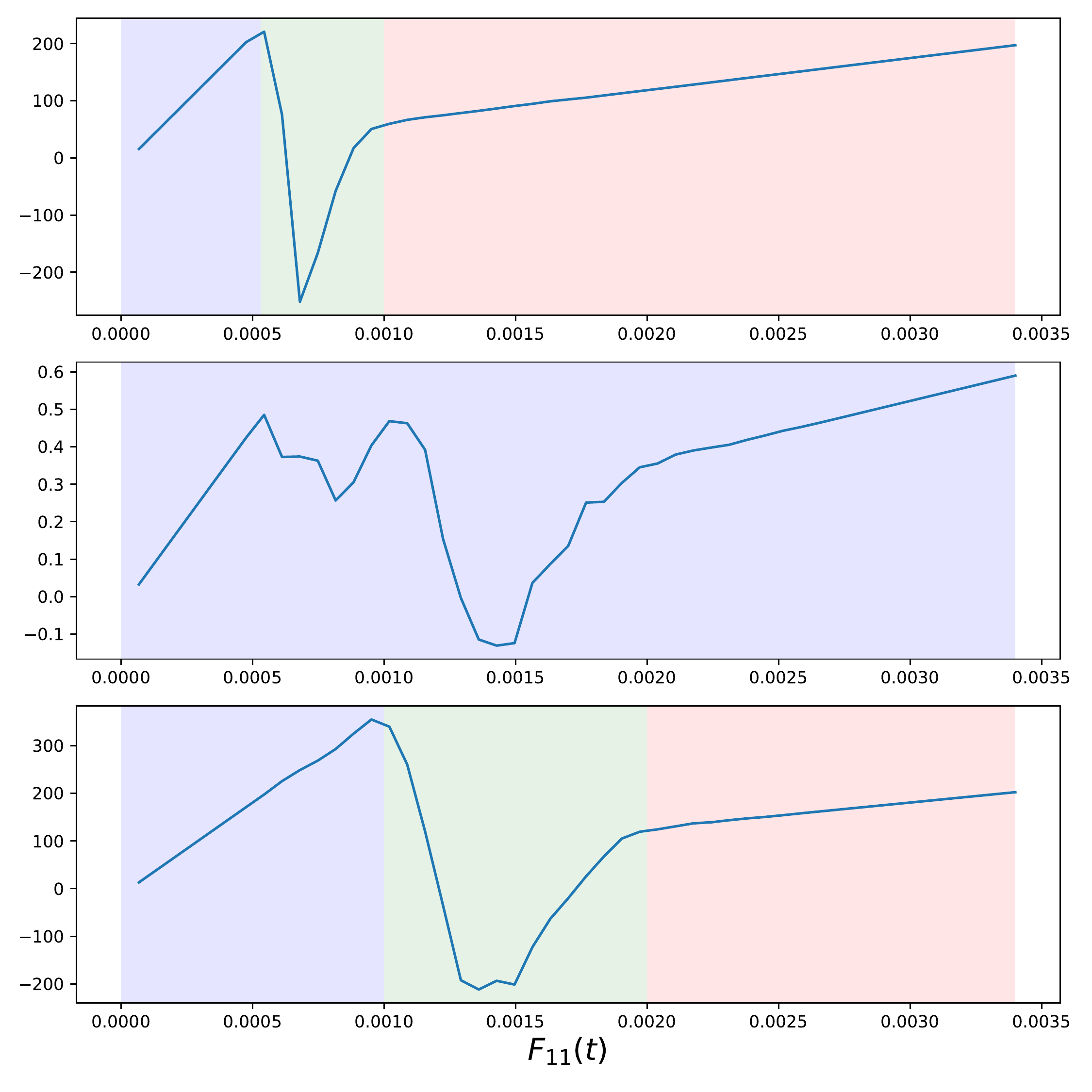}
   }
  \end{minipage}

  \caption{Biased sample variance of $s_L(t)$, cf.~\eqref{samplevariance} and its numerical slope, cf.~\eqref{numerical_slope}, for $L=30$.
  (a) Behavior of the sample variance of $s_L$ as a function of $F_{11}(t)$ in all three components; (b) Slope of $s_L$ as a function of $F_{11}(t)$ in all three components.}
  \label{fig:VarianceSigmaHom}
 \end{figure} 

 The quadratic dependence on $F_{11}(t)$ in the elastic regime confirms the theoretical
 estimate~\eqref{eq:RandErrorReference}, while the quadratic behavior in the plastic regime
 is not covered by the existing theory.
 Moreover, similarly to the systematic error we observe that the variance has an
 non-monotone, oscillating dependence
 on $F_{11}(t)$ in the transitional regime.
 In particular, when the horizontal springs enter the transitional regime at $t\approx 0.0006$, the sample variance and 
 its slope in the first component drop sharply. On the other hand, the third component is only influenced very mildly by the transitions of horizontal springs: By closely looking at the slope of the third component of the sample variance $t\approx 0.0006$, we observe a small deviation from the quadratic dependence., cf.\ Figure~\ref{fig:ZoomSigmaHomTimeFirstComp}.

Next, we study the scaling of the squared random error on the size $L$ of the RVE. We only consider the first component. In Figure~\ref{fig:RandomErrorGridSize}
the first component of the relative, biased sample variance
\begin{equation}\label{rel_samplevariance}
  \frac{  {\boldsymbol s}^2_M[s_{L,\alpha}(t)]}{(F_{11}(t))^2}
\end{equation}
is plotted as a function of the parameter $L$  for the times
$t\in\{t_{\rm elast}, t_{\rm trans}, t_{\rm plast}\}$ defined in \eqref{def:times_regimes}.
Recall that these pseudo times correspond to a state of the system in the elastic,
transitional and plastic regime, respectively. We choose the values for $L$ and $M$
as in Figure~\ref{fig:SystematicError}. For a better understanding, the reference function
$L \mapsto L^{-2}$, which describes the theoretical scaling of the squared random error in the elastic regime, cf.~\eqref{eq:RandErrorReference}, is shown as well.

\begin{figure}
  \centering
  \includegraphics[width=0.6\textwidth]{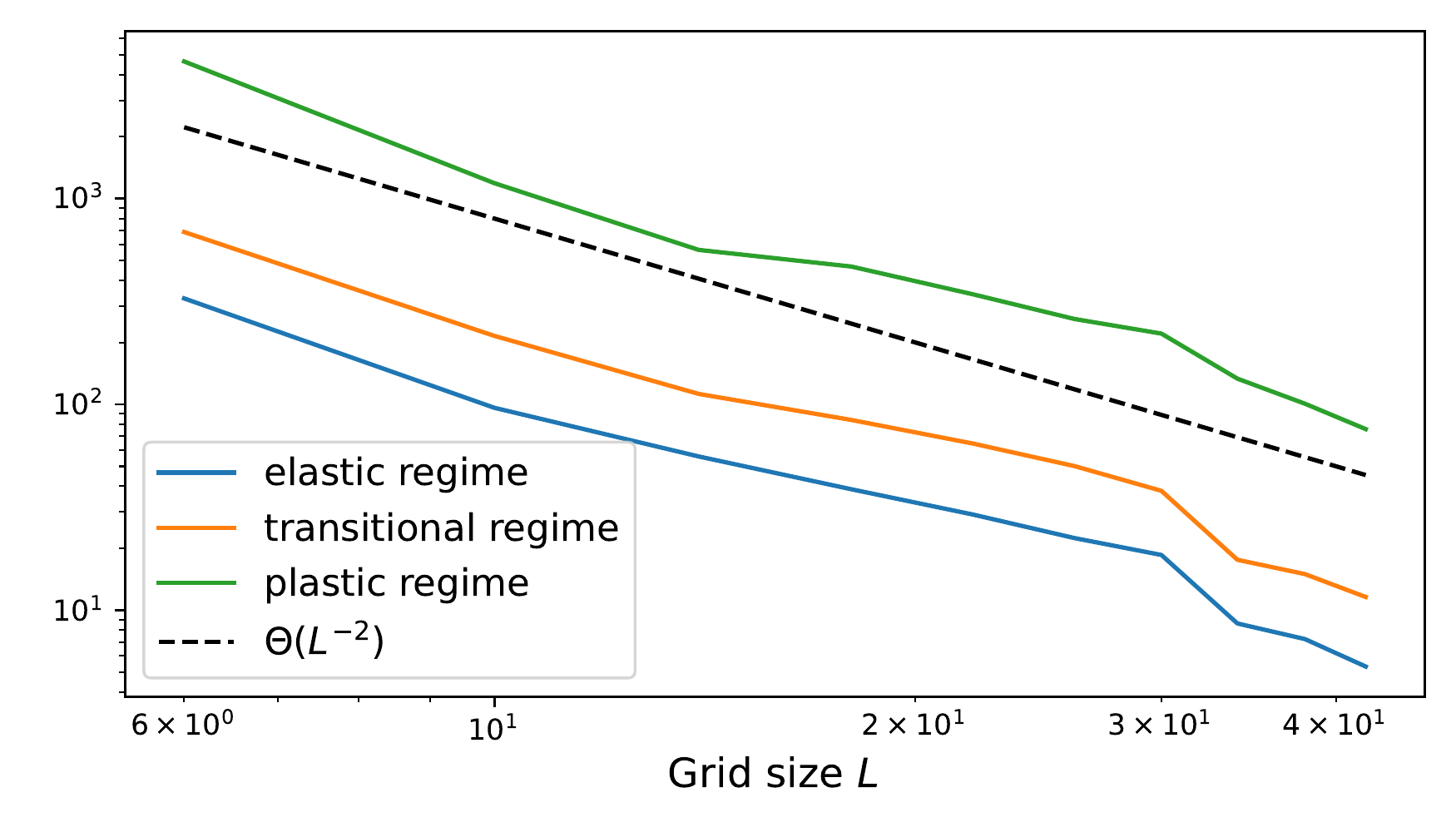}
  \caption{Relative sample variance~\eqref{rel_samplevariance} of the first component,
    at the pseudo time steps~\eqref{def:times_regimes} in the elastic, transitional
    and pseudo-plastic regimes. Monte-Carlo sampling uses $M=25$ realizations and $L \in \cb{6,10,14,18,22,26,30,34,38,42}$; the dashed line represents the (square of the) upper bound of \eqref{eq:RandErrorReference}.}
  \label{fig:RandomErrorGridSize}
\end{figure} 
Our simulations show that in all regimes the variance decays with a quadratic rate. In the elastic regime, this confirms the theoretical result \eqref{eq:RandErrorReference}.

\subsubsection{Conclusion}
We may summarize our numerical study of the RVE error for uniaxial, monotonic loading experiment as follows:
\begin{itemize}
\item In all three regimes the systematic and random error decay with a rate that is similar to the one theoretically predicted for linear elasticity.
\item The systematic and squared random error for fixed $L$ seen as a function of the pseudo time $t$ (which in the loading experiment is proportional to the loading of the system) are affine and quadratic, respectively, as long as none of the springs is in the transitional regime, i.e. in the two intervals approximately given by \eqref{nontransitional}.
\item The dependence of the errors on the pseude time $t$ is non-monotone and oscillatory if springs are in the transitional regime. In our simulation this is the case for the two intervals approximately given by \eqref{eqtransitional}. In particular, if springs of type $\alpha$ enter the transitional regime, then notable changes in the behavior of the corresponding component of the error can be observed, while the other components of the error are only mildly effected.
\end{itemize}
These observations resemble the ones of Remark~\ref{R:conclusion}.
We see that in the considered specific case of a monotonic, uniaxial loading experiment, where the applied load is non-zero only in the direction of horizontal springs, the cross-coupling between springs of different orientation is very mild.
By symmetry we expect a similar behavior, when the load is exclusively applied in the direction of the vertical and diagonal springs.
The question whether a similar behavior can be observed for more general loading paths or networks with a more complex geometry will be discussed in future works.


\section{Proof of analytical results}\label{sec:proofs}
\subsection{Preliminaries and auxiliary results}
Before presenting the proofs we first introduce some additional structures and auxiliary results.

\subsubsection{Calculus on \texorpdfstring{$L^2(\Omega)$}{L2(Omega)} and the space \texorpdfstring{$L^2_{\rm pot}(\Omega)$}{L2pot(Omega)}.}
Recall the definition of the probability space $(\Omega,\mathcal F,P)$ and the discrete dynamical system $\{\tau_{x}\}_{x\in\Z^d}$ from Section~\ref{S:probsetup}.
For  a measurable function $\varphi: \Omega\rightarrow \R^d$ we define its \textit{stochastic gradient} as
\begin{align*}
  D\varphi(\omega)=(D_1\varphi(\omega),...,D_d\varphi(\omega)),\qquad \text{where }D_i\varphi(\omega)=\varphi(\tau_{e_i}\omega)-\varphi(\omega).
\end{align*}
Note that $D:L^2(\Omega)^d\rightarrow L^2(\Omega;\R^{d\times d})$ is linear and bounded. We introduce the space of potential fields
\begin{equation}\label{eq:280:ud2}
  L^2_{\mathrm{pot}}(\Omega)^d:=\overline{\{D\varphi\,:\,\varphi\in L^2(\Omega)^d\}},
\end{equation}
where $\overline{(\cdot)}$ denotes the closure in $L^2(\Omega;\R^{d\times d})$.
\begin{lemma}[Relation between $\chi$ and $\chi_s$]\label{L:vector_pot}
  Assume \ref{B0:pl:d} and \ref{item:assumption_measurability}--\ref{item:assumption_ergodicity}.
  \begin{enumerate}[label=(\roman*)]
  \item For $i=1,\ldots,k$ there exists a function $B_i:\Z^d\to\Z^d$ with compact support such that for all $\varphi\in L^2(\Omega)$ we have
    \begin{equation*}
      \varphi(\tau_{e_i}\omega)-\varphi(\omega)=\sum_{x\in\Z^d}\chi(\tau_{-x}\omega)B_i(x).
    \end{equation*}
  \item 
  The map
  \begin{equation*}
    T:L^2_{\rm pot}(\Omega)^d\to L^2_s(\Omega),\qquad T\chi(\omega)\colonequals\sum_{x\in\Z^d}\chi(\tau_{-x}\omega)B_i(x),
  \end{equation*}
  is a linear isomorphism.
  \end{enumerate}
\end{lemma}
\begin{proof}
  \begin{enumerate}[label=(\roman*)]
  \item Since any $e_i\in\mathsf E_0$ can be written as a linear combination of $e_1,\ldots,e_d$, there exists a path $\ell_0=0,\ldots,\ell_m=e_i$ such that $\ell_{j}-\ell_{j-1}\in\{\pm e_1,\ldots,\pm e_d\}$. Hence, by the group property of $x\mapsto\tau_x$, we have $\varphi(\tau_{e_i}\omega)-\varphi(\tau_0\omega)=\sum_{j=1}^m\varphi(\tau_{\ell_j}\omega)-\varphi(\tau_{\ell_{j-1}}\omega)$. Since the differences $\varphi(\tau_{\ell_j}\omega)-\varphi(\tau_{\ell_{j-1}}\omega)$ can be expressed by one of the column vectors of $D\varphi$, the claim follows.
  \item In view of (i), and the definition of the spaces $L^2_{\rm pot}(\Omega)^d$ and $L^2_s(\Omega)$, the map $T$ is linear, bounded and onto. Furthermore the Korn inequality for random field, see \cite[Lemma 4.17]{neukamm2017stochastic}, which follows from the Korn inequality  of Assumption~\ref{B0:pl:d}, implies that $\|\varphi\|_{L^2(\Omega;\R^{d\times d})}\leq C\|T\varphi\|_{L^2(\Omega)^k}$ and thus we see that $T$ is an isomorphism.
  \end{enumerate}
\end{proof}
%

\subsubsection{The transformation \texorpdfstring{$\mathcal T_L$}{TL}}
In the proofs of Theorems~\ref{thm:1462:d} and ~\ref{T:RVE-W} we appeal to the transformation $\mathcal T_L$ and a variant denoted by $\widehat{\mathcal T}_L$. They are closely related to the stochastic unfolding operator introduced in \cite{neukamm2017stochastic}.
We start this section by summarizing some auxiliary results and properties of the transformations $\mathcal T_L$ and $\widehat{\mathcal T}_L$.

We use the following notation: For $x\in\R^d$ we denote by $\lfloor x\rfloor$ the integer part of $x$, i.e.,
\begin{equation}
  \lfloor x\rfloor =z\text{ where $z_i=\max\{m\in\Z\,:\,m\leq x\}$ for all $i=1,\ldots,d$.}
\end{equation}
Furthermore, we set $\Lambda:=[0,1)^d$ and define
\begin{align*}
  H^1_{\mathrm{per}}(\Lambda)&:=\big\{\varphi\in H^1(\Lambda)\,:\,\varphi\text{ is periodic, i.e., }\varphi(\cdot+k)=\varphi\text{ for all }k\in\Z^d\,\big\},\\
  H^1_{\mathrm{per,av}}(\Lambda)&:=\big\{\varphi\in H^1_{\rm per}(\Lambda)\,:\,\int_\Lambda\varphi=0\,\big\}.
\end{align*}

\newcommand{\unfh}{\widehat{\mathcal T}_L}
The first proposition is about properties of the transformation
\begin{equation*}
  \unfh:L^2(\Omega)\otimes L^2(\Lambda_L)\to L^2(\Omega)\otimes L^2(\Lambda),\qquad \unfh\varphi(\omega,x)\colonequals\varphi(\tau_{-\lfloor Lx\rfloor}\omega, \lfloor Lx\rfloor).
\end{equation*}

\begin{lemma}[Properties and definition of $\unfh$]\label{L:unfh}
  Assume \ref{B0:pl:d} and \ref{item:assumption_measurability}--\ref{item:assumption_ergodicity}.
  \begin{enumerate}[label=(\roman*)]
  \item (Isometry property). For any $L\in\N$ the map $\unfh$ is a linear isometry.
  \item (Transformation property). Let $g:\Omega\times\R\to\R$ be measurable and non-negative. Then
    \begin{equation*}
      \mathbb E\Big[L^{-d}\sum_{z\in\Lambda_L}g\big(\tau_z\omega,\varphi(\omega,z)\big)\Big]=\mathbb E\Big[\int_\Lambda g(\omega,\unfh\varphi(\omega,x))\,dx\Big],
    \end{equation*}
    for all $\varphi\in L^2(\Omega)\otimes L^2(\Lambda_L)$.
    
  \item (Compactness). Consider a sequence $(\varphi_L)$, $\varphi_L\in L^2(\Omega)\otimes L^2(\Lambda_L)$ that is bounded, i.e.,
    \begin{equation*}
      \limsup\limits_{L\to\infty}\|\varphi_L\|_{L^2(\Omega)\otimes L^2(\Lambda_L)}<\infty.
    \end{equation*}
    Then there exists $\varphi\in L^2(\Omega)\otimes L^2(\Lambda)$ and a subsequence such that
    \begin{equation*}
      \unfh\varphi_L\wto \varphi\qquad\text{weakly in }L^2.
    \end{equation*}
  \item (Compactness for gradients). Consider a sequence $(\varphi_L)$, $\varphi_L\in L^2(\Omega)\otimes L^2(\Lambda_L)^d$ that is bounded in the sense that
    \begin{equation*}
      \limsup\limits_{L\to\infty}\big(L^{-1}\|\varphi_L\|_{L^2(\Omega)\otimes L^2(\Lambda_L)}+\|\nabla_s\varphi_L\|_{L^2(\Omega)\otimes L^2(\Lambda_L)}\big)<\infty.
    \end{equation*}
    Then there exists $\varphi\in H^1_{\mathrm{per}}(\Lambda)$, $\chi_s\in L^2_s(\Omega)\otimes L^2(\Lambda)$ and a subsequence such that
    \begin{equation*}
      \big(L^{-1}\unfh\varphi_L,\unfh\nabla_s\varphi_L\big)\wto \big(\varphi,P_s\nabla\varphi+\chi_s\big)\qquad\text{weakly in }L^2.
    \end{equation*}

  \item (Recovery sequence).
    Let $\varphi\in H^1_{\mathrm{per,av}}(\Lambda)^d$, $\chi_s\in L^2_s(\Omega)\otimes L^2(\Lambda)$. Then there exists a sequence  $(\varphi_L)$, $\varphi_L\in L^2(\Omega)\otimes L^2_{\mathrm{av}}(\Lambda_L)^d$ such that
    \begin{equation*}
      \big(L^{-1}\unfh\varphi_L,\unfh\nabla_s\varphi_L\big)\to \big(\varphi,P_s\nabla\varphi+\chi_s\big)\qquad\text{strongly in }L^2.
    \end{equation*}
  \end{enumerate}
\end{lemma}
\begin{proof}
  The properties (i) and (ii) follow from (S3) as can be seen by a direct computation.
  Property (iii) follows from (i) and  the fact that bounded sets in $L^2(\Omega)\otimes L^2(\Lambda)$ are weakly compact.
  The remaining properties follow from analogous statements for the stochastic unfolding operator $\mathcal T_\eps$ introduced in \cite{neukamm2017stochastic}. It maps a discrete function $\varphi_\eps:\Omega\times\eps\Z^d\to\R$ to the continuum function $\mathcal T_\eps\varphi_\eps:\Omega\times\R^d\to\R$ defined by
  \begin{equation*}
    {\mathcal T}_{\eps}\varphi_\eps(\omega,x)\colonequals\varphi_\eps(\tau_{-\lfloor\tfrac{x}{\eps}\rfloor}\omega,\lfloor\tfrac{x}{\eps}\rfloor).
  \end{equation*}
  The connection to our setting is as follows: If we consider $\varphi_L\in L^2(\Omega)\otimes L^2(\Lambda_L)$, and set
  \begin{equation*}
    \varphi_\eps(\omega,z)\colonequals \eps\varphi_L(\omega,\tfrac{z}{\eps}),\qquad \eps=L^{-1},
  \end{equation*}
  then for all $x\in\R^d$ and $P$-a.e.~$\omega\in\Omega$ we have
  \begin{equation}\label{eq:equivunfold}
    \frac1L\unfh\varphi_L(\omega,x)=\mathcal T_\eps\varphi_\eps(\omega,x)\qquad\text{and}\qquad (\unfh\nabla_s\varphi_L)(\omega,x)=(\mathcal T_\eps\nabla_s\varphi_\eps)(\omega,x).
  \end{equation}
  Furthermore, we have the following identities for the norms:
  \begin{equation}\label{eq:equivnorm}
    \frac1L\|\varphi_L\|_{L^2(\Lambda_L)}=\|\varphi_\eps\|_{L^2(\Lambda)},\qquad 
    \|\nabla_s\varphi_L\|_{L^2(\mathsf E\cap \Lambda_L)}=\|\nabla_s\varphi_\eps\|_{L^2(\eps\mathsf E\cap \Lambda)}.
  \end{equation}
  Based on these principles and Lemma~\ref{L:vector_pot}, (iv) and (v) can be deduced from \cite[Proposition 3.5 and Proposition 3.8]{neukamm2017stochastic}.
\end{proof}

By extending $\unfh$ to the domain $L^2(\Omega)\otimes L^2(\Lambda_L)\otimes L^2(Q)$ with $Q\subset\R^d$ open and bounded, we recover the operator $\mathcal T_L$ introduced in Section~\ref{S:RVE_conv}, and the properties of $\unfh$ summarized in Lemma~\ref{L:unfh} are inherited by $\mathcal T_L$:
\begin{lemma}\label{lem:1204}
  Let $\mathcal T_L$ be defined as in \eqref{eq:TL}.
  \begin{enumerate}[label=(\roman*)]
  \item (Isometry property). For any $L\in\N$ the map $\mathcal T_L$ is
    a linear isometry.
\item (Compactness). Consider a sequence $(u_L)$, $u_L\in L^2(\Omega)\otimes L^2(\Lambda_L) \otimes L^2(Q)$ that is bounded, i.e.,
    \begin{equation*}
      \limsup\limits_{L\to\infty}\|\varphi_L\|_{L^2(\Omega)\otimes L^2(\Lambda_L)\otimes L^2(Q)}<\infty.
    \end{equation*}
    Then there exists $\varphi\in L^2(\Omega)\otimes L^2(\Lambda)\otimes L^2(Q)$ and a subsequence such that
    \begin{equation*}
      \mathcal T_L\varphi_L\wto \varphi\qquad\text{weakly in }L^2.
    \end{equation*}
  \item (Compactness for gradients).
    Consider a sequence $(\varphi_L)$, $\varphi_L\in L^2(\Omega)\otimes L^2(\Lambda_L)^d\otimes L^2(Q)$ that is bounded in the sense that
    \begin{equation}\label{eq:comp:grad}
      \limsup\limits_{L\to\infty}\big(L^{-1}\|\varphi_L\|_{L^2(\Omega)\otimes L^2(\Lambda_L)\otimes L^2(Q)}+\|\nabla_s\varphi_L\|_{L^2(\Omega)\otimes L^2(\Lambda_L)\otimes L^2(Q)}\big)<\infty.
    \end{equation}
    Then there exists $\varphi\in H^1_{\mathrm{per}}(\Lambda)\otimes L^2(Q)^d$, $\chi_s\in L^2_s(\Omega)\otimes L^2(\Lambda)\otimes L^2(Q)$ and a subsequence such that
    \begin{equation*}
      \big(L^{-1}\mathcal T_L\varphi_L,\mathcal T_L\nabla_s\varphi_L\big)\wto \big(\varphi,P_s\nabla\varphi+\chi_s\big)\qquad\text{weakly in }L^2.
    \end{equation*}
    Furthermore, if additionally $\varphi_L\in L^2(\Omega)\otimes L^2_{\rm av}(\Lambda_L)^d$, then $\varphi\in L^2(\Omega)\otimes H^1_{\mathrm{per,av}}(\Lambda)^d$.
  \item (Recovery sequence).
    Let $\varphi\in H^1_{\mathrm{per,av}}(\Lambda)\otimes L^2(Q)^d$, $\chi_s\in L^2_s(\Omega)\otimes L^2(\Lambda)\otimes L^2(Q)$. Then there exists a sequence  $(\varphi_L)$, $\varphi_L\in L^2(\Omega)\otimes L^2_{\mathrm{av}}(\Lambda_L)^d\otimes L^2(Q)$ such that
    \begin{equation*}
      \big(L^{-1}\mathcal T_L\varphi_L,\mathcal T_L\nabla_s\varphi_L\big)\to \big(\varphi,P_s\nabla\varphi+\chi_s\big)\qquad\text{strongly in }L^2.
    \end{equation*}
  \end{enumerate}
\end{lemma}
The proof is obvious and left to the reader.

\begin{lemma}[Compactness in the $z$ independent case]\label{L:comp_z_indep}
  Consider a bounded sequence $(u_L)\subset L^2(\Omega)\otimes L^2(Q)$. Then there exists $u\in L^2(Q)$ such that
  \begin{equation*}
    \mathbb E[u_L]\wto u\quad\text{weakly in }L^2(Q),\text{ and }\\
    \mathcal T_Lu_L\wto u\quad\text{weakly in }L^2(\Omega)\times L^2(\Lambda)\otimes L^2(Q).
  \end{equation*}
\end{lemma}
\begin{proof}
  We pass to a subsequence such that $u_L\wto \tilde u$ and $\mathcal T_Lu_L\wto v$ weakly for some $\bar u\in L^2(\Omega)\otimes L^2(Q)$ and $v\in L^2(\Omega)\otimes L^2(\Lambda)\otimes L^2(Q)$. Since $u_L$ is independent of $z\in\Lambda_L$, we note that for all $z\in\Lambda$ we have
  \begin{equation*}
    \mathcal T_Lu_L(\omega,z,x)=u_L(\tau_{-\lfloor Lz\rfloor}\omega,x).
  \end{equation*}
  Hence,
  \begin{align*}
    &\mathbb E\big[\int_\Lambda\int_Q     \mathcal T_Lu_L(\omega,z,x)\eta(x)\psi(z)\varphi(\omega)\,dx\,dz\big]
      \,=\,\mathbb E\big[\big(\int_Q  u_L(\omega,x)\eta(x)\,dx\big)\int_\Lambda\psi(z)\varphi(\tau_{\lfloor Lz\rfloor}\omega)\big].
  \end{align*}
  By the von-Neumann ergodic theorem, we have $\int_\Lambda\psi(z)\varphi(\tau_{\lfloor Lz\rfloor}\omega)\,dz\to \mathbb E[\varphi]\int_\Lambda\psi(z)\,dz$ strongly in $L^2(\Omega)$, while $\int_Q  u_L\eta\,dx\wto \int_Q \tilde u\eta\,dx$ weakly in $L^2(\Omega)$. Hence, the right-hand side converges to $\mathbb E\big[\int_\Lambda\int_Q \mathbb E[\tilde u]\eta\psi\varphi\,dx\,dz\big]$. On the other hand, we may also pass to the limit on the left-hand side, since $\mathcal T_Lu_L\wto v$. We conclude that $v=\mathbb E[\tilde u]=:u$.
\end{proof}

\vspace{5pt}

\subsection{Proof of Theorem~\ref{T:hyst_hom}}
By Theorem~\ref{T:energetic_solution} the energetic solution $y_\hom$ of $(Y_\hom,\mathcal{E}_\hom,\rcal_\hom)$  satisfies for a.a.~$t\in[0,T]$ the force balance equation
\begin{equation}\label{eq:FBE_hom}
  0 \in D_{y}\E_\hom(t,y_\hom(t))+ \partial \rcal_\hom(\dot{y}_\hom(t)).
\end{equation}
In the following we use the notation $y_\hom=(u_\hom,p_\hom,\chi_{s,\hom})=(u_\hom,\hat y_\hom)$ and the factorization 
\begin{equation*}
  Y_\hom=H^1_0(Q)\times \widetilde Y_\hom,\qquad \widetilde Y_\hom:=L^2(\Omega)\otimes L^2_0(Q)^k\times L^2_s(\Omega)\otimes L^2(Q).
\end{equation*}

By definition, $\rcal_\hom(\dot y_\hom)$ only depends on the $\dot{\hat y}_\hom$-component, and thus
\begin{equation*}
  \partial\rcal_\hom(\dot y_\hom(t))=(0,\partial\widetilde{\rcal}_\hom(\dot{\hat y}_\hom(t)))^\top,
\end{equation*}
where $\widetilde{\rcal}_\hom:\widetilde Y_\hom\to[0,\infty]$ can be expressed by integrating the dissipation functional $\widehat{\rcal}_\hom$ introduced in \eqref{def:hat_ERIS}:
\begin{equation*}
  \widetilde{\rcal}_\hom(\dot{\hat y}_\hom)\colonequals  \widetilde{\rcal}_\hom(\dot p_\hom,\dot\chi_{s,\hom})\colonequals\int_Q\widehat{\rcal}_{\hom}\big(\dot{p}_\hom(\cdot,x)\big)\,dx.
\end{equation*}
Since the first component of $\partial\rcal_\hom(\dot y_\hom(t))$ is zero, \eqref{eq:FBE_hom} takes the form
\begin{subequations}
  \begin{alignat}{2}
    \label{eq:ERIS_hom_a}0 & = D_{u}\mathcal{E}_\hom(t, y_\hom(t)),\\
    \label{eq:ERIS_hom_b}0 & \in \partial\widetilde{\rcal}_\hom(\dot{\hat y}_\hom(t))+ D_{\hat y}\widetilde{\mathcal{E}}_\hom(t, \hat y_\hom(t)),
  \end{alignat}
\end{subequations}
where $\widetilde{\mathcal{E}}_\hom$ can  be expressed by integrating the energy functional $\widehat{\mathcal E}_{\hom}(\cdot;F(\cdot,x))$ of \eqref{def:hat_ERIS} with $F(t,x)={\rm sym}\nabla u_\hom(t,x)$:
\begin{align*}
  &\widetilde{\mathcal{E}}_\hom(t, \hat y_\hom(t)):=\widetilde{\mathcal{E}}_\hom(t, p_\hom(t),\chi_{s,\hom}(t))\\
  &\qquad :=\int_Q\widehat{\mathcal E}_{\hom}\Big(t,p_\hom(t,x),\chi_{s,\hom}(t,x);F(\cdot,x)\Big)\,dx.
\end{align*}
Testing \eqref{eq:ERIS_hom_a} with $\eta\in H^1_0(Q)$, we deduce that
\begin{align*}
  0=&\expect{D_{u}\E_\hom(t,y_\hom(t)),\eta}\\
    &\qquad =
  \int_Q\mathbb E\Big[A(\omega)
  \begin{pmatrix}
    P_s\nabla u_\hom(t,x)+    \chi_{s,\hom}(t,\omega,x)\\
    p_\hom(t,\omega,x)
  \end{pmatrix}\cdot
  \begin{pmatrix}
    P_s\nabla\eta(x)\\0
  \end{pmatrix}
  \Big]-l(t,x)\cdot\eta(x)\,dx\\
    &\qquad =\int_Q\sigma_\hom(t,x)\cdot\nabla\eta\,dx-\int_Ql(t,x)\cdot\eta(x)\,dx,
\end{align*}
with
\begin{equation*}
  \sigma_\hom(t,x):=(P_s^{*}\circ\pi_k)\left(\mathbb E\Big[A(\omega)
    \begin{pmatrix}
      P_s\nabla u_\hom(t,x)+\chi_{s,\hom}(t,\omega,x)\\
      p_\hom(t,\omega,x)
    \end{pmatrix}\Bigg]\right).
\end{equation*}
Hence, it remains to show that
\begin{equation}\label{eq:sigma_Whom}
  \sigma_\hom(t,x)=\mathcal W_\hom[{\rm sym}\nabla u_\hom(\cdot,x)](t).
\end{equation}
To that end, for $x\in Q$ let $\hat y(\cdot,x)$ denote the energetic solution with $\hat y(0,x)=0$ of the ERIS $(\widehat Y_{\hom},\widehat{\mathcal E}_{\hom}(\cdot;F(\cdot,x)),\widehat{\rcal}_\hom)$ of Section~\ref{S:hyst}.
Then $\hat y(\cdot,x)$ satisfies the force balance equation
\begin{equation*}
  0\in\partial\widehat{\rcal}_\hom(\dot{\hat y}(t,x))+D_{\hat y}\widehat{\mathcal E}_\hom(\hat y(t,x);F(t,x))\qquad\text{for a.a.~}t\in[0,T]\text{ and a.e. }x\in Q.
\end{equation*}
By integration in $x$ we obtain \eqref{eq:ERIS_hom_b}, and we deduce that $\hat y=\hat y_\hom$ by appealing to the uniqueness of energetic solutions.
Hence, thanks to the definition of $\mathcal W_\hom$ and $\sigma_\hom$, \eqref{eq:sigma_Whom} follows.
\qed

\subsection{Proof of Theorem~\ref{thm:1462:d}}
  The idea of the argument is to prove in Step~1--Step~4 that $y_L$ (after a suitable transformation) converges for $L\to\infty$ to a solution $w$ of an auxiliary ERIS $(\widetilde Y,\widetilde\E,\widetilde\rcal)$.
  The state space of the auxiliary ERIS is given by
  \begin{equation*}
    \widetilde{Y}=H^1_0(Q)^d\times \brac{L^2(\Omega)\otimes L^2(\Lambda)\otimes L^2(Q)}^k \times \brac{H^1_{\mathrm{per,av}}(\Lambda)\otimes L^2(Q)}^d\times \brac{L^2_s(\Omega)\otimes L^2(\Lambda)\otimes L^2(Q)},
  \end{equation*}
  and we use the notation $w=(u,p,\varphi,\chi_s)$ for the state variable.
  The energy functional is given by $\widetilde{\E}:[0,T]\times \widetilde Y\to \R$,
  \begin{equation*}
    \widetilde{\E}(t,w)=\mathbb E\Big[\int_{\Lambda}\int_{Q}A(\omega)\binom{P_s(\nabla u+\nabla_{z}\varphi)+\chi_s}{p}\cdot\binom{P_s(\nabla u+\nabla_{z}\varphi)+\chi_s}{p}\,dx\,dz\Big]-\int_{Q}l(t)\cdot u\, dx,
  \end{equation*}
  and the dissipation functional $\widetilde{\rcal}:\widetilde{Y}\to [0,\infty]$ is defined as
  \begin{equation*}
    \widetilde{\rcal}(\dot{w})=\expect{\int_{\Lambda}\int_{Q}\rho(\omega,\dot{p})\,dx\,dz}.
  \end{equation*}
  We note that $(\widetilde Y,\widetilde{\E},\widetilde{\rcal})$ is a quadratic ERIS in the sense of Definition~\ref{D:ERIS} as can be easily checked.
  We shall see that the auxiliary ERIS can be viewed as an extension of the homogenized ERIS $(\widehat Y_\hom, \widehat{\E}_\hom,\widehat{\rcal}_\hom)$.
  This is made precise in Step~6, where we show that energetic solutions to the homogenized ERIS are also energetic solutions to the auxiliary ERIS.
  This enables us to express $w$ (which is the limit of $(y_L)$) in terms $y_\hom$ and the claim of the theorem follows.

  \emph{Step 1 - Compactness.}
  From \eqref{eq:429:p} we conclude that the energetic solution $y_L$ satisfies the a priori estimate,
  \begin{equation}\label{eq:P:1}
    \|y_L(t)\|_{Y_L}\leq C\text{ and }\int_s^t\|y_L(\tau)\|_{Y_L}\,d\tau\leq C\int_s^t|\dot F(\tau)|\,d\tau\text{ for all }0\leq s\leq t\leq T,
  \end{equation}
  for a constant $C$ that is independent of $L$.
  With help of the map $\mathcal T_L$, we transform $y_L$ to a domain that is independent of $L$.
  To that end recall the notation $y_L=(u_L,p_L,\varphi_{L})$ and set
  \begin{equation}\label{eq:vL}
    v_L:=\Big(\mathcal T_L u_L,\mathcal T_L(P_s\nabla u_L),\mathcal T_Lp_L,L^{-1}\mathcal T_L\varphi_L,\mathcal T_L\nabla_s\varphi_L\Big)^\top.
  \end{equation}
  Note that the prefactor $L^{-1}$ accounts for the fact that $\|\cdot\|_{L^2_{\mathrm{av}}(\Lambda)}=\frac1L\|\cdot\|_{L^2(\Lambda_L)}$.
  Since $\mathcal T_L$ is an isometry, we have $v_L\in W^{1,1}((0,T);H)$ where $H:=(L^2(\Omega)\otimes L^2(\Lambda)\otimes L^2(Q))^{d+k+k+d+k}$, and \eqref{eq:P:1} turns into
  \begin{equation*}
    \sup_{t\in[0,T]}\|v_L(t)\|_{H}\leq C,\qquad \|v_L(t)-v_L(s)\|_{H}\leq C\int_s^t|\dot F(\tau)|\,d\tau\text{ for all }0\leq s\leq t\leq T.
  \end{equation*}
  Since $F\in W^{1,1}_o$, the sequence $(v_L)\subset C([0,L];H)$ is bounded and  pointwise equicontinuous.
  Since closed bounded balls in the Hilbert space $H$ are weakly sequentially compact, we may apply the Arzel\'a-Ascoli theorem.
  Thus, arguing as in the proof of \cite[Theorem~3.5.2]{mielke2015rate} we find a limit $v\in C([0,L];H)\cap W^{1,1}((0,L);H)$ and a subsequence (not relabeled) such that for all $t\in[0,T]$ we have
  \begin{equation*}
    v_L(t)\wto v(t)\qquad\text{weakly in }H,
  \end{equation*}
  and
  \begin{equation*}
    \|\dot v(t)\|_H\leq C|\dot F(t)|\text{ for a.a.~}t\in[0,T].
  \end{equation*}
  By applying Lemma~\ref{L:comp_z_indep} to $\big(\mathcal T_Lu_L,\mathcal T_L(P_s\nabla u_L)\big)$ and Lemma~\ref{lem:1204} to $\big(L^{-1}\mathcal T_L\varphi_L,\mathcal T_L(\nabla_{s}\varphi_L)\big)$,
  we see that the components of $v$ have a specific structure:
  \begin{equation}\label{eq:P:1:1}
    v=\big(u,P_s\nabla u,\,p,\,\varphi, P_s\nabla\varphi+\chi_s\big)^\top\qquad\text{with }w\colonequals (u,p,\varphi,\chi_s)\in W^{1,1}((0,T);\widetilde Y),
  \end{equation}
  where $\widetilde Y$ denotes the state space of the auxiliary ERIS.
  Our assumption on the convergence of the initial state $y_L^0$ implies that
  \begin{equation}\label{eq:P1:w0}
    w(0)=(u^0_\hom,p^0_\hom,0,\chi_{s,\hom}^0).
  \end{equation}

  \emph{Step 2. Stability.}
  We claim that the limit $w$ (see \eqref{eq:P:1:1}) is stable, in the sense that
  \begin{equation*}
    w(t)\in\widetilde S(t)\qquad\text{for all }t\in[0,T],
  \end{equation*}
  where $\widetilde S(t)$ denotes the set of stable states of the auxiliary ERIS. 
  For the argument we fix $t\in[0,T]$.
  Let $\tilde w=(\tilde u,\tilde p, \tilde\varphi,\tilde\chi_s)\in \widetilde Y$ be an arbitrary state.
  We need to show that
  \begin{equation*}
    \widetilde{\mathcal E}(t,w(t))\leq \widetilde{\mathcal E}(t,\tilde w)+\widetilde{\rcal}(\tilde w-w(t)).
  \end{equation*}
  In order to do so, we use the recovery sequence construction of Lemma \ref{lem:1204} to find a sequence $\psi_L\in (L^2(\Omega)\otimes L^2(\Lambda_L)\otimes L^2(Q))^d$ such that
  \begin{equation*}
    L^{-1}\mathcal T_L\psi_L\to \tilde\varphi - \varphi(t), \quad \mathcal T_L\nabla_{s,z}\tilde\psi_L \to P_s(\nabla_{z}(\tilde\varphi-\varphi(t)))+\tilde\chi_s-\chi_s(t) \quad \text{strongly in }L^2. 
  \end{equation*}
  Also, we set
  \begin{equation*}
    q_L(\omega,z,x)\colonequals L^d\int_{\frac1L(z+\Lambda)} \tilde p(\tau_z\omega,\tilde z,x)-p(t,\tau_z\omega,\tilde z,x)\,d\tilde z,
  \end{equation*}
  and note that this defines a sequence with  $q_L\in (L^2(\Omega)\otimes L^2(\Lambda_L)\otimes L^2(Q))^k$ and
  \begin{equation*}
    \mathcal T_Lq_L\to \tilde p-p(t) \quad \text{strongly in }L^2. 
  \end{equation*}
  We now define $\tilde y_L=(\tilde u_L,\tilde p_L,\tilde\varphi_L)\in Y_L$ componentwise:
  \begin{equation*}
    \tilde u_L=u_L(t)+\tilde u-u(t),\qquad \tilde p_L=p_L(t)+q_L,\qquad \tilde\varphi_L=\varphi_L(t)+\psi_L,
  \end{equation*}
  and consider
  \begin{equation*}
    \Psi^{\pm}_L:=\mathcal T_L\binom{P_s(\nabla(\tilde u_L\pm u_L(t))+\nabla_{s,z}(\tilde\varphi_L\pm\varphi_L(t)}{\tilde p_L\pm p_L(t)}.
  \end{equation*}
  By construction we have
  \begin{equation}\label{eq:P:1:weak:psi}
    \Psi^-_L\to \Psi^-\text{ strongly in $L^2$ and }    \Psi^+_L\wto \Psi^+\text{ weakly in $L^2$},
  \end{equation}
  where
  \begin{equation*}
    \Psi^{\pm}:=  \binom{P_s(\nabla(\tilde u \pm u(t))+P_s(\nabla_{z}(\tilde\varphi\pm\varphi(t))+\tilde\chi_s\pm\chi_s(t)}{\tilde p\pm p(t)}.
  \end{equation*}
  By expanding the quadratic energy functional $\mathcal E_L$ and by appealing to the isometry property of $\mathcal T_L$, we deduce that
  \begin{align}\label{quadr}
    \mathcal{E}_L(t, \tilde y_L)-\mathcal{E}_L(t,y_L(t))\,=\,\mathbb E\Big[\int_\Lambda\int_Q\tfrac12 A(\omega)\Psi^-_L\cdot\Psi^+_L\,dx\,dz\Big]-\mathbb E\big[\int_Q l(t) \cdot (u_L-u(t))\big].\nonumber
  \end{align}
  In view of \eqref{eq:P:1:weak:psi} we may pass to the limit $L\to \infty$.
  Using the quadratic expansion of $\widetilde{\mathcal E}$ we obtain
  \begin{equation}\label{quad:form}
    \lim\limits_{L\to\infty}\big(\mathcal{E}_L(t, \tilde y_L)-\mathcal{E}_L(t,y_L(t))\Big)\,=\,\mathbb E\Big[\int_\Lambda\int_Q\tfrac12 A(\omega)\Psi^-\cdot\Psi^+\,dx\Big]=\widetilde{\mathcal E}(t,\tilde w)-\widetilde{\mathcal E}(t,w(t)).
  \end{equation}
  Next, we consider the dissipation functional.
  By definition of $\rcal_L$, $\tilde y_L$, and $q_L$ we have
  \begin{equation*}
    \rcal_L(\tilde y_L-y_L(t))=\mathbb E\big[L^{-d}\sum_{z\in\Lambda_L}\int_Q\rho\Big(\omega, \dashint_{\frac{z+\Lambda}{L}}\delta p(\omega,\tilde z,x)\Big)\,d\tilde z\,dx\big],
  \end{equation*}
  where $\delta p(\omega,\tilde z,x)\colonequals \tilde p(\omega,\tilde z,x)-p(t,\omega,\tilde z,x)$.
  Since $\rho(\omega,\cdot)$ is convex, Jensen's inequality yields
  \begin{equation}\label{eq:P:1:rcal}
    \begin{aligned}
    \rcal_L(\tilde y_L-y_L(t))\leq&\,\mathbb E\big[L^{-d}\sum_{z\in\Lambda_L}\dashint_{\frac{z+\Lambda}{L}}\int_Q\rho\Big(\omega,\delta p(\omega,\tilde z,x)\Big)\,\,dxd\tilde z\big]\\
    =&\,\widetilde{\rcal}(\tilde w-w(t)).
  \end{aligned}
  \end{equation}
  Since $y_L$ is a energetic solution, we have $\rcal_L(\tilde y_L-y_L(t))\geq \mathcal E_L(t,y_L(t))-\mathcal E_L(t,\tilde y_L)$.
  Combined with \eqref{eq:P:1:rcal} and \eqref{quad:form} we obtain
  \begin{equation*}
    \widetilde\rcal(\tilde w-w(t))\geq \widetilde{\mathcal E}(t,\tilde w)-\widetilde{\mathcal E}(t,w(t)),
  \end{equation*}
  and thus $w(t)\in\widetilde S(t)$.

  \emph{Step 3. Energy balance.}
  We claim that the limit $w$ of Step~1 satisfies the global energy balance inequality
  \begin{equation}\label{eb:final}
    \widetilde{\mathcal{E}}(t,w(t))+\int_0^t \widetilde{\rcal}(\dot{w}(s))ds \le \widetilde{\mathcal{E}}(0,w(0))-\int_0^t \int_Q \dot{l}(s)\cdot u(s) ds.
  \end{equation}
  Since $y_L$ is an energetic solution, it satisfies the energy balance inequality
  \begin{equation*}
    \mathcal{E}_L(t,y_L(t))+\int_0^t \rcal_L(\dot{y}_L(s))ds \le \mathcal{E}_L(0,y_L(0))-\int_0^t \mathbb E\Big[\int_Q \dot{l}(s)\cdot u_L(s)\Big] ds.   
  \end{equation*}
  Hence, to conclude \eqref{eb:final} it suffices to prove
  \begin{align}
    \label{eq:EB:1}
    &\liminf\limits_{L\to\infty}\mathcal E_L(t,y_L(t))\geq \widetilde{\mathcal E}(t,w(t)),\\
    \label{eq:EB:2}
    &\liminf\limits_{L\to\infty}\int_0^t\rcal_L(\dot y_L(s))\,ds\geq \int_0^t\widetilde{\rcal}(\dot w(s))\,ds,\\
    \label{eq:EB:3}
    &\lim\limits_{L\to\infty}\Big(\mathcal E_L(0,y_L(0))-\int_0^t\mathbb E\Big[\int_Q\dot l(s)\cdot u_L(s)\Big]\,ds\Big)
    =\widetilde{\mathcal E}(0,w(0))-\int_0^t\Big[\int_Q\dot l(s)\cdot u(s)\Big]\,ds.
  \end{align}
  We start with \eqref{eq:EB:1}.
  With help of the operator $\mathcal T_L$ and with the shorthand notation
  \begin{equation*}
    \Phi_L(t)\colonequals \mathcal T_L\binom{P_s(\nabla u_L(t))+\nabla_{s,z}\varphi_L(t)}{p_L(t)},
  \end{equation*}
  we have
  \begin{equation}\label{eq:P:1:EBEL2}
    \mathcal E_L(t,y_L(t))=\mathbb E[\int_{\Lambda}\int_Q\frac12 A(\omega)\Phi_L(t)\cdot\Phi_L(t)\,dx\,dz\Big]-\mathbb E\Big[\int_{\Lambda}\int_Ql(t)\cdot \mathcal T_Lu_L(t)\Big].
  \end{equation}
  By Step~1 we have
  \begin{equation}\label{eq:P:1:EBEL2b}
    \mathcal T_Lu_L(t)\wto u(t),\qquad \Phi_L(t)\wto \Phi(t)\colonequals \binom{P_s(\nabla u(t)+\nabla_{z}\varphi(t))+\chi_s(t)}{p(t)}\qquad\text{weakly in }L^2,
  \end{equation}
  where $u(t)$ is independent of $\omega$ and $z$.
  Hence, by weak lower-semicontinuity of convex integral functionals we obtain
  \begin{equation*}
    \liminf\limits_{L\to\infty}\mathcal E_L(t,y_L(t))\geq \mathbb E[\int_{\Lambda}\int_Q\frac12 A(\omega)\Phi(t)\cdot\Phi(t)\,dx\,dz\Big]-\int_Ql(t)\cdot u(t)=\widetilde{\mathcal E}(t,w(t)).
  \end{equation*}
  The argument for \eqref{eq:EB:3} is similar: By assumption on the initial data, we have $\Phi_L(0)\to\Phi(0)$ strongly $L^2$, which upgrades \eqref{eq:P:1:EBEL2b}.
  As a consequence we may pass in \eqref{eq:P:1:EBEL2} to the limit (and not only the limit inferior) and thus conclude \eqref{eq:EB:3}.

  We finally discuss \eqref{eq:EB:2}.
  To that end we consider a (finite) partition  $\cb{t_i}_{i\in I}$ of $[0,t]$ and claim that
  \begin{equation}\label{eq:EB:2:1}
    \sum_{i\in I} \widetilde{\rcal}(w(t_{i})-w(t_{{i-1}}))\leq \liminf_{L\to\infty} \sum_{i\in I}\rcal_L(y_L(t_{i})-y_L(t_{{i-1}})).
  \end{equation}
  To see this, we first note that thanks to the definition of $\mathcal T_L$ we have
  \begin{equation*}
    \rcal_L(y_L(t_{i})-y_L(t_{{i-1}}))=\mathbb E\Big[\int_\Lambda\int_Q\rho\Big(\omega,\mathcal T_L(p_L(t_i)-p_L(t_{i-1})\Big)\Big].
  \end{equation*}
  Furthermore, from Step~1 we deduce that $\mathcal T_L(p_L(t_i)-p_L(t_{i-1}))\wto p(t_i)-p(t_i)$ weakly $L^2$.
  Hence, \eqref{eq:EB:2:1} follows from the weak lower semicontinuity of the convex integral functional $\mathbb E\big[\int_{\Lambda}\int_Q \rho(\omega,\cdot)\big]$.
  Taking the supremum over all partitions $\cb{t_i}_{i\in I}$ of $[0,t]$ in \eqref{eq:EB:2:1}, and by exploiting the 1-homogeneity of $\widetilde{\rcal}$ and the growth condition from \ref{B2:pl:d}, we obtain \eqref{eq:EB:2}.
  \medskip

  \emph{Step 4. Energetic solution and uniqueness}.
  We may combine the energy balance inequality \eqref{eb:final} with the stability of $w$ obtained in Step~2.
  By appealing to the general argument of \cite[Theorem 4.4]{mainik2005existence} we conclude that \eqref{eb:final} also holds with $\leq$ replaced by $\geq$.
  Hence, $w$ satisfies the global energy balance equality \eqref{eq:413:p} and we conclude that $w$ is an energetic solution of the ERIS  $\brac{\widetilde{Y},\widetilde{\E},\widetilde{\rcal}}$ with $w(0)=\brac{u^0,p^0,0,\chi^0}$.
  Since solutions to the ERIS are unique, we conclude that convergences of Step~1 hold for the entire sequence.
  Furthermore, we conclude convergence of the energy
  \begin{equation}\label{eq:669:a}
    \E_L(t,y_L(t))\to \widetilde{\E}(t,w(t)).
  \end{equation}
  by an argument similar to \cite[Proof of Theorem 4.10]{neukamm2017stochastic}.
  \medskip

  \emph{Step 5. Strong convergence.}
  From Step~1 recall the definition of $v_L$ (see \eqref{eq:vL}) and the identities
  \begin{equation*}
    v=(u,P_s\nabla u,p,\varphi,P_s\nabla_z\varphi+\chi_s)^\top\text{ and }w=(u,p,\varphi,\chi_s).
  \end{equation*}
  We claim that for all $t\in[0,T]$ we have
  \begin{equation}\label{eq:strong_vL}
    \lim\limits_{L\to\infty}\|v_L(t)-v(t)\|_H=0.
  \end{equation}
  Here comes the argument: Proceeding as in Step~3, we find for $w(t)$ a strong recovery sequence $\tilde y_L=\brac{\tilde u_L,\tilde p_L,\tilde\varphi_L} \in Y_L$ such that $$\tilde v_L=\brac{\mathcal T_L\tilde u_L,\mathcal T_L(P_s\nabla \tilde u_L),\mathcal T_L\tilde p_L,L^{-1}\mathcal T_L\tilde\varphi_L,\mathcal T_L(\nabla_{s,z}\tilde\varphi_L)}^\top$$ satisfies
\begin{equation*}
  \tilde v_L\to v(t) \quad \text{strongly in }H.
\end{equation*}
By the triangle inequality, we have 
\begin{equation}\label{inequality123}
  \|v_L(t) - v(t)\|_H \leq \|v_L(t) - \tilde v_L\|_H+\|\tilde v_L - v(t)\|_H.  
\end{equation}
The second term on the right-hand side vanishes in the limit $L\to\infty$.
Also, since the energy is coercive, we obtain, using the isometry property of $\mathcal T_L$ and a discrete Poincar{\'e}-Korn inequality,
\begin{equation}\label{eq:P1:331}
  \frac1{c} \|v_L(t) - \tilde v_L\|^2_H  \leq  \mathcal{E}_L(t,y_L(t)-\tilde y_L)+\mathbb E\Big[\int_Ql(t)\cdot (u_L(t)-\tilde u_L)\,dx\Big]
\end{equation}
with a constant $c$ that is independent of $L$.
By expanding the quadratic term in the energy functional in the sense of the binomial identity ``$(a-b)^2=a^2-b^2-2b(a-b)$'', we deduce that
\begin{equation}\label{eq:P1:332}
  \mathcal{E}_L(t,y_L(t)-\tilde y_L)=\mathcal{E}_L(t,y_L(t))-\mathcal{E}_L(t,\tilde y_L)-\mathbb E\Big[\int_\Lambda\int_QA(\omega)\widetilde\Psi_L\cdot(\Psi_L-\widetilde\Psi_L)\,dx\,dz\Big],
\end{equation}
where
\begin{equation*}
  \widetilde\Psi_L\colonequals \mathcal T_L\binom{P_s\nabla\tilde u_L+\nabla_{s,z}\tilde\varphi_L}{\tilde p_L},\qquad \Psi_L\colonequals\mathcal T_L\binom{P_s\nabla u_L(t)+\nabla_{s,z}\varphi_L(t)}{p_L(t)}
\end{equation*}
We claim that for $L\to\infty$ the right-hand side of \eqref{eq:P1:332} converges to $0$.
By construction, as $L\to\infty$, the sequence $A(\omega)\widetilde\Psi_L$ strongly converges in $L^2$, while $\Psi_L-\widetilde\Psi_L\wto 0$ weakly in $L^2$.
Hence, the last term on the right-hand side of \eqref{eq:P1:332} converges to $0$.
Similarly, the second term on the right-hand of \eqref{eq:P1:332} can be expressed as
\begin{equation*}
  \mathcal{E}_L(t,\tilde y_L)=\mathbb E\Big[\int_\Lambda\int_Q\tfrac12 A(\omega)\widetilde\Psi_L\cdot \widetilde\Psi_L\,dx\,dz\Big]-\mathbb E\Big[\int_Ql(t)\cdot\tilde u_L\Big],
\end{equation*}
and thus converges to $\widetilde{\E}(t,w(t))$.
Since the latter is also the limit of the first term on the right-hand side of \eqref{eq:P1:332}, see \eqref{eq:669:a}, we conclude that \eqref{eq:P1:332} indeed vanishes for $L\to\infty$.
Hence, \eqref{eq:P1:332} and the convergence $u_L(t)-\tilde u_L\wto 0$ implies that $\|v_L(t)-v(t)\|_H\to 0$.
\medskip

\emph{Step 6. Conclusion.}
To conclude the proof we need to exploit the relation between the auxiliary and the homogenized ERIS. To that end we consider the map ,
\begin{equation*}
  \iota:Y_{\hom}\to \widetilde Y,\qquad  \iota(\bar u, \bar p,\bar\chi_s):=(\bar u,\bar p,0,\bar\chi_s),
\end{equation*}
and note that this defines a linear (nonsurjective) isometry. From the definition of $\widetilde{\E}$ it directly follows that
\begin{equation}\label{eq:identification}
  \E_{\hom}(t,\bar y)=\widetilde{\E}(t,\iota(\bar y))\text{ and }\rcal_{\hom}(\bar y)=\widetilde\rcal(\iota(\bar y))\text{ for all $\bar y\in Y_{\hom}$ and all $t\in[0,T]$}.
\end{equation}
Next, we claim that $y_{\hom}^0=(u_\hom^0,p_\hom^0,\chi_{s,\hom}^0)$ is a stable state for the homogenized ERIS. Indeed, by \eqref{eq:P1:w0} and Step~2 we have $\iota(y_\hom^0)=w(0)\in\widetilde S(0)$ (where the latter denotes the set of stable states of the auxiliary ERIS). In view of \eqref{eq:identification} this implies that $y_{\hom}^0\in S_{\hom}(0)$. Since $y_{\hom}^0$ is a stable initial state, the homogenized ERIS  $\brac{Y_\hom,\E_\hom,\rcal_\hom}$ admits an energetic solution $y_\hom=(u_\hom,p_\hom,\chi_{s,\hom})$ with $y_{\hom}(0)=y^0_{\hom}$. We claim that $w_\hom:=\iota(y_{\hom})$ is an energetic solution to the auxiliary ERIS  $\brac{\widetilde{Y},\widetilde{\E},\widetilde{\rcal}}$ with $w_\hom(0)=\iota(y_{\hom}(0))$.

We first note that the global energy balance equality for $y_{\hom}$ combined with \eqref{eq:identification} implies that $w_{\hom}$ satisfies the global energy balance equality as well, i.e.,
\begin{equation*}
  \widetilde{\E}(t,w_\hom(t))+\int_0^t\widetilde{\rcal}(\dot w_\hom(s))\,ds=  \widetilde{\E}(0,w_\hom(0))-\int_0^t\int_Q\dot l(s)\cdot u_{\hom}(s)\,dx\,ds.
\end{equation*}
Hence, it remains to check stability, i.e., $w_\hom(t)\in \widetilde S(t)$ for all $t\in[0,T]$. To that end let $\tilde w=(\tilde u,\tilde p,\tilde\varphi,\tilde\chi_s)\in \widetilde Y$ be an arbitrary state. Then $\bar y\colonequals\brac{\tilde u,\int_\Lambda\tilde p\,dz,\int_\Lambda\tilde\chi_s\,dz}\in Y_\hom$ and since $y_{\hom}(t)\in S_{\hom}(t)$ we have
\begin{equation}\label{eq:Stabhom}
  \E_\hom(t,y_{\hom}(t))\leq \,\E_{\hom}(t,\bar y)+\rcal_\hom(\bar y-y_{\hom}(t)).
\end{equation}
Note that by periodicity of $\tilde\varphi$ in $z$ we have
\begin{equation}\label{eq:integrate1231213}
  \int_\Lambda\tilde\chi_s\,dz=\int_\Lambda P_s\nabla_z\tilde\varphi+\tilde\chi_s\,dz,
\end{equation}
and thus 
\begin{align*}
  \E_{\hom}(t,\bar y)\,=\,&\mathbb E\Big[\int_Q\tfrac12 A(\omega)\binom{P_s\nabla\tilde u+\int_\Lambda P_s\nabla\tilde\varphi+\tilde\chi_s\,dz}{\int_\Lambda\tilde p\,dz}\cdot\binom{P_s\nabla\tilde u+\int_\Lambda P_s\nabla\tilde\varphi+\tilde\chi_s\,dz}{\int_\Lambda\tilde p\,dz}\Big]\\&\,-\int_Ql(t)\cdot \tilde u\,dx
  \,\leq\,\widetilde{\E}(t,\tilde w(t)),
\end{align*}
where the last bound holds by Jensen's inequality. Likewise, by Jensen's inequality we have
\begin{equation*}
  \rcal_\hom(\bar y-y_{\hom}(t))=\mathbb E\Big[\int_Q\rho\big(\omega, \int_\Lambda\tilde p-p_\hom(t)\,dz\big)\,dx\Big]\leq \widetilde\rcal(\tilde w-w_\hom(t)).
\end{equation*}
By combining the previous two bounds with \eqref{eq:Stabhom} and the identity $\widetilde{\E}(t,w_{\hom}(t))=\widetilde{\E}(t,\iota(y_{\hom}(t))=\E_{\hom}(t,y_{\hom}(t))$ (cf.~\eqref{eq:identification}), we conclude $w_{\hom}(t)\in\widetilde S(t)$.

Finally, since solutions to the auxiliary ERIS are unique and since $w_{\hom}(0)=w(0)$, we conclude that $w=w_\hom=\iota(y_{\hom})$. Combined with the strong convergence \eqref{eq:strong_vL} we get
\begin{align*}
  &\|\mathcal T_L(u_L(t))-u_{\hom}(t)\|+  \|\mathcal T_L(P_s\nabla u_L(t))-P_s\nabla u_{\hom}(t)\|\\
  &\qquad +\|\mathcal T_Lp_L(t)-p_{\hom}(t)\|+\|\mathcal T_L\varphi_L(t)\|+\|\mathcal T_L\nabla_{s,z}\varphi_L(t)-\chi_{s,\hom}(t)\|\to 0,
\end{align*}
where $\|\cdot\|\colonequals\|\cdot\|_{L^2(\Omega)\otimes L^2(\Lambda)\otimes L^2(Q)}$. Since $u_{\hom}(t)$ only depends on $x\in Q$ and $u_L(t)$ only on $(\omega,x)\in\Omega\times Q$, we conclude that $\|\mathcal T_L(u_L(t))-u_{\hom}(t)\|+  \|\mathcal T_L(P_s\nabla u_L(t))-P_s\nabla u_{\hom}(t)\|=\|u(t)-u_{\hom}(t)\|_{L^2(\Omega)\otimes H^1(Q)}$, and thus the the claimed convergences follow.
\qed

\subsection{Proof of Theorem~\ref{T:RVE-W}}
We first introduce an ERIS $(\widehat Y_L,\widehat{\E}_L(\cdot;F),\widehat{\rcal}_L)$ which is the  \emph{mean} version of the \emph{quenched} ERIS $(\widehat Y_L^\omega,\widehat{\E}_L^\omega(\cdot;F),\widehat{\rcal}_L^\omega)$. The latter is used to define $\mathcal W^\omega_L$. The mean ERIS is obtained from the quenched one by extending the state space of the quenched ERIS to $L^2(\Omega)\otimes \widehat{Y}^\omega_L$, and by taking the expectation of the energy and dissipation functional. More precisely, we define
  \begin{align*}
    \widehat Y_L\,\colonequals\,&(L^2(\Omega)\otimes L^2(\Lambda_L))^k\times(L^2(\Omega)\otimes L^2_{\mathrm{av}}(\Lambda_L))^d,\\
    \widehat\E_L(t,\hat y_L;F)\,\colonequals\,&\mathbb E\Big[\widehat\E_L^\omega(t,\hat y_L(\omega,\cdot;F)\Big],\\
    \widehat\rcal_L(\hat y_L)\,\colonequals\,&\mathbb E\Big[\widehat{\rcal}_L^\omega(\hat y_L(\omega,\cdot),
  \end{align*}
  where we use the notation $\hat y_L=(\hat p_L,\hat\varphi_L)$ for state variables in $\widehat Y_L$.
  One can easily check that $(\widehat Y_L,\widehat{\E}_L(\cdot;F),\widehat{\rcal}_L)$ is indeed a quadratic ERIS in the sense of Definition~\ref{D:ERIS}.
  In the following we denote by
  \begin{subequations}
  \begin{align}
    &\text{$\hat y^\omega_L=(\hat p^\omega_L,\hat\varphi^\omega_L)$ the energetic solution with $\hat y^\omega_L(0)=0$ to $(\widehat Y_L^\omega,\widehat{\E}_L^\omega(\cdot;F),\widehat{\rcal}_L^\omega)$,}\\
    &\text{$\hat y_L=(\hat p_L,\hat\varphi_L)$ the energetic solution with $\hat y_L(0)=0$ to $(\widehat Y_L,\widehat{\E}_L(\cdot;F),\widehat{\rcal}_L)$,}\\
    \label{eq:haty}
    &\text{$\hat y_\hom=(\hat p_\hom,\hat\chi_{s,\hom})$ the energetic solution with $\hat y_\hom(0)=0$}\\\nonumber
    &\qquad\qquad\qquad\qquad\qquad\text{ to $(\widehat Y_\hom,\widehat{\E}_\hom(\cdot;F),\widehat{\rcal}_\hom)$ see~Definition~\ref{D:Whom}.}    
  \end{align}
\end{subequations}
  Note that by construction and the uniqueness of energetic solutions we have  $\hat y_L(t,\omega)=\hat y_L^\omega(t)$ for $P$-a.e.~$\omega\in\Omega$ and all $t$. Hence, by the definitions of $\mathcal W^\omega_L[F]$ and $\mathcal W_{\hom}[F]$, and since $P_s^{*}\circ\pi_k:\R^{2k}\to\R^{d\times d}_{\rm sym}$ is linear and bounded, the claim of the theorem can be reformulated as
  \begin{align}\label{eq:P2:claim}
    &\lim\limits_{L\to\infty}\mathbb E\Big[\Big|L^{-d}\sum_{z\in\Lambda_L}\Xi_L(z)-\mathbb E\Big[\Xi_{\hom}\Big]\Big|^2\Big]=0,\qquad\text{where }\\\nonumber
    &\Xi_L(\omega,z)\colonequals A(\tau_z\omega)
    \binom{P_s(F(t))+\nabla_s \hat\varphi_L(t)}{\hat p_L(t)},\quad \Xi_{\hom}(\omega)\colonequals A(\omega)
    \binom{P_s(F(t))+\hat\chi_{s,\hom}(t,\omega)}{\hat p_{\hom}(t,\omega)}.
  \end{align}
  To prove \eqref{eq:P2:claim} we first show in Step~1 (closely following the proof of Theorem~\ref{thm:1462:d}) that $\hat y_L$ converges (after some transformation) to the energetic solution $\hat y_{\hom}$ of $(\widehat Y_\hom,\widehat{\E}_\hom(\cdot;F),\widehat{\rcal}_\hom)$.
  From that we deduce in Step~2 that $L^{-d}\sum_{z\in\Lambda_L}\Xi_L\wto \mathbb E\Big[\Xi_{\hom}\Big]$ weakly in $L^2(\Omega)$. 
  We then exploit the ergodic theorem to upgrade the convergence to strong convergence in $L^2(\Omega)$, which completes the proof.
  \medskip

  \noindent
  \emph{Step 1. Convergence of $\hat y_L$ to $\hat y_{\hom}$.}
  Consider
  \begin{equation*}
    v_L:=\big(\widehat{\mathcal T}_L\hat p_L,L^{-1}\widehat{\mathcal T}_L\hat\varphi_L,\widehat{\mathcal T}_L\nabla_s\hat\varphi_L\big)^\top,
  \end{equation*}
  which is a function in $W^{1,1}((0,T);H)$ with $H=L^2(\Omega)\otimes L^2(\Lambda))^{k+d+k}$. We claim that
  \begin{equation}\label{eq:P2:1}
    v_L(t)\to (\hat p_\hom(t),0,\hat\chi_{s,\hom}(t))\qquad\text{strongly in }H\text{ for all }t\in[0,T],
  \end{equation}
  where $\hat p_\hom$ and $\hat\chi_{s,\hom}$ are the components of the energetic solution $\hat y_\hom$, see \eqref{eq:haty}. This can been seen by following step by step the proof of Theorem~\ref{thm:1462:d}, where as auxiliary ERIS $(\widetilde Y,\widetilde{\E},\widetilde{\rcal})$ we consider
  \begin{align*}
    \widetilde Y\colonequals\,&\,(L^2(\Omega)\otimes L^2(\Lambda))^k\times H^1_{\mathrm{per,av}}(\Lambda)^d\times L^2_s(\Omega)\otimes L^2(\Lambda),\\
    \widetilde{\E}(t,\tilde y;F)\colonequals\,&\,\mathbb E\Big[\int_\Lambda\tfrac12A(\omega)\binom{P_s\nabla\tilde\varphi+\tilde\chi_s}{\tilde p}\cdot\binom{P_s\nabla\tilde\varphi+\tilde\chi_s}{\tilde p}\,dx\Big]\\
                              &\qquad-\mathbb E\Big[\int_\Lambda A(\omega)\binom{P_s\nabla\tilde\varphi+\tilde\chi_s}{\tilde p}\cdot\binom{P_sF(t)}{0}\,dx\Big],\\
    \widetilde\rcal(\dot{\tilde y})\colonequals\,&\,\mathbb E\Big[\int_\Lambda\rho(\omega,\dot{\tilde p})\,dx\Big],
  \end{align*}
  with $\tilde y=(\tilde p,\tilde\varphi,\tilde\chi_s)$.
  Since the required modifications are obvious, we leave the details to the reader to avoid repetition.
  \medskip

  \noindent
  \emph{Step 2. Convergence of $\Xi_L$.}
  Consider
  \begin{equation*}
    \Psi_L(\omega,z)\colonequals\binom{\nabla_s\hat \varphi_L(t,\omega,z)}{\hat p_{L}(t,\omega,z)},\qquad \Psi_{\hom}(\omega) \colonequals\binom{\hat\chi_{s,\hom}(t,\omega)}{\hat p_\hom(t,\omega)}.
  \end{equation*}
  Then \eqref{eq:P2:1} implies that $\widehat{\mathcal T}_L\Psi_L\to \Psi_{\hom}$ strongly in $L^2(\Omega)\times L^2(\Lambda)^{2k}$.
  Consider $\overline\Psi_{\hom}:\Omega\times\Z^d\to\R^{2k}$, $\overline\Psi_{\hom}(\omega,z):=\Psi_{\hom}(\tau_z\omega)$, and note that $\widehat{\mathcal T}_L\overline\Psi_\hom(\omega,x)=\Psi_\hom(\omega)$. Hence, by the isometry property of $\widehat{\mathcal T}_L$ we conclude that
  \begin{equation}\label{eq:P2:111}
    \|\Psi_L-\overline\Psi_\hom\|_{L^2(\Omega)\otimes L^2(\Lambda_L)}=\|\widehat{\mathcal T}_L(\Psi_L-\overline\Psi_{\hom})\|^2_{L^2(\Omega)\otimes L^2(\Lambda)}=\|\widehat{\mathcal T}_L\Psi_L-\Psi_{\hom}\|^2_{L^2(\Omega)\otimes L^2(\Lambda)}\to 0.
  \end{equation}
  By the definition of $\Xi_L$ we have
  \begin{equation}\label{eq:P2:decomp}
    \begin{aligned}
    \frac{1}{L^{d}}\sum_{z\in\Lambda_L}\Xi_L(\omega,z)=\,&\,\frac{1}{L^d}\sum_{z\in\Z^d}A(\tau_z\omega)\Psi_L(\omega,z)\\
                                                         =\,&
                                                            \frac{1}{L^d}\sum_{z\in\Z^d}A(\tau_z\omega)\overline\Psi_\hom(\omega,z)+\frac{1}{L^d}\sum_{z\in\Z^d}A(\tau_z\omega)\big(\Psi_L(\omega,z)-\overline\Psi_\hom(\omega,z)\big).
                                                          \end{aligned}
  \end{equation}
  Since $(\omega,z)\mapsto A(\tau_z\omega)\overline\Psi_{\hom}(\omega,z)$ is stationary, the ergodic Theorem~\ref{T:ergodic}) implies
  \begin{equation*}
    \frac{1}{L^2}\sum_{z\in\Z^d}A(\tau_z\omega)\overline\Psi_\hom(\omega,z)\to \mathbb E\Big[A(\tau_z\omega)\overline\Psi_\hom(\omega,z)\Big]=\mathbb E\Big[\Xi_\hom\Big],
  \end{equation*}
  strongly in $L^2(\Omega)$. Thus, it remains to show that the second term on the right-hand side of \eqref{eq:P2:decomp} converges to $0$ in $L^2(\Omega)$. But this directly follows from \eqref{eq:P2:111}, since
  \begin{equation*}
    \mathbb E\Big[\Big|\frac{1}{L^d}\sum_{z\in\Z^d}A(\tau_z\omega)\big(\Psi_L(\omega,z)-\overline\Psi_\hom(\omega,z)\big)\Big|^2\Big]\leq C\frac{1}{L^d}\sum_{z\in\Z^d}    \mathbb E\Big[|\Psi_L(\omega,z)-\overline\Psi_{\hom}(\omega,z)|^2\Big],
  \end{equation*}
  where $C>0$ is a constant only depending on $A$, and since the term on the right-hand is $C\|\Psi_L-\overline\Psi_{\hom}\|_{L^2(\Omega)\otimes L^2(\Lambda_L)}$.
\qed


\section*{Acknowledgments}
SN and OS acknowledge support by the German Research Foundation (DFG) via the
research unit FOR 3013 ``Vector- and tensor-valued surface PDEs'' (grant no.~NE2138/3-1).
SH, SN and MV acknowledge support by the German Research Foundation (DFG) – project number 405009441.

\printbibliography


\appendix
\section{Quadratic evolutionary rate-independent systems}
\label{sec:eris}

We recall the basic definition and existence result of solutions for evolutionary rate-independent systems with quadratic energy. For a detailed review of the theory of ERIS we refer to \cite{mielke2015rate}. In the paper we consider various quadratic ERIS, see Table~\ref{T:ERIS}.

\begin{definition}[quadratic ERIS, cf.~\mbox{\cite[Section 3.5]{mielke2015rate}}]\label{D:ERIS}
  We call $(Y,\mathcal E,\mathcal R)$ a \emph{quadratic evolutionary rate-independent system}, if
  \begin{enumerate}
  \item[(a)] the \emph{state space} $Y$ is a separable Hilbert space with dual space $Y^*$ and dual pairing $\expect{\cdot,\cdot}$,
  \item[(b)] the \emph{dissipation functional} $\rcal: Y\rightarrow [0,\infty]$ is convex, lower semicontinuous and positively homogeneous of degree $1$, i.e., $\rcal(\alpha \dot y)=\alpha \rcal(\dot y)$ for all $\alpha \ge 0$ and $\dot y\in Y$, and $\rcal(0)=0$,
  \item[(c)] the \emph{energy functional} $\E:[0,T] \times Y \to \R$ is of the form
    \begin{equation}
      \E(t,y) \colonequals \frac{1}{2}\expect{\mathcal Ay,y}- \expect{\ell(t),y},\label{eq:375:p}
    \end{equation}
    where $\mathcal A:Y\to Y^*$ is a symmetric, positive definite, bounded linear operator, and $\ell \in W^{1,1}((0,T);Y^*)$.
  \end{enumerate}
  The set 
  \begin{equation*}
    S(t)
    \colonequals
    \Big\{y\in Y:\; \E(t,y)\leq \E(t,\widetilde{y})+\rcal\brac{\widetilde{y}-y} \quad \text{for all }\widetilde{y}\in Y\Big\}.
  \end{equation*}
  defined for $t\in[0,T]$ is called the \emph{set of stable states}.
\end{definition}

\begin{theorem}[Energetic solution; see \mbox{\cite[Theorem 3.5.2]{mielke2015rate}}]\label{T:energetic_solution}
  Let $(Y,\mathcal E, \rcal)$ be a quadratic ERIS in the sense of Definition~\ref{D:ERIS}. Then for any initial condition $y^0\in S(0)$ there exists a unique $y\in W^{1,1}((0,T);Y)$ with $y(0)=y^0$ such that for all $t\in[0,T]$ the \emph{global energy balance equality},
  \begin{align}
    & \E(t,y(t))+\int_{0}^t \rcal(\dot{y}(s))ds = \E(0,y(0))- \int_{0}^{t}\expect{\dot{\ell}(s),y(s)} ds, \tag{E} \label{eq:413:p}
  \end{align}
  and the \emph{global stability condition}
  \begin{equation*}
    y(t)\in S(t)\qquad \tag{S}\label{eq:414:p}
  \end{equation*}
  holds. Moreover, we have the estimate
  \begin{equation}\label{eq:429:p}
    \|\dot y(t)\|_Y\leq \frac1c\|\dot\ell(t)\|_{Y^*}\qquad\text{for a.a.~}t\in[0,T],
  \end{equation}
  where  $c\colonequals \inf\big\{\expect{\mathcal Ay,y}\,:\,y\in Y,\,\|y\|_Y=1\big\}$ denotes the coercivity constant of $\mathcal A$.
\end{theorem}

\begin{lemma}[Equivalence to force balance equation; see \mbox{\cite{mielke2015rate}}]
  Let $(Y,\mathcal E, \rcal)$ be a quadratic ERIS in the sense of Definition~\ref{D:ERIS}, $y^0\in S(0)$, and let $y\in W^{1,1}((0,T);Y)$ satisfy $y(0)=y^0$. Then $y$ is an energetic solution, if and only if 
  for a.a.~$t\in[0,T]$ the force balance equation
  \begin{equation*}
    0 \in \partial \rcal(\dot{y}(t))+D_{y}\E(t,y(t)),
  \end{equation*}
  holds,  where $D_{y}\E(t,\cdot)$ denotes the Gateaux derivative of $\mathcal E(t,\cdot)$ and $\partial\rcal$ the subdifferential. 
\end{lemma}

\begin{table}\label{T:ERIS}
  \centering
  \renewcommand{\arraystretch}{1.5}
  \begin{tabular}{l|l|p{6cm}}
    Symbols & State variables &Description\\
    \hline
    \hline
    $(Y\e,\mathcal E\e,\rcal\e)$& $y\e=(u\e,p\e)$& mean model for random heterogeneous network on scale $\eps$ \\
    \hline
    $(Y\e^\omega,\mathcal E\e^\omega,\rcal\e^\omega)$& & quenched model for random heterogeneous network \\
    \hline
    $(Y_{\hom},\mathcal E_{\hom},\rcal_{\hom})$& $y_{\hom}=(u_{\hom},p_{\hom},\chi_{s,\hom})$& homogenized continuum model\\
    \hline
    $(Y_{L},\mathcal E_{L},\rcal_{L})$& $y_L=(u_{L},p_L,\varphi_L)$& periodic RVE approximation in the mean on scale $L$ stands for the size of the RVE\\
    \hline
    $(Y_{L}^\omega,\mathcal E_{L}^\omega,\rcal_{L}^\omega)$& & quenched, periodic RVE approximation on scale $L$\\
    \hline
    $(\widehat{Y}_{\hom},\widehat{\mathcal E}_{\hom},\widehat{\rcal}_{\hom})$& $\hat y_{\hom}=(\hat p_{\hom},\hat\chi_{s,\hom})$& rate-independent system describing the evolution of the internal variables of the homogenized, continuum model\\
    \hline
    $(\widehat{Y}_{L},\widehat{\mathcal E}_{L},\widehat{\rcal}_{L})$& $\hat y_L=(\hat p_L,\hat \varphi_L)$& rate-independent system describing the evolution of the internal variables of the RVE approximation in the mean\\
    \hline
    $(\widehat{Y}_{L}^\omega,\widehat{\mathcal E}_{L}^\omega,\widehat{\rcal}_{L}^\omega)$& & rate-independent system describing the evolution of the internal variables of the quenched, RVE approximation\\
  \end{tabular}
  \caption{The different evolutionary rate-independent systems (ERIS) considered in the paper}
\end{table}

\section{Tensor product spaces notation}\label{app:tensorproduct}
For Hilbert spaces $H_1,\ldots,H_n$ we denote by $H_1\otimes\cdots\otimes H_n$ the associated tensor product Hilbert space and by $\|\cdot\|_{H_1\otimes\cdots\otimes H_n}$ the associated norm. In the paper we use this notation mainly in situations where the Hilbert spaces are (closed subspaces of) $L^2$-spaces. In that case one can identify the tensor product Hilbert space with a (closed subspace) of the $L^2$-space associated with the product of the measure spaces. For instance, $L^2(\Omega)\otimes L^2(Q)$ (where $(\Omega,\mathcal F,\mathbb P)$ is a separable probability space and $Q\subset\R^d$ is considered with the Lebesgue $\sigma$-algebra) is isometric isomorph to $L^2(\Omega\times Q)$, as well as to the Bochner spaces $L^2(\Omega;L^2(Q))$ and $L^2(Q;L^2(\Omega))$. We therefore identify these spaces. We also remark that the norm of $\varphi\in L^2(\Omega)\otimes L^2(Q)$ takes the form
\begin{equation*}
  \|\varphi\|_{L^2(\Omega)\otimes L^2(Q)}=\mathbb E\Big[\int_Q|\varphi|^2\,dx\Big]^\frac12.
\end{equation*}
The above convention comes in hand when speaking about spaces of the form $L^2_s(\Omega)\otimes L^2(Q)$, since it would be cumbersome to define that space as a subspace of $L^2(\Omega\times Q)^k$.

\end{document}